\documentclass{article}

\usepackage{stylefiles/packages}
\usepackage{stylefiles/format}
\usepackage{stylefiles/abbrev}

\title{Goodness-of-fit testing for Hölder continuous densities under local differential privacy}

\author[,1]{Amandine Dubois \thanks{Financial support from GENES and from the French ANR grant ANR-18-EURE-0004}}
\author[2]{Thomas B. Berrett}
\author[3]{Cristina Butucea
\thanks{Financial support from GENES and the French National Research Agency (ANR) under the grant Labex Ecodec (ANR-11-LABEX-0047) }}
\affil[1]{CREST, ENSAI, Campus de Ker-Lann - Rue Blaise Pascal - BP 37203 - 35172 BRUZ cedex  amandine.dubois@ensai.fr}
\affil[2]{Department of Statistics, University of Warwick - Coventry - CV4 7AL - United Kingdom
tom.berrett@warwick.ac.uk}
\affil[3]{CREST, ENSAE, Institut Polytechnique de Paris, 5 avenue Henry Le Chatelier, F-91120 Palaiseau 
cristina.butucea@ensae.fr}

\date{}

\begin{document}

\maketitle

\begin{abstract}
We address the problem of goodness-of-fit testing for Hölder continuous densities under local differential privacy constraints. We study minimax separation rates when only non-interactive privacy mechanisms are allowed to be used and when both non-interactive and sequentially interactive can be used for privatisation.
We propose privacy mechanisms and associated testing procedures whose analysis enables us to obtain upper bounds on the minimax rates. These results are complemented with lower bounds.
By comparing these bounds, we show that the proposed privacy mechanisms and tests are optimal up to at most a logarithmic factor for several choices of $f_0$ including densities from uniform, normal, Beta, Cauchy, Pareto, exponential distributions. In particular, we observe that the results are deteriorated in the private setting compared to the non-private one. Moreover, we show that sequentially interactive mechanisms improve upon the results obtained when considering only non-interactive privacy mechanisms.
\end{abstract}

\section{Introduction}

Over the past few years, data privacy has become a fundamental problem in statistical data analysis. While more and more personal data are collected each day, stored and analyzed,  private data analysis aims at publishing valid statistical results without compromising the privacy of the individuals whose data are analysed. 
Differential privacy has emerged from this line of research as a strong mathematical framework which provides rigorous privacy guarantees.

Global differential privacy has been formalized by Dwork et al. \cite{Dwork_et_al_2006_calibrating_noise}. 
Their definition requires a curator who gathers the confidential data of $n$ individuals  and generates a privatized output from this complete information.
Only this privatized output can be released.
In a nutshell, the differential privacy constraints require that altering a single entry in the original dataset does not affect the probability of a privatized output too much.
One intuition behind this definition is that if the distribution of the privatized output does not depend too much on any single element of the database, then it should be difficult for an adversary to guess if one given person is in the database or not.
We refer the reader to \cite{Wasserman_Zhou_2010_statistical_framework_DP} for a precise definition of global differential privacy and more discussion on its testing interpretation.
In this paper, we will rather focus on the stronger notion of local differential privacy for which no trusted curator is needeed.
In the local setup, each individual generates a privatized version of its true data on its own machine, and only the privatized data are collected for analysis.
Thus, the data-owners do not have to share their true data with anyone else. However, some interaction between the $n$ individuals can be allowed.
We will consider two specific classes of locally differentially privacy mechanisms : non-interactive and sequentially interactive privacy mechanisms, respectively.
In the local non-interactive scenario, each individual generates a private view $Z_i$ of its original data $X_i$ on its own machine independently of all the other individuals.
In the sequentially interactive scenario, the privatized data $Z_1, \ldots , Z_n$ are generated such that the $i$-th individual has access to the previously privatized data $Z_1, \ldots , Z_{i-1}$ in addition to the original data $X_i$ in order to generate its own $Z_i$.

In this paper, we study a  goodness-of-fit testing problem for densities under local differential privacy constraints. Goodness-of-fit testing problems consist in testing whether  $n$ independent and identically distributed random variables $X_1,...,X_n$ were drawn from a  specified distribution $P_0$ or from any other distribution $P$ with $d(P_0,P)\geq \rho$ for some distance between distributions $d$ and some  \textit{separation parameter} $\rho>0$. Here, the considered distributions will be assumed to have Hölder smooth densities  and we will measure the separation between distributions using the $L_1$ norm which corresponds (up to a constant) to the total variation distance. Moreover, only privatised data $Z_1,...,Z_n$ are supposed available to be used in order to design testing procedures. Therefore we proceed in two steps: first randomize the original sample into a private sample, then build a  test using the latter sample. Optimality is shown over all test procedures and additionally over all privacy mechanisms satisfying the privacy constraints. We adopt a minimax point of view and aim at determining the private minimax testing radius which is the smallest separation parameter for which there exists a private testing procedure whose first type and second type error probabilities are bounded from above by a constant fixed in advance.

\subsection*{Contributions}
Our contributions can be summarized as follows.
First, when non-interactive privacy mechanisms are used, we present an $\alpha$-locally differentially private such mechanism and construct a testing procedure based on the privatized data. Its analysis indicates how to tune the parameters of the test statistic and the threshold of the test procedure in order to get a least upper bound on the non interactive testing radius. This result is further complemented with a lower bound.

Next, we prove that these bounds can be improved when allowing for sequential interaction. When previously privatized random variables are publicly available, we may proceed in two steps in order to improve on the detection rates. The first part of the sample is privatized as in the non-interactive case and it is used to acquire partial information on the unknown probability density. This information is further encoded in the private versions of the second part of the sample and the whole procedure benefits and attains faster rates of detection. This idea was previously introduced in  \cite{Butucea_Rohde_Steinberger_2020} and was also successful for testing discrete distributions in \cite{Berrett_Butucea_2020_DiscreteDistribTesting}. 

Finally, we investigate the optimality of our results for many choices of the null density $f_0$. We prove that our lower bounds and upper bounds match up to a constant in the sequentially interactive scenario, and up to a logarithmic factor in the non-interactive scenario, for several $f_0$ including densities from uniform, gaussian, beta, Cauchy, Pareto and exponential distributions.

\subsection*{Related work}
Goodness-of-fit testing for separation  norm $\|\cdot \|_1$ has recently received great attention in the non-private setting.  Valiant and Valiant~\cite{Valiant_Valiant_2017} studies the case of discrete distributions. Given a discrete distribution $P_0$ and an unknown discrete distribution $P$, they tackle the problem of finding how many samples from $P$ one should obtain to be able to distinguish with high probability the case that $P=P_0$ from the case that $\Vert P- P_0\Vert_1\geq \varepsilon$. They provide both upper bounds and lower bounds on this sample complexity as a function of $\varepsilon$ and the null hypothesis $P_0$.
Other testing procedures for this problem have been proposed in \cite{Diakonikolas_Kane_2016}, and   \cite{Balakrishnan_Wasserman_2019_hypothesisTesting} has revisited the problem in a minimax framework similar to the one considered in this paper (without privacy constraints).
Note that before these papers, the majority of the works on this problem focused on the case where $P_0$ is the uniform distribution, or considered a worst-case setting. The upper and lower bounds obtained in \cite{Valiant_Valiant_2017} and  \cite{Balakrishnan_Wasserman_2019_hypothesisTesting} appear to match in most usual cases but do not match for some pathological distributions.
This problem has been fixed in \cite{Chhor_Carpentier_2020_sharp}, where the authors provide matching upper and lower bounds on the minimax separation distance for separation norm $\|\cdot \|_t$, $t$ in $[1,2]$.
As for the continuous case,  \cite{Balakrishnan_Wasserman_2019_hypothesisTesting} studies goodness-of-fit testing for densities with separation norm $\Vert \cdot \Vert_1$, focusing on the case of Hölder continuous densities. As it has already been observed for the discrete case, they prove that the local minimax testing radius (or minimax separation distance) strongly depends on the null distribution. We extend their results to the private setting.

Many papers have been devoted to the study of testing problems under global differential privacy constraints. This includes goodness-of-fit testing \cite{Gaboardi_Lim_Rogers_Vadhan_2016_DPgoodnessOfFitAndIndependenceTesting,Acharya_Sun_Zhang_2018_DP_Testing_Identity_closeness,Aliakbarpour_Diakonikolas_Rubinfield_2018_DP_identity_Equiv_testing,Cai_Daskalakis_Kamath_2017_privIT,Wang_Lee_Kifer_2015_Hypothesis_test_categorical_data}, independence testing \cite{Gaboardi_Lim_Rogers_Vadhan_2016_DPgoodnessOfFitAndIndependenceTesting, Wang_Lee_Kifer_2015_Hypothesis_test_categorical_data} and closeness testing \cite{Acharya_Sun_Zhang_2018_DP_Testing_Identity_closeness,Aliakbarpour_Diakonikolas_Rubinfield_2018_DP_identity_Equiv_testing}.
In the local setting of differential privacy, \cite{Kairouz_Oh_Viswanath_2014_extremal_mechanisms,Kairouz_Oh_Viswanath_2016_extremal_mechanisms, Joseph_Mao_Neel_Roth_2019_role_interactivity} study simple hypothesis testing, and \cite{Gaboardi_Rogers_2018_local_private_Hypothesis_testing_chi2,Sheffet_2018_locally_Private_Hypothesis_Testing,Acharya_et_al_Test_without_trust_short} consider independence testing. 
Some of these references and a few others also deal with goodness-of-fit testing under local differential privacy constraints: 
\cite{Gaboardi_Rogers_2018_local_private_Hypothesis_testing_chi2} studies the asymptotic distribution of several test statistics used for fitting multinomial distributions, while  \cite{Sheffet_2018_locally_Private_Hypothesis_Testing} and \cite{Acharya_et_al_Test_without_trust_short} provide upper and lower bounds on the sample complexity for fitting more general but finitely supported discrete distributions. 
However, \cite{Acharya_et_al_Test_without_trust_short} considers only the case where the null distribution $P_0$ is the uniform distribution, and both papers prove lower bounds only with respect to the choice of the test statistic for a fixed specific privacy mechanism. In the minimax results below we prove optimality over all test statistics and also over all privacy mechanisms submitted to the local differential privacy constraints.

Minimax goodness-of-fit testing for discrete random variables has first been studied with $\mathbb{L}_2$ separation norm in  \cite{LamWeil_Laurent_Loubes_2020_densityTesting}. They consider the non-interactive scenario exclusively, and their lower bound result is proven for the uniform distribution $P_0$ under the null.
Lam-Weil {\it et al.} \cite{LamWeil_Laurent_Loubes_2020_densityTesting} also tackles the problem of goodness-of-fit testing for continuous random variables with $\Vert \cdot \Vert_2$ separation norm. They are the first to study minimax testing rates for the problem of goodness-of-fit testing for compactly supported densities over Besov balls $\mathcal{B}^s_{2,\infty}(L)$  in the setting of non-interactive local differential privacy. They provide an upper bound which holds for any density $f_0$, and a matching lower bound in the special case where $f_0$ is the uniform density over $[0,1]$.
In a parallel work, \cite{Butucea_Rohde_Steinberger_2020} investigates the estimation of the integrated square of a density over general Besov classes $\mathcal{B}^s_{p,q}$, and prove that allowing for sequential interaction improves over the results obtained in the non-interactive scenario in terms of minimax estimation rates.
As an application, they discuss non-interactive and sequentially interactive $L_2$-goodness-of-fit testing for densities supported on $[0,1]$ which lie in Besov balls. They thus extend the results obtained in \cite{LamWeil_Laurent_Loubes_2020_densityTesting} to more general Besov balls, to the interactive scenario, and to the case where $f_0$ is not assumed to be the uniform distribution, but has to be bounded from below on its support.

Later, locally differentially private goodness-of-fit testing for discrete random variables (not necessarily finite supported) has  been studied in \cite{Berrett_Butucea_2020_DiscreteDistribTesting} in a minimax framework.
The authors aim at computing the minimax testing rates when $d(P,P_0)=\sum_{j=1}^d\vert P(j)-P_0(j)\vert^{i}$, $i\in\{1,2\}$. They provide upper bounds on the minimax testing rates by constructing and analysing specific private testing procedures, complement these results with lower bounds, and investigate the optimality of their results for several choices of the null distribution $P_0$. 
Interestingly, they tackle both the sequentially interactive case and the non-interactive case and prove that the minimax testing rates are improved when sequential interaction is allowed.
Such a phenomenon appears neither for simple hypothesis testing \cite{Joseph_Mao_Neel_Roth_2019_role_interactivity}, nor for many estimation problems (see for instance \cite{Duchi_Joradn_Wainwright2018minimax,Berrett_Butucea_2019_Classification,Rohde_Steinberger_2020_geometrizing,Butucea_Dubois_Kroll_Saumard}).

We pursue these works  by considering goodness-of-fit testing of Hölder-smooth probability densities and the separation norm $\Vert \cdot \Vert_1$. Moreover, similarly to \cite{Balakrishnan_Wasserman_2019_hypothesisTesting}, we consider densities with H\"older smoothness $\beta$ in (0,1] and that can tend to 0 on their support,  with possibly unbounded support. Our goal is to show how differential privacy affects the minimax separation radius for this goodness-of-fit test. Balakrishnan and Wasserman\cite{Balakrishnan_Wasserman_2019_hypothesisTesting}, following works in discrete testing initiated by \cite{Valiant_Valiant_2017}, have shown that two procedures need to be aggregated in this case. They split the support of the density $f_0$ into a compact set $B $ where $f_0$ is bounded from below by some positive constant and they build a weighted $\mathbb{L}_2$ test on this set; then they build a tail test on $\overline{B}$ which is based on estimates of the total probabilities $(P-P_0)(\overline{B})$. They show that the separation rates are of order 
$$
\left( \frac{(\int_{B} f_0(x)^\gamma dx)^{1/\gamma} }{n} \right)^{\frac{2 \beta}{4 \beta +d}}, \quad \text{where } \gamma = \frac{2 \beta}{3\beta +d},
$$ 
for $d-$dimensional observations and depend of $f_0$ via an integral functional. The cut-off (choice of $B$) will depend on $n$ and their separation rates are not minimax optimal due to different cut-offs in the upper and lower bounds.

We show that under local differential privacy constraints, we get for an optimal choice of $B$ the separation rates
$$
\vert B \vert^{\frac{3\beta+3}{4\beta+3}}(n\alpha^2)^{-\frac{2\beta}{4\beta+3}}
$$
when only non-interactive privacy mechanisms are allowed, and we show that better rates are obtained
$$
\vert B \vert^{\frac{\beta+1}{2\beta+1}}(n\alpha^2)^{-\frac{2\beta}{4\beta+2}}
$$
when interactive privacy mechanims are allowed (using previously published privatized information). We see that our rates only depend on $f_0$ in a global way through the length $|B|$ of the set $B$ and that explains why we do not need to weight the $\mathbb{L}_2$ test statistic. Further work will include extension to more general H\"older and Besov classes with $\beta >0$ and adaptation to the smoothness $\beta$ by aggregation of an increasing number of tests as introduced by \cite{Spokoiny}.

\subsection*{Organization of the paper}

The paper is organized as follows. In Section \ref{Section Problem statement} we introduce the notion of local differential privacy and describe the minimax framework considered in the rest of the paper. In Section \ref{Section Non-interactive scenario} we introduce a non-interactive privacy mechanism and an associated testing procedure. Its analysis leads to an upper bound on the non-interactive testing radius which is complemented by a lower bound. In Section \ref{Section Intercative scenario} we give a lower bound on the testing radius for the sequentially interactive scenario and present a sequentially interactive testing procedure which improves on the rates of the non interactive case.  In Section \ref{Section Examples} we prove that our results are optimal (at most up to a logarithmic factor) for several choices of the null density $f_0$.

\section{Problem statement}\label{Section Problem statement}

Let $(X_1,...,X_n)\in \Xc^n$ be i.i.d. with common probability density function (pdf) $f:\Xc \to \Rb_+$. We assume that $f$ belongs to the smoothness class $H(\beta, L)$ for some smoothness $ 0<\beta \leq 1$ and $L>0$, where
$$ H(\beta, L)= \left\{f:\Xc\to \Rb_+ \; : \; \vert f(x)-f(y)\vert \leq L\vert x-y\vert^\beta, \quad \forall x,y\in\Xc \right\}. $$
In the sequel, we will omit the space $\Xc$ in the definition of functions $f$ and $f_0$ and integrals, and we will choose a set $B$ such that $B \subset \Xc$ and denote by $\overline{B}=\Xc\setminus B$.

Given a probability density function $f_0$ in $H(\beta,L_0)$ for some $L_0<L$, we want to solve the goodness-of-fit test
\begin{eqnarray*}
H_0 &:& f \equiv f_0\\
H_1(\rho) &:& f \in H(\beta,L) \text{ and } \|f-f_0\|_1 \geq \rho,
\end{eqnarray*}
where $\rho >0$ under an $\alpha$-local differential privacy constraint.
We will consider two classes of locally differentially private mechanisms : sequentially interactive mechnisms and non-interactive mechanisms.
In the sequentially interactive scenario, privatized data $Z_1,\ldots,Z_n$ are obtained by successively applying suitable Markov kernels : given $X_i=x_i$ and $Z_1=z_1,\ldots,Z_{i-1}=z_{i-1}$, the i-th data-holder draws
$$
Z_i\sim Q_i(\cdot\mid X_i= x,Z_1= z_1,\ldots,Z_{i-1}=z_{i-1})
$$
for some Markov kernel $Q_i: \Zs \times \Xc \times \Zc^{i-1} \to [0,1]$ where the measure spaces of the non-private and private data are denoted with $(\Xc,\Xs)$ and $(\Zc,\Zs)$, respectively.
We say that the sequence of Markov kernels $(Q_i)_{i=1,\ldots,n}$ provides $\alpha$-local differential privacy or that $Z_1,\ldots,Z_n$ are $\alpha$-local differentially private views of $X_1,\ldots,X_n$ if 
\begin{equation}\label{eq alphaLDPconstraint}
    \sup_{A \in \Zs}\sup_{z_1,\ldots,z_{i-1}\in\Zc}\sup_{x,x'\in\Xc} \frac{Q_i(A \mid X_i=x,Z_1=z_1,\ldots,Z_{i-1}=z_{i-1})}{Q_i(A \mid X_i=x^\prime,Z_1=z_1,\ldots,Z_{i-1}=z_{i-1})} \leq e^\alpha, \, \text{ for all } i=1,\ldots,n.
\end{equation}
We will denote by $\Qc_\alpha$ the set of all $\alpha$-LDP sequentially interactive mechanisms.
In the non-interactive scenario $Z_i$ depends only on $X_i$ but not on $Z_k$ for $k<i$.
We have 
$$
Z_i \sim Q_i(\cdot \mid X_i = x_i), 
$$
and condition \eqref{eq alphaLDPconstraint} becomes
$$
\sup_{A \in \Zs}\sup_{x,x'\in\Xc} \frac{Q_i(A \mid X_i=x)}{Q_i(A \mid X_i=x^\prime)} \leq e^\alpha, \, \text{ for all } i=1,\ldots,n.
$$
We will denote by $\Qc_\alpha^{\text{NI}}$ the set of all $\alpha$-LDP non-interactive mechanisms.
Given an $\alpha$-LDP privacy mechansim $Q$, let $\Phi_Q=\{ \phi : \Zc^n\rightarrow \{0,1\}\}$ denote the set of all tests based on $Z_1,\ldots Z_n$.

The sequentially interactive $\alpha$-LDP minimax testing risk is given by 
$$
\Rc_{n,\alpha}(f_0,\rho):= \inf_{Q\in \Qc_\alpha}\inf_{\phi\in \Phi_Q} \sup_{f\in H_1(\rho)}\left\{\Pb_{Q_{f_0}^n}(\phi=1)+ \Pb_{Q_{f}^n}(\phi=0) \right\}.
$$
We  define  similarly the non-interactive $\alpha$-LDP minimax testing risk $\Rc^{\text{NI}}_{n,\alpha}(f_0,\rho)$, where the first infimum is taken over the set $\Qc_\alpha^{\text{NI}}$ instead of $\Qc_\alpha$.
Given $\gamma\in(0,1)$, we study the $\alpha$-LDP minimax testing radius defined by 
$$
\Ec_{n,\alpha}(f_0,\gamma):=\inf \left\{\rho>0 :  \Rc_{n,\alpha}(f_0,\rho)\leq \gamma \right\},
$$
and we define similarly $\Ec^{\text{NI}}_{n,\alpha}(f_0,\gamma)$.

{\bf Notation}  For any positive integer number $n$, we denote by  $\llbr 1,n\rrbr$ the set of integer values $\{1,2,...,n\}$. If $B$ is a compact set on $\mathbb{R}$, we denote by $|B|$ its length (its Lebesgue measure). For any function $\psi$ and any positive real number $h$, we denote the rescaled function by $\psi_h = \frac 1h \psi \left(\frac {\cdot}h \right)$.
For two sequences $(a_n)_n$ and $(b_n)$, we denote by $a_n\lesssim b_n$ that there exists some constant $C>0$ such that $a_n\leq C b_n$, and we write $a_n \asymp b_n$ if both $a_n\lesssim b_n$ and $b_n\lesssim a_n$.

\section{Non-interactive Privacy Mechanisms}\label{Section Non-interactive scenario}

In this section we design a non-interactive $\alpha$-locally differentially private mechanism  and the associated testing procedure.
We study successively its first and second type error probabilities in order to obtain an upper bound on the testing radius $\Ec^{\text{NI}}_{n,\alpha}(f_0,\gamma)$.
We then present a lower bound on the testing radius.
The test and privacy mechanism proposed in this section will turn out to be (nearly) optimal for many choices of $f_0$ since the lower bound and the upper bound match up to a logarithmic factor for several $f_0$, see Section \ref{Section Examples} for many examples.

\subsection{Upper bound in the non-interactive scenario}\label{Section NI upper bound}

We propose a testing procedure that, like  \cite{Balakrishnan_Wasserman_2019_hypothesisTesting}, combines an $\mathbb{L}_2$ procedure on a bulk set $B$ where the density $f_0$ under the null is bounded away from 0 by some (small) constant and an $\mathbb{L}_1$ procedure on the tail $\overline B$. However, we note that, unlike \cite{Balakrishnan_Wasserman_2019_hypothesisTesting}, the rate depends on $f_0$ in a global way, only through the length $|B|$ of the set $B$. Our procedure also translates to the case of continuous distributions the one proposed by Berrett and Butucea~\cite{Berrett_Butucea_2020_DiscreteDistribTesting} for locally private testing of discrete distributions.
It consists in the following steps:
\begin{enumerate}
    \item Consider a compact set $B\subset\Rb$ (its choice depends on $f_0$, and on values of $n$ and $\alpha$).
    \item Using the first half of the (privatized) data, define an estimator $S_B$ of $\int_B(f-f_0)^2$.
    \item Using the second half of the (privatized) data, define an estimator $T_B$ of $\int_{\bar{B}}(f-f_0)$.
    \item Reject $H_0$ if either $S_B\geq t_1$ or $T_B\geq t_2$.
\end{enumerate}

Assume without loss of generality that the sample size is even and equal to $2n$ so that we can split the data into equal parts, $X_1,\ldots,X_n$ and $X_{n+1},\ldots,X_{2n}$.
Let  $B\subset \Rb$ be a nonempty  compact set, and let $(B_j)_{j=1,\ldots,N}$ be a partition of $B$, $h>0$ be the bandwidth and $(x_1,\ldots,x_N)$ be the centering points, that is $B_j=[x_j-h,x_j+h]$ for all $j\in \llbr 1,N\rrbr$.
Let $\psi:\Rb\rightarrow \Rb$ be a function satisfying the following assumptions.
\begin{assumption}\label{Assumption on the kernel}
$\psi$ is a bounded function supported in $[-1,1]$ such that 
$$
\int_{-1}^1\psi(t)\dd t=1,\quad \text{and} \quad \int_{-1}^1\vert t\vert^\beta \vert \psi(t)\vert \dd t <\infty.
$$
\end{assumption}
In particular, Assumption \ref{Assumption on the kernel} implies that $\psi_h(x_j-y)=0$ if $y\not\in B_j$, where $\psi_h(u)=\frac{1}{h}\psi(\frac{u}{h})$.\\
We now define our first privacy mechanism.
For $i\in \llbr 1,n\rrbr $ and $j\in\llbr 1, N\rrbr$ set
$$
Z_{ij}=\frac{1}{h}\psi\left(\frac{x_j-X_i}{h} \right)+\frac{2\Vert \psi\Vert_\infty}{\alpha h} W_{i j},
$$ 
where $(W_{ij})_{i\in\llbr 1,n\rrbr, j\in\llbr 1, N\rrbr}$ is a sequence of i.i.d Laplace($1$) random variables.
Using these privatized data, we define the following U-statistic of order $2$.
$$
S_B:=\sum_{j=1}^N\frac{1}{n(n-1)}\sum_{i\neq k}(Z_{ij}-f_0(x_j))(Z_{kj}-f_0(x_j)).
$$
The second half of the sample is used to design a tail test.
For all $i\in \llbr n+1, 2n\rrbr $ set
$$
Z_i = \pm c_\alpha, \text{ with probabilities }\frac 12 \left(1 \pm \frac{I(X_i \not \in B)}{c_\alpha} \right),
$$
where $c_\alpha = (e^\alpha +1)/(e^\alpha - 1)$. 
Using these private data, we define the following statistic.
$$
T_B =\frac{1}{n}\sum_{i=n+1}^{2n}Z_i - \int_{\overline B} f_0.
$$
We then put 
\begin{equation}\label{Non Interactive Test procedure}
\Phi=
\begin{cases}
1 & \text{ if } S_B\geq t_1 \text{ or }  T_B\geq t_2 \\
0 & \text{ otherwise }
\end{cases},
\end{equation}
where
\begin{equation}\label{Eq Non Interactive t_1 and t_2}
t_1= \frac{3}{2}L_0^2C_\beta^2Nh^{2\beta}+ \frac{196\Vert\psi\Vert_\infty^2\sqrt{N}}{\gamma n\alpha^2h^2}, \quad t_2=\sqrt{\frac{20}{n\alpha^2\gamma}},
\end{equation}
with $C_\beta=\int_{-1}^1\vert u\vert^\beta \vert \psi(u)\vert \dd u$.
The privacy mechanism that outputs $(Z_1,\ldots,Z_n,Z_{n+1},\ldots,Z_{2n})$ is non-interactive since for all $i\in\llbr 1,2n\rrbr$ $Z_i$ depends only on $X_i$.
The following result establishes that this mechanism also provides $\alpha$-local differential privacy.
Its proof is deferred to Section \ref{Section Proof Privacy of the NI mechanism} in the Appendix.

\begin{proposition}\label{Prop Privacy of the NI Mechanism}
For all $i\in\llbr 1, 2n\rrbr$, $Z_i$ is an $\alpha$-locally differentially private view of $X_i$.
\end{proposition}
The following proposition studies the properties of the test statistics. Its proof is given in the Appendix  \ref{App Proof Upper Bound NI}. 
\begin{proposition}\label{Prop Test Stat NI}
1. It holds
\begin{equation}\label{Eq Mean of the chi2stat NI}
\Eb_{Q_{f}^n}\left[S_B \right]=\sum_{j=1}^N\left([\psi_h\ast f](x_j)-f_0(x_j) \right)^2.
\end{equation}
Under Assumption \ref{Assumption on the kernel} it also holds if $\alpha\in (0,1]$
\begin{equation}\label{Eq Variance of the chi2stat NI}
\Var_{Q_{f}^n}\left(S_B \right)\leq \frac{36\Vert \psi\Vert_\infty^2}{n \alpha^2 h^2}\sum_{j=1}^N\left([\psi_h\ast f](x_j)-f_0(x_j) \right)^2+\frac{164\Vert \psi\Vert_\infty^4N}{n(n-1)\alpha^4h^4}.
\end{equation}

2. It holds
$$
 \Eb_{Q_{f}^n}[T_B] =\int_{\overline B} (f-f_0),\quad \text{and} \quad\Var_{Q_{f}^n}( T_B)=\frac{1}{n}\left(c_\alpha^2-\left(\int_{\overline B} f \right)^2  \right).
 $$
\end{proposition}

The study of the first and second type error probabilities of the test $\Phi$ in (\ref{Non Interactive Test procedure}) with a convenient choice of $h$ leads to the following upper bound on $\Ec_{n,\alpha}^\text{NI}(f_0,\gamma)$. 
\begin{theorem}\label{Thrm Upper Bound NI}
Assume that $\alpha\in(0,1)$ and $\beta\leq 1$.
The test procedure $\Phi$ in \eqref{Non Interactive Test procedure} with $t_1$ and $t_2$ in \eqref{Eq Non Interactive t_1 and t_2} and bandwidth $h$ given by $h\asymp \vert B \vert^{-1/(4\beta+3)}(n\alpha^2)^{-2/(4\beta+3)}$ attains the following bound on the separation rate
$$
\Ec_{n,\alpha}^\text{NI}(f_0,\gamma)\leq C(L,\gamma,\psi)\cdot \left\{\vert B \vert^{\frac{3\beta+3}{4\beta+3}}(n\alpha^2)^{-\frac{2\beta}{4\beta+3}} + \int_{\overline B} f_0 + \frac 1{\sqrt{n \alpha^2}}\right\},
$$
for all compact set $B\subset \Rb$.
\end{theorem}

The proof can be found in Appendix \ref{App Proof Upper Bound NI}.
Note that the tightest upper bound is obtained for the sets $B$ that minimize the right-hand sides in Theorem \ref{Thrm Upper Bound NI}.
In order to do this, we note that the upper bounds sum  a term which increases with $B$, a term which decreases with $B$:  $\int_{\overline{B}}f_0$ and a term $1/\sqrt{n \alpha^2}$ free of $B$. Thus we suggest to choose $B = B_{n,\alpha}$ as a level set
\begin{equation}\label{optB}
B_{n,\alpha} \in \arg \inf_{B \text{ compact set}}
\left\{|B|: \int_{\overline{B}}{ f_0} \geq \vert B \vert^{\frac{3\beta+3}{4\beta+3}}(nz_\alpha^2)^{-\frac{2\beta}{4\beta+3}} + \frac 1{\sqrt{n\alpha^2}} \text{ and } \inf_B f_0 \geq \sup_{\overline{B}} f_0 
\right\}. 
\end{equation}
\subsection{Lower bound in the non-interactive scenario}

We now complete the study of the testing radius $\Ec^{\text{NI}}_{n,\alpha}(f_0,\gamma)$ with the following lower bound.

\begin{theorem}\label{Thrm Lower bound NI}
Let $\alpha>0$. Assume that $\beta\leq 1$. Set $z_\alpha=e^{2\alpha}-e^{-2\alpha}$ and  $C_0(B)=\min\{f_0(x) : x\in B\}$. 
For all compact set $B\subset \Rb$ we get
$$
\Ec^{\text{NI}}_{n,\alpha}(f_0,\gamma)\geq C(\gamma,L,L_0)\left[\log\left(C \vert B\vert^{\frac{4\beta+4}{4\beta+3}}(nz_\alpha^2)^\frac{2}{4\beta+3}  \right)\right]^{-1}
\min \left\{ |B| C_0(B), 
\vert B \vert^{\frac{3\beta+3}{4\beta+3}}(nz_\alpha^2)^{-\frac{2\beta}{4\beta+3}}\right\}.
$$
If, moreover, the compact set $B$ is satisfying 
\begin{equation}\label{Eq. Condition on B for NI Lower bound}
\vert B \vert^{\beta/(4\beta+3)}C_0(B) \geq C(nz_\alpha^2)^{-2\beta/(4\beta+3)}
\end{equation}
for some $C>0$, it holds
$$
\Ec^{\text{NI}}_{n,\alpha}(f_0,\gamma)\geq C(\gamma,L,L_0)\left[\log\left(C \vert B\vert^{\frac{4\beta+4}{4\beta+3}}(nz_\alpha^2)^\frac{2}{4\beta+3}  \right)\right]^{-1}\vert B \vert^{\frac{3\beta+3}{4\beta+3}}(nz_\alpha^2)^{-\frac{2\beta}{4\beta+3}}.
$$
\end{theorem}



\textbf{Discussion of the optimality of the bounds.} The choice of the set $B$ is crucial for obtaining matching rates in the upper and lower bounds.\\
In the case where the support $\Xc$ of $f_0$ is compact with $c_1\leq \vert \Xc \vert\leq c_2$ for two constants $c_1>0$ and $c_2>0$ and if $f_0$ is bounded from below on $\Xc$, one can take $B=\Xc$. Indeed, for such functions, the choice $B=\Xc$ yields an upper bound of order $(n\alpha^2)^{-\frac{2\beta}{4\beta+3}}$. Moreover, \eqref{Eq. Condition on B for NI Lower bound} holds with this choice of $B$ and Theorem~\ref{Thrm Lower bound NI} proves that the upper bound is optimal up to (at most) a logarithmic factor.\\
In the case of densities with bounded support but which can tend to $0$ on their support, and in the case of densities with unbounded support, we suggest to choose $B=B_{n,\alpha}$ as defined in (\ref{optB}) both in the upper and lower bounds.

By inspection of the proof, we can also write that $B_{n,\alpha}$ in (\ref{optB}) is such that
$$
B_{n,\alpha} \in \arg \inf_{B \text{ compact set}}
\left\{|B|: \int_{\overline{B}}{ f_0} \geq \psi_{n,\alpha}(B)  \text{ and } \inf_B f_0 \geq \sup_{\overline{B}} f_0  \right\},
$$
where $\psi_{n,\alpha}(B)  = |B| h^\beta + \frac{\vert B\vert^{3/4}}{h^{3/4}\sqrt{n\alpha^2}}+\frac{1}{\sqrt{n\alpha^2}} = \vert B \vert^{\frac{3\beta+3}{4\beta+3}}(n\alpha^2)^{-\frac{2\beta}{4\beta+3}} + \frac 1{\sqrt{n\alpha^2}}$ for an optimal choice of $h=h^*(B) = (|B|^{1/2} n\alpha^2)^{-2/(4 \beta +3)}$.
Indeed,  we  choose $B_{n,\alpha}$ as a level set such that $\int_{\overline{B}} f_0$ (which is decreasing with $B$) be equal to $\psi_{n,\alpha}(B)$ (which is increasing with $B$). For the choices $B=B_{n,\alpha}$ and $h=h^*(B_{n,\alpha})$ we thus obtain an upper bound on $\Ec^{\text{NI}}_{n,\alpha}(f_0,\gamma)$ of order 
$$
\vert B_{n,\alpha} \vert^{\frac{3\beta+3}{4\beta+3}}(n\alpha^2)^{-\frac{2\beta}{4\beta+3}} + \frac 1{\sqrt{n \alpha^2}}.
$$
Recall that $f_0$ is a H\"older smooth function and thus uniformly bounded. {Moreover, $\psi_{n,\alpha}(B)$ and $\int_{\overline{B}} f_0$ are continuous quantities of the length of the set $B$ when it varies in the family of level sets.} Thus, for small rates $\psi_{n,\alpha}(B_{n,\alpha})$ we have necessarily $\int_{B_{n,\alpha}} f_0$ that does not tend to 0, hence $|B_{n,\alpha}|$ does not tend to 0.
Then the term $\vert B_{n,\alpha} \vert^{\frac{3\beta+3}{4\beta+3}}(n\alpha^2)^{-\frac{2\beta}{4\beta+3}}$ will be dominant.\\
The following proposition gives a sufficient condition so that our upper and lower bounds match up to a logarithmic factor.
\begin{proposition}
Let $B_{n,\alpha}$ be defined by \eqref{optB}. If there exists a compact set $K\subset \overline{B}_{n,\alpha}$ and some $c\in]0,1[$ such that 
    \begin{equation}\label{Eq. Condition on K}
    \int_K f_0\geq c\int_{\overline{B}_{n,\alpha}}f_0 \quad \text{and} \quad c\frac{\vert B_{n,\alpha}\vert}{\vert K\vert} \gtrsim 1,
    \end{equation}
    then it holds 
    $$\left[\log\left( \vert B_{n,\alpha}\vert^{\frac{4\beta+4}{4\beta+3}}(n\alpha^2)^\frac{2}{4\beta+3}  \right)\right]^{-1}\vert B_{n,\alpha} \vert^{\frac{3\beta+3}{4\beta+3}}(n\alpha^2)^{-\frac{2\beta}{4\beta+3}} \lesssim \Ec_{n,\alpha}^\text{NI}(f_0,\gamma) \lesssim \vert B_{n,\alpha} \vert^{\frac{3\beta+3}{4\beta+3}}(n\alpha^2)^{-\frac{2\beta}{4\beta+3}}.$$
\end{proposition}
\begin{proof} Indeed, if $K$ satisfies (\ref{Eq. Condition on K}),  then it holds
$$\int_K f_0\leq \vert K\vert \sup_{K} f_0\leq \vert K\vert \sup_{\overline{B}_{n,\alpha}} f_0 \leq \vert K\vert \inf_{B_{n,\alpha}} f_0,$$
and
$$\int_K f_0\geq c\int_{\overline{B}_{n,\alpha}}f_0 \geq c\psi_{n,\alpha}(B_{n,\alpha})\geq c\vert B_{n,\alpha} \vert \left(h^*(B_{n,\alpha}) \right)^\beta\gtrsim \vert K \vert \left(h^*(B_{n,\alpha}) \right)^\beta,$$
which yields $\inf_{B_{n,\alpha}} f_0\gtrsim  \left(h^*(B_{n,\alpha}) \right)^\beta$, and condition \eqref{Eq. Condition on B for NI Lower bound} is thus satisfied with $B=B_{n,\alpha}$. 
Thus, the choice $B=B_{n,\alpha}$ ends the proof of the proposition.
\end{proof}
Let us now discuss a sufficient condition for the existence of a compact set $K\subset \overline{B}_{n,\alpha}$ satisfying $\eqref{Eq. Condition on K}$.
Let us consider the special case of decreasing densities $f_0$ with support $\Xc=[0,+\infty)$. Note that for such functions, $B_{n,\alpha}$ takes the form $B_{n,\alpha}=[0,a]$. Writing $f_0(x)=\ell(x)/(1+x)$, a sufficient condition for the existence of a compact set $K\subset \overline{B}_{n,\alpha}$ satisfying $\eqref{Eq. Condition on K}$ is that $$\sup_{x\geq 1}\frac{\ell(tx)}{\ell(x)}\leq c $$
for some constant $c<1$ and some $t>1$. Indeed, in this case, taking $K=[a,ta]$, it holds $c\vert B_{n,\alpha}\vert /\vert K \vert =c/t$, and 
\begin{align*}
    \int_{ta}^\infty f_0&\leq c\int_{ta}^\infty  \frac{\ell(x/t)}{1+x}\dd x = ct\int_a^\infty \frac{\ell(u)}{1+tu}\dd u \leq c \left( \sup_{x\geq a}\frac{t(1+x)}{1+tx} \right)\int_a^\infty \frac{\ell(u)}{1+u}\dd u\\
    &\leq c\left(1+\frac{t-1}{1+ta} \right)\int_a^\infty f_0,
\end{align*}
and thus $$\frac{\int_K f_0}{\int_{\overline{B}_{n,\alpha}}f_0}=1-\frac{\int_{ta}^\infty f_0}{\int_a^\infty f_0}\geq 1-c\{1+o(1)\}, $$
and $\eqref{Eq. Condition on K}$ is satisfied if $a$ is  large enough. In this case our upper and lower bounds match up to a logarithmic factor.\\
{Note that $f_0$ in Example 5.2 checks the condition for all $t>1$ and the only example where this condition is not satisfied is Example \ref{Ex Slow varying function}. In the latter, the density  $f_0(x)= \frac{A \log(2)^A}{(x+2)\left(\log(x+2)\right)^{A+1}}$, $x\in[0,\infty)$, for some $A>0$ arbitrarily small but fixed, has very slowly decreasing tails. An additional logarithmic factor is lost in the lower bounds in this least favorable case.}
\\

\begin{proof}[Proof of Theorem \ref{Thrm Lower bound NI}]\label{Section Proof Lower Bound NI}
We use the well-known reduction technique. The idea is to build a family $\{f_\nu : \nu\in \Vc\}$ that belong to the alternative set of densities $H_1(\rho)$ and then reduce the test problem to testing between $f_0$  and the mixture of the $f_\nu$.
Our construction of such functions is inspired by the one proposed in \cite{LamWeil_Laurent_Loubes_2020_densityTesting} for goodness-of-fit testing over Besov Balls $\mathcal{B}^s_{2,\infty}$ in the special case where $f_0$ is the uniform distribution over $[0,1]$, and in \cite{Butucea_Rohde_Steinberger_2020} for the minimax estimation over Besov ellipsoids $\Bc^s_{p,q}$ of the integrated square of a density supported in $[0,1]$. However, we need to make some modifications in order to consider Hölder smoothness instead of Besov smoothness and to tackle the case of densities with unbounded support.
Let  $B\subset \Rb$ be a nonempty  compact set, and let $(B_j)_{j=1,\ldots,N}$ be a partition of $B$, $h>0$ be the bandwidth and $(x_1,\ldots,x_N)$ be the centering points, that is $B_j=[x_j-h,x_j+h]$ for all $j\in \llbr 1,N\rrbr$.
Let $\psi : [-1,1]\rightarrow \Rb$ be such that $\psi\in H(\beta,L)$, $\int \psi=0$ and $\int \psi^2=1$.
For $j\in\llbr 1,N\rrbr$, define
$$
\psi_j:t\in\Rb \mapsto \frac{1}{\sqrt{h}}\psi\left(\frac{t-x_j}{h} \right) .
$$
Note that the support of $\psi_j$ is $B_j$, $\int \psi_j=0$ and $(\psi_j)_{j=1,\ldots,N}$ is an orthonormal family.

Fix a privacy mechanism $Q=(Q_1,\ldots, Q_n)\in \Qc_\alpha^{\text{NI}}$.
According to lemma B.3 in \cite{Butucea_Rohde_Steinberger_2020}, we can consider for every $i\in\llbr 1,n\rrbr$  a probability measure $\mu_i$ on $\Zc_i$ and a family of $\mu_i$-densities $(q_i(\cdot \mid x))_{x\in\Rb}$ such that for every $x\in\Rb$ one has $dQ_i(\cdot \mid x)=q_i(\cdot\mid x) d\mu_i$ and $e^{-\alpha }\leq q_i(\cdot\mid x)\leq e^\alpha$.
Denote by $g_{0,i}(z_i)=\int_\Rb q_i(z_i\mid x)f_0(x)\dd x$ the density of $Z_i$ when $X_i$ has density $f_0$.
Define for all $i=1,\ldots,n$ the operator $K_i: L_2(\Rb)\rightarrow L_2(\Zc_i, d\mu_i)$ by
$$
K_i f=\int_\Rb \frac{q_i(\cdot\mid x)f(x)\1_B(x)}{\sqrt{g_{0,i}(\cdot)}}\dd x, \quad f\in L_2(\Rb).
$$
Note that this operator is well-defined since $g_{0,i}(z_i)\geq \int_\Rb e^{-\alpha} f_0(x)\dd x=e^{-\alpha}>0$ for all $z_i$.
Observe that its adjoint operator $K_i^\star$ is given by
$$
K_i^\star : \ell\in  L_2(\Zc_i, d\mu_i)\mapsto \int_{\Zc_i}\frac{\ell(z_i)q_i(z_i\mid \cdot)\1_B(\cdot)}{\sqrt{g_{0,i}(z_i)}}\dd \mu_i(z_i).
$$
Using Fubini's theorem we thus have for all $f\in L_2(\Rb)$
\begin{align*}
 K_i^\star K_i f&=\int_{\Zc_i}\left( \int_\Rb \frac{q_i(z_i\mid y)f(y)\1_B(y)}{\sqrt{g_{0,i}(z_i)}}\dd y \right)\frac{q_i(z_i\mid \cdot)\1_B(\cdot)}{\sqrt{g_{0,i}(z_i)}}\dd \mu_i(z_i)\\
 &=\int_\Rb\left(\int_{\Zc_i} \frac{q_i(z_i\mid y)q_i(z_i\mid \cdot)\1_B(y)\1_B(\cdot)}{g_{0,i}(z_i)}\dd \mu_i(z_i) \right)f(y)\dd y,   
\end{align*}
meaning that $K_i^\star K_i$ is an integral operator with kernel $F_i(x,y)=\int_{\Zc_i} \frac{q_i(z_i\mid x)q_i (z_i\mid y)\1_B(x)\1_B(y)}{g_{0,i}(z_i)}\dd\mu_i(z_i)$.
Define the operator 
$$
K=\frac{1}{n}\sum_{i=1}^n K_i^\star K_i,
$$
which is symmetric and positive semidefinite.
Define also 
$$
W_N=\text{span}\{ \psi_j, j=1,\ldots, N\}.
$$
Let $(v_1,\ldots,v_N)$  be an orthonormal family of eigenfunctions of $K$ as an operator on the linear $L_2(\Rb)$-subspace $W_N$. 
Note that since $v_k$ can be written as a linear combination of the $\psi_j$'s, it holds $\int_{\Rb}v_k=0$ and $\Supp(v_k)\subset B$.
We also denote by $\lambda_1^2,\ldots,\lambda_N^2$ the corresponding eigenvalues. Note that they are non-negative.

\vspace{0.3cm}
Define the functions 
$$f_\nu:x\in\Rb\mapsto f_0(x)+\delta\sum_{j=1}^N\frac{\nu_j}{\tilde{\lambda}_j}v_j(x),$$
where for $j=1,\ldots,N$ $\nu_j\in\{-1,1\}$, $\delta>0$ may depend on $B$,$h$, $N$, $\psi$, $\gamma$, $L$, $L_0$, $\beta$, $n$ and $\alpha$,  and will be specified later, and 
$$
\tilde{\lambda_j}=\max\left\{ \frac{\lambda_j}{z_\alpha}, \sqrt{2h}\right\}, \quad z_\alpha=e^{2\alpha}-e^{-2\alpha}.
$$
The following lemma shows that for $\delta$ properly chosen, for most of the possible $\nu\in\{-1,1\}^N$, $f_\nu$ is a density belonging to $H(\beta,L)$ and $f_\nu$ is sufficiently far away from $f_0$ in a $L_1$ sense.

\begin{lemma}\label{LemmaBSNONI}
Let $\Pb_\nu$ denote the uniform distribution on $\{-1,1\}^N$. 
Let $b>0$.
If the parameter $\delta$  appearing in the definition of $f_\nu$ satisfies
$$
\delta\leq \frac{h}{\sqrt{\log(2N/b)}}\min\left\{ \frac{C_0(B)}{\Vert\psi\Vert_\infty} , \frac{1}{2}\left(1-\frac{L_0}{L} \right)h^\beta\right\},
$$
where $C_0(B):=\min\{f_0(x) : x\in B \}$,  then there exists  a subset $A_b\subseteq \{-1,1\}^N$ with $\Pb_\nu(A_b)\geq 1-b$ such that
\begin{enumerate}[label=\roman*)]
\item $f_\nu\geq 0$ and $\int f_\nu=1$, for all $\nu\in A_b$,
\item $f_\nu\in H(\beta,L)$, for all $\nu\in A_b$,
\item $\Vert f_\nu-f_0\Vert_1\geq \frac{3C_1}{8}\frac{\delta  N }{\sqrt{\log\left( \frac{2N}{b} \right)}}$, for all $\nu\in A_b$, with $C_1=\int_{-1}^1\vert \psi\vert$.
\end{enumerate}
\end{lemma}

Denote by $g_{\nu,i}(z_i)=\int_\Rb q_i(z_i\mid x)f_\nu(x)\dd x$ the density of $Z_i$ when $X_i$ has density $f_\nu$, and 
$$
d Q_n(z_1,\ldots,z_n)=\Eb_\nu\left[ \prod_{i=1}^n g_{\nu,i}(z_i)\dd \mu_i(z_i) \right].
$$
If $\delta$ is chosen such that $\delta\leq \frac{ h}{\sqrt{\log(2N/b)}}\min\left\{ \frac{C_0(B)}{\Vert\psi\Vert_\infty} , \frac{1}{2}\left(1-\frac{L_0}{L} \right)h^\beta\right\}$, setting  
\begin{equation*}
\rho^\star =\frac{3C_1}{8}\frac{\delta  N }{\sqrt{\log\left( \frac{2N}{b} \right)}},
\end{equation*}
we deduce from the above lemma that if
\begin{equation}\label{Proof LB NI Cond1}
\Eb_{Q_{f_0}^n}\left[\left( \frac{d Q_n}{d Q_{f_0}^n} \right)^2  \right]\leq 1+(1-\gamma-b)^2 \text{ for all } Q\in \Qc_\alpha^{\text{NI}},
\end{equation}
 then it holds
$$
 \inf_{Q\in \Qc_\alpha^{\text{NI}}}\inf_{\phi\in \Phi_Q} \sup_{f\in H_1(\rho^\star)}\left\{\Pb_{Q_{f_0}^n}(\phi=1)+ \Pb_{Q_{f}^n}(\phi=0) \right\}\geq \gamma,
$$
where $H_1(\rho^\star):= \{ f \in H(\beta, L) : f\geq 0, \int f=1, \Vert f-f_0\Vert_1 \geq \rho^\star\}$, and consequently $\Ec^{\text{NI}}_{n,\alpha}(f_0,\gamma)\geq \rho^\star$.
Indeed, if \eqref{Proof LB NI Cond1} holds, then we have
\begin{align*}
&\inf_{Q\in \Qc_\alpha^{\text{NI}}}\inf_{\phi\in \Phi_Q} \sup_{f\in H_1(\rho^\star)}\left\{\Pb_{Q_{f_0}^n}(\phi=1)+ \Pb_{Q_{f}^n}(\phi=0) \right\} \\
&\geq \inf_{Q\in \Qc_\alpha^{\text{NI}}}\inf_{\phi\in \Phi_Q} \left( \Pb_{Q_{f_0}^n}(\phi=1)+ \sup_{\nu\in A_b}\Pb_{Q_{f_\nu}^n}(\phi=0) \right)\\
&\geq \inf_{Q\in \Qc_\alpha^{\text{NI}}}\inf_{\phi\in \Phi_Q} \left( \Pb_{Q_{f_0}^n}(\phi=1)+ \Eb_\nu\left[I(\nu\in A_b)\Pb_{Q_{f_\nu}^n}(\phi=0)\right] \right),
\end{align*}
and 
\begin{align*}
\Eb_\nu\left[I(\nu\in A_b)\Pb_{Q_{f_\nu}^n}(\phi=0)\right]&= \Pb_{Q_n}(\phi=0)-\Eb_\nu\left[I(\nu\in A_b^c)\Pb_{Q_{f_\nu}^n}(\phi=0)\right]\\
&\geq \Pb_{Q_n}(\phi=0)-\Pb_\nu(A_b^c)\\
&\geq \Pb_{Q_n}(\phi=0)-b.
\end{align*}
Thus, if \eqref{Proof LB NI Cond1}  holds, we have
\begin{align*}
&\inf_{Q\in \Qc_\alpha^{\text{NI}}}\inf_{\phi\in \Phi_Q} \sup_{f\in H_1(\rho^\star)}\left\{\Pb_{Q_{f_0}^n}(\phi=1)+ \Pb_{Q_{f}^n}(\phi=0) \right\}\\
&\geq \inf_{Q\in \Qc_\alpha^{\text{NI}}}\inf_{\phi\in \Phi_Q} \left( \Pb_{Q_{f_0}^n}(\phi=1)+ \Pb_{Q_n}(\phi=0)-b
 \right)\\
 &\geq \inf_{Q\in \Qc_\alpha^{\text{NI}}} \left(1-\text{TV}(Q_n,Q_{f_0}^n)-b
 \right)\\
 &= \inf_{Q\in \Qc_\alpha^{\text{NI}}} \left(1-b-\sqrt{\Eb_{Q_{f_0}^n}\left[\left( \frac{d Q_n}{d Q_{f_0}^n} \right)^2  \right]-1}
 \right) \geq \gamma.
\end{align*}
We now prove that \eqref{Proof LB NI Cond1} holds under an extra assumption on $\delta$.

We have that 
\begin{align*}
&\Eb_{Q_{f_0}^n}\left[ \left(\frac{d Q_n}{d Q_{f_0}^n} \right)^2 \right] =\Eb_{Q_{f_0}^n}\left[\left( \frac{\Eb_\nu\left[ \prod_{i=1}^n g_{\nu,i}(Z_i) \right]}{\prod_{i=1}^n g_{0,i}(Z_i)} \right)^2 \right] \\
&= \Eb_{Q_{f_0}^n}\left[ \Eb_{\nu,\nu'}\prod_{i=1}^n \left(1+\delta\sum_{k=1}^N\frac{\nu_k}{\tilde{\lambda}_k}\cdot \frac{\langle q_i(Z_i\mid \cdot),v_k\rangle}{g_{0,i}(Z_i)}  \right)\cdot\left(1+\delta\sum_{k=1}^N\frac{\nu'_k}{\tilde{\lambda}_k}\cdot \frac{\langle q_i(Z_i\mid \cdot),v_k\rangle}{g_{0,i}(Z_i)}  \right) \right]\\
&=\Eb_{\nu,\nu'}\prod_{i=1}^n\left(1+\delta\sum_{k=1}^N\frac{\nu_k}{\tilde{\lambda}_k}\Eb_{Q_{f_0}}\left[ \frac{\langle q_i(Z_i\mid \cdot),v_k\rangle}{g_{0,i}(Z_i)}\right]+\delta\sum_{k=1}^N\frac{\nu'_k}{\tilde{\lambda}_k}\Eb_{Q_{f_0}}\left[ \frac{\langle q_i(Z_i\mid \cdot),v_k\rangle}{g_{0,i}(Z_i)}\right] \right.\\
&\left. \hspace{2cm}+\delta^2 \sum_{k_1,k_2=1}^N\frac{\nu_{k_1}\nu'_{k_2}}{\tilde{\lambda}_{k_1}\tilde{\lambda}_{k_2}}\Eb_{Q_{f_0}}\left[ \frac{\langle q_i(Z_i\mid \cdot),v_{k_1}\rangle\langle q_i(Z_i\mid \cdot),v_{k_2}\rangle}{(g_{0,i}(Z_i))^2}\right]\right),
\end{align*}
where we have interverted $\Eb_{Q_{f_0}^n}$ and $\Eb_{\nu,\nu'}$ and used the independence of the $Z_i$, $i=1,\ldots,n$.
Now, observe that
\begin{align*}
\Eb_{Q_{f_0}}\left[ \frac{\langle q_i(Z_i\mid \cdot),v_k\rangle}{g_{0,i}(Z_i)}\right]&=\int_{\Zc_i} \frac{\langle q_i(z_i\mid \cdot),v_k\rangle}{g_{0,i}(z_i)}\cdot g_{0,i}(z_i)\dd\mu_i(z_i)\\
&=\int_{\Zc_i} \left(\int_\Rb q_i(z_i\mid x )v_k(x)\dd x \right)\dd\mu_i(z_i)\\
&=\int_\Rb v_k=0,
\end{align*}
and, using that $\Supp(v_k)\subset B$ for all $k$,
\begin{align*}
&\Eb_{Q_{f_0}}\left[ \frac{\langle q_i(Z_i\mid \cdot),v_{k_1}\rangle\langle q_i(Z_i\mid \cdot),v_{k_2}\rangle}{(g_{0,i}(Z_i))^2}\right]\\
&=\int_{\Zc_i}\frac{\langle q_i(z_i\mid \cdot),v_{k_1}\rangle\langle q_i(z_i\mid \cdot),v_{k_2}\rangle}{(g_{0,i}(z_i))^2}\cdot g_{0,i}(z_i)\dd\mu_i(z_i)\\
&=\int_{\Zc_i}\frac{1}{g_{0,i}(z_i)}\left( \int_\Rb q_i(z_i\mid x )v_{k_1}(x)\dd x \right)\left( \int_\Rb q_i(z_i\mid y )v_{k_2}(y)\dd y \right)\dd\mu_i(z_i)\\
&=\int_{\Rb}\int_{\Rb}\left( \int_{\Zc_i}\frac{q_i(z_i\mid x )q_i(z_i\mid y )\1_B(x)\1_B(y)}{g_{0,i}(z_i)} \dd\mu_i(z_i)\right)v_{k_1}(x)v_{k_2}(y)\dd x\dd y\\
&=\int_{\Rb}\int_{\Rb}F_i(x,y)v_{k_1}(x)v_{k_2}(y)\dd x\dd y
=\langle v_{k_1},K_i^\star K_iv_{k_2}\rangle.
\end{align*}

Using $1+x\leq \exp(x)$, we thus obtain 
\begin{align*}
\Eb_{Q_{f_0}^n}\left[ \left(\frac{dQ_n}{d Q_{f_0}^n} \right)^2 \right]&=\Eb_{\nu,\nu'}\prod_{i=1}^n\left(1+\delta^2 \sum_{k_1,k_2=1}^N\frac{\nu_{k_1}\nu'_{k_2}}{\tilde{\lambda}_{k_1}\tilde{\lambda}_{k_2}} \langle v_{k_1},K_i^\star K_iv_{k_2}\rangle \right)\\
&\leq \Eb_{\nu,\nu'}\left[\exp\left(\delta^2\sum_{i=1}^n \sum_{k_1,k_2=1}^N\frac{\nu_{k_1}\nu'_{k_2}}{\tilde{\lambda}_{k_1}\tilde{\lambda}_{k_2}} \langle v_{k_1},K_i^\star K_iv_{k_2}\rangle \right) \right]\\
&=\Eb_{\nu,\nu'}\left[\exp\left(n \delta^2 \sum_{k_1,k_2=1}^N\frac{\nu_{k_1}\nu'_{k_2}}{\tilde{\lambda}_{k_1}\tilde{\lambda}_{k_2}} \langle v_{k_1},Kv_{k_2}\rangle \right) \right]\\
&=\Eb_{\nu,\nu'}\left[\exp\left( n \delta^2 \sum_{k_1,k_2=1}^N\frac{\nu_{k_1}\nu'_{k_2}}{\tilde{\lambda}_{k_1}\tilde{\lambda}_{k_2}}\cdot \lambda_{k_2}^2 \langle v_{k_1},v_{k_2}\rangle \right) \right]\\
&\leq \Eb_{\nu,\nu'}\left[\exp\left(n \delta^2z_\alpha^2 \sum_{k=1}^N\nu_k\nu'_k  \right) \right],
\end{align*}
where we have used that
$$
\frac{\lambda_k^2}{\tilde{\lambda}_k^2}=\frac{\lambda_k^2}{\max\{z_\alpha^{-2}\lambda_k^2, 2h \}}\leq z_\alpha^2.
$$
Now, using that for $k=1,\ldots,N$, $\nu_k$, $\nu_k'$ are Rademacher distributed and independent random variables, we obtain
\begin{align*}
&\Eb_{Q_{f_0}^n}\left[ \left(\frac{dQ_n}{d Q_{f_0}^n} \right)^2 \right] \leq  \Eb_{\nu,\nu'}\left[\prod_{k=1}^N\exp\left(n \delta^2z_\alpha^2 \nu_k\nu'_k   \right) \right]\\
&=\Eb_\nu\left[\prod_{k=1}^N\cosh\left(n \delta^2z_\alpha^2 \nu_k  \right)\right] 
=\prod_{k=1}^N\cosh\left(n \delta^2z_\alpha^2 \right)
\leq \exp\left(\frac{Nn^2\delta^4 z_\alpha^4}{2} \right),
\end{align*}
where the last inequality follows from $\cosh(x)\leq \exp(x^2/2)$ for all $x\in\Rb$.
Thus, \eqref{Proof LB NI Cond1} holds as soon as 
$$
\delta\leq \left[\frac{2\log\left(1+(1-b-\gamma)^2\right)}{Nn^2z_\alpha^4}\right]^{1/4}.
$$
Finally, taking $\delta=\min\left\{\frac{h}{\sqrt{\log(2N/b)}}\min\left\{ \frac{C_0(B)}{\Vert\psi\Vert_\infty} , \frac{1}{2}\left(1-\frac{L_0}{L} \right)h^\beta\right\}, \left[\frac{2\log\left(1+(1-b-\gamma)^2\right)}{Nn^2z_\alpha^4}\right]^{1/4}  \right\}$, we obtain
$$
\Ec^{\text{NI}}_{n,\alpha}(f_0,\gamma)\geq C(\psi,b,\gamma)\frac{ 1}{\sqrt{\log\left( 2N/b\right)}}\min\left\{\frac{\vert B\vert}{\sqrt{\log(2N/b)}}\min\left\{ \frac{C_0(B)}{\Vert\psi\Vert_\infty} , \frac{1}{2}\left(1-\frac{L_0}{L} \right)h^\beta\right\}, \frac{N^{3/4}}{\sqrt{n z_\alpha^2}}\right\}.
$$
If $B$ is chosen such that $C_0(B)=\min\{ f_0(x), x\in B\}\geq Ch^\beta$, then the bound becomes
$$
\Ec^{\text{NI}}_{n,\alpha}(f_0,\gamma)\geq C(\psi,b,\gamma,L,L_0)\frac{ 1}{\sqrt{\log\left( 2N/b \right)}}\min\left\{\frac{\vert B\vert h^\beta}{\sqrt{\log(2N/b)}}, \frac{N^{3/4}}{\sqrt{n z_\alpha^2}}\right\},
$$
and the choice $h\asymp \vert B \vert^{-1/(4\beta+3)}(nz_\alpha^2)^{-2/(4\beta+3)}$ yields
$$
\Ec^{\text{NI}}_{n,\alpha}(f_0,\gamma)\geq C(\psi,b,\gamma,L,L_0)\left[\log\left(C \vert B\vert^{\frac{4\beta+4}{4\beta+3}}(nz_\alpha^2)^\frac{2}{4\beta+3}  \right)\right]^{-1}\vert B \vert^{\frac{3\beta+3}{4\beta+3}}(nz_\alpha^2)^{-\frac{2\beta}{4\beta+3}}.
$$
Note that with this choice of $h$, the condition $C_0(B)\geq Ch^\beta$ becomes 
$$\vert B \vert^{\beta/(4\beta+3)}C_0(B) \geq C(nz_\alpha^2)^{-2\beta/(4\beta+3)}.
$$
\end{proof}


\section{Interactive Privacy Mechanisms}\label{Section Intercative scenario}

In this section, we prove that the results obtained in Section \ref{Section Non-interactive scenario} can be improved when sequential interaction is allowed between data-holders.

\subsection{Upper bound in the interactive scenario}

We first propose a testing procedure which relies on some sequential interaction between data-holders. We then prove that this test achieves a better separation rate than the one obtained in Section \ref{Section Non-interactive scenario}.

We assume that the sample size is equal to $3n$ so that we can split the data in three parts.
Like in the non-interactive scenario, we consider a non-empty compact set $B\subset\Rb$, and $B=\bigcup_{j=1}^NB_j$  a partition of $B$ with $\vert B_j\vert=2h$ for all $j\in\llbr 1,N\rrbr$.

With the first third of the data, $X_1,\ldots,X_n$, we generate privatized arrays $Z_i=(Z_{ij})_{j=1,\ldots,N}$ that will be used to estimate $p(j):=\int_{B_j}f$.
Let's consider the following privacy mechanism. We first generate an i.i.d. sequence $(W_{ij})_{i\in\llbr 1,n\rrbr, j\in\llbr 1,N\rrbr}$ of Laplace(1) random variables and for $i=1,\ldots,n$ and $j=1,\ldots,N$ we set
$$
Z_{ij}=I(X_i\in B_j)+\frac{2}{\alpha}W_{ij}.
$$ 
For each $j=1,\ldots, N$, we then build an estimator of $p(j):=\int_{B_j}f$ via
$$
\widehat{p}_j=\frac{1}{n}\sum_{i=1}^nZ_{ij}.
$$
We now privatize the second third of the data.
Set $c_\alpha=\frac{e^\alpha+1}{e^\alpha-1}$ and $\tau=(n\alpha^2)^{-1/2}$.
For all $i\in\llbr n+1,2n\rrbr$, we generate $Z_i\in\{-c_\alpha\tau,c_\alpha\tau\}$ using the estimator $\widehat{p}_j$ and the true data $X_i$ by
$$
\Pb\left(Z_i= \pm c_\alpha\tau \mid X_i\in B_j \right)=\frac{1}{2}\left( 1 \pm \frac{[\widehat{p}_j-p_0(j)]_{-\tau}^\tau}{c_\alpha\tau} \right),
$$
$$
\Pb\left(Z_i= \pm c_\alpha\tau \mid X_i\in \bar{B} \right)=\frac{1}{2},
$$
where $[x]_{-\tau}^\tau=\max\{-\tau,\min(x,\tau)\}$,  and $p_0(j)=\int_{B_j}f_0$.
We then define the test statistic
$$
D_B=\frac{1}{n}\sum_{i=n+1}^{2n}Z_i-\sum_{j=1}^N p_0(j)[\widehat{p}_j-p_0(j)]_{-\tau}^\tau.
$$
The analysis of the mean and variance of this statistic can be found in Appendix \ref{App Analysis D_B}.
It will be crucial in the analysis of our final test procedure.\\
Finally, we define the same tail test statistic as in Section \ref{Section Non-interactive scenario}. 
For all $i\in\llbr 2n+1,3n \rrbr $, a private view $Z_i$ of $X_i$ is generated by 
$$
Z_i = \pm c_\alpha, \text{ with probabilities }\frac 12 \left(1 \pm \frac{I(X_i \not \in B)}{c_\alpha} \right),
$$
and we set
$$
T_B =\frac{1}{n}\sum_{i=2n+1}^{3n}Z_i - \int_{\overline B} f_0.
$$
The final test is 
\begin{equation}\label{Eq Test in the interactive scenario}
\Phi=
\begin{cases}
1 & \text{ if } D_B\geq t_1 \text{ or }  T_B\geq t_2 \\
0 & \text{ otherwise }
\end{cases},
\end{equation}
where
\begin{equation}\label{Eq interactive t1 and t2}
t_1= \frac{2\sqrt{5}}{n\alpha^2\sqrt{\gamma}}, \quad  \quad t_2= \sqrt{\frac{20}{n\alpha^2\gamma}}.    
\end{equation}
We denote the privacy mechanism that outputs $(Z_1,\ldots,Z_n,Z_{n+1},\ldots,Z_{2n},Z_{2n+1},\ldots, Z_{3n})$ by $Q$.
It is sequentially interactive since each $Z_i$ for $i\in\llbr n+1,2n\rrbr $ depends on the privatized data $(Z_1,\ldots,Z_n)$ through $\widehat{p}_j$, but does not depend on the other $Z_k$, $k\in\llbr n+1,2n\rrbr$, $k\neq i$. 
The following result establishes that this mechanism provides $\alpha$-local differential privacy.
Its proof is deferred to Appendix \ref{App Proof privacy interactive mechanism}.

\begin{proposition}\label{Prop Privacy Interactive mechanism}
The sequentially interactive privacy mechanism $Q$ provides $\alpha$-local differential privacy. 
\end{proposition}
The following Proposition gives properties of the test statistic $D_B$. Its proof is in the Appendix~\ref{App Analysis D_B}.
\begin{proposition}\label{Prop Test Stat Int}
1. It holds $\Eb_{Qf^{n}}[D_B]= \sum_{j=1}^N\{p(j)- p_0(j)\}\Eb\left[[\widehat{p}_j-p_0(j)]_{-\tau}^\tau\right].$
In particular, $\Eb_{Qf_0^n}[D_B]= 0$. 
Moreover, we have 
\begin{equation}\label{Eq Lower bound on the Mean of the Interactive test statistic}
\Eb_{Qf^n}[D_B]\geq \frac{1}{6}D_\tau(f)-6\frac{\tau}{\sqrt{n}},
\end{equation}
with $D_\tau(f)= \sum_{j=1}^N\vert p(j)-p_0(j)\vert \min\left\{ \vert p(j)-p_0(j)\vert, \tau  \right\} $ where we recall that $p(j):=\int_{B_j}f$.

2. It holds
$$
\Var_{Qf^n}(D_B)\leq \frac{5}{(n\alpha^2)^2}+67\frac{D_\tau(f)}{n\alpha^2}.
$$
\end{proposition}

The following result presents an upper bound on $\Ec_{n,\alpha}(f_0,\gamma)$.
Its proof is in Appendix \ref{App Proof upper bound interactive scenario}.

\begin{theorem}\label{Thrm Interactive Upper bound}
Assume that $\alpha\in(0,1)$ and $\beta<1$.
The test procedure $\Phi$ in \eqref{Eq Test in the interactive scenario} with $t_1$ and $t_2$ in \eqref{Eq interactive t1 and t2} and bandwidth $h$ given by
$$
h\asymp \vert B\vert ^{-\frac{1}{2\beta+1}}(n\alpha^2)^{-\frac{1}{2\beta+1}},
$$
attains the following bound on the separation rate
$$
\Ec_{n,\alpha}(f_0,\gamma)\leq C(L,L_0,\gamma) \left\{ \vert B\vert ^{\frac{\beta+1}{2\beta+1}}(n\alpha^2)^{-\frac{\beta}{2\beta+1}} +  \int_{\overline B} f_0 + \frac 1{\sqrt{n \alpha^2}}\right\}.
$$
\end{theorem}
This result indicates to choose the optimal set $B = B_{n,\alpha}$ as a level set
\begin{equation}\label{optBInt}
B_{n,\alpha} = \arg \inf_{B \text{ compact set}}
\left\{ \int_{\overline{B}}{ f_0} \geq\vert B\vert ^{\frac{\beta+1}{2\beta+1}}(n\alpha^2)^{-\frac{\beta}{2\beta+1}} + \frac 1{\sqrt{n\alpha^2}} \text{ and } \inf_B f_0 \geq \sup_{\overline{B}} f_0 
\right\}. 
\end{equation}

\subsection{Lower bound in the interactive scenario}

In this subsection we complement the study of $\Ec_{n,\alpha}(f_0,\gamma)$ with a lower bound.
This lower bound will turn out to match the upper bound for several $f_0$, proving the optimality of the test and privacy mechanism proposed in the previous subsection for several $f_0$. See Section \ref{Section Examples} for the optimality.

\begin{theorem}\label{Thrm Interactive Lower bound}
Let $\alpha\in(0,1)$. Assume that $\beta\leq 1$. Recall that $z_\alpha=e^{2\alpha}-e^{-2\alpha}$ and $C_0(B)=\min\{f_0(x) : x\in B\}$.
For all compact sets $B\subset \Rb$ we get
$$
\Ec_{n,\alpha}(f_0,\gamma)\geq C(\gamma,L,L_0) \min \left\{ |B| C_0(B), \vert B \vert^{\frac{\beta+1}{2\beta+1}}(nz_\alpha^2)^{-\frac{\beta}{2\beta+1}}\right\}.
$$
If, moreover, $B$ is satisfying 
\begin{equation}\label{Eq. Condition on B for Inter Lower bound}
\vert B \vert^{\beta/(2\beta+1)}C_0(B) \geq C(nz_\alpha^2)^{-\beta/(2\beta+1)}
\end{equation}
for some $C>0$, it holds
$$
\Ec_{n,\alpha}(f_0,\gamma)\geq C(\gamma,L,L_0)\vert B \vert^{\frac{\beta+1}{2\beta+1}}(nz_\alpha^2)^{-\frac{\beta}{2\beta+1}}.
$$
\end{theorem}
The proof is deferred to Appendix \ref{App Proof Lower bound interactive scenario}.

Let us note that the same comment after Theorem~\ref{Thrm Lower bound NI} holds in this case. In all examples, we choose the set $B_{n,\alpha}$ as defined  in (\ref{optBInt}) and show that it checks the condition (\ref{Eq. Condition on B for Inter Lower bound}) giving thus minimax optimality of the testing rates.
\section{Examples}\label{Section Examples}

In this section, we investigate the optimality of our lower and upper bounds for some examples of densities $f_0$.  
For all the examples studied below, our bounds are optimal (up to a constant) in the interactive scenario, and optimal up to a logarithmic factor in the non-interactive scenario.

\begin{table}
\centering
\begin{tabular}{|c|c|p{4cm}|p{4cm}|}
\hline
 & Non-private separation rate & Private separation rate, non-interactive scenario (up to a log factor) & Private separation rate, interactive scenario\\
\hline
$\Uc([a,b])$ & $n^{-2/5}$ & $(n\alpha^2)^{-2/7}$ & $(n\alpha^2)^{-1/3}$   \\
\hline
$\Nc(0,1)$ & $n^{-2/5}$ & $\log(n\alpha^2)^{3/7}(n\alpha^2)^{-2/7}$ & $\log(n\alpha^2)^{1/3}(n\alpha^2)^{-1/3}$ \\
\hline
$\text{Beta}(a,b)$ & $n^{-2/5}$ & $(n\alpha^2)^{-2/7}$ & $(n\alpha^2)^{-1/3}$ \\
\hline
Spiky null & $n^{-2/5}$ & $(n\alpha^2)^{-2/7}$ & $(n\alpha^2)^{-1/3}$  \\
\hline
$\text{Cauchy}(0,a)$ & $(\log n)^{4/5}n^{-2/5}$ & $(n\alpha^2)^{-2/13}$ & $(n\alpha^2)^{-1/5}$\\
\hline
$\text{Pareto}(a,k)$ & $n^{-2k/(2+3k)}$ & $(n\alpha^2)^{-2k/(7k+6)}$ & $(n\alpha^2)^{-k/(3k+2)}$ \\
\hline
$\text{Exp}(\lambda)$ & $n^{-2/5}$ & $\log(n\alpha^2)^{6/7}(n\alpha^2)^{-2/7}$ & $\log(n\alpha^2)^{2/3}(n\alpha^2)^{-1/3}$ \\
\hline
\end{tabular}
\caption{Some examples of separation rates for different choices of densities $f_0$ and $\beta=1$. The non-private separation rates can be found in \cite{Balakrishnan_Wasserman_2019_hypothesisTesting}}
\label{Comparative table}
\end{table}

\bigskip

The densities considered in this section are Hölder continuous with exponent $\beta$ for all $\beta\in(0,1]$ unless otherwise specified. The results are stated for $n$ large enough and $\alpha\in(0,1)$ such that $n\alpha^2\rightarrow +\infty$ as $n\rightarrow \infty$.
They are summarised in Table \ref{Comparative table} for $\beta=1$ and compared to the non-private separation rates. The proofs can be found in Appendix \ref{App Proof examples}.

\begin{example}\label{Ex uniform}
Assume that $f_0$ is the density of the continuous uniform distribution on $[a,b]$ where $a$ and $b$ are two constants satisfying $a<b$, that is 
\begin{equation*}\label{Eq. density Uniform}
f_0(x)=\frac{1}{b-a}I(x\in[a,b]).
\end{equation*}
Taking $B=[a,b]$ in Theorems \ref{Thrm Lower bound NI}, \ref{Thrm Upper Bound NI}, \ref{Thrm Interactive Lower bound} and \ref{Thrm Interactive Upper bound} yields the following bounds on the minimax radius
$$
\left[\log\left(C (n\alpha^2)^\frac{2}{4\beta+3}  \right)\right]^{-1}(n\alpha^2)^{-\frac{2\beta}{4\beta+3}} \lesssim \Ec^{\text{NI}}_{n,\alpha}(f_0,\gamma) \lesssim(n\alpha^2)^{-\frac{2\beta}{4\beta+3}} ,
$$
and 
$$
 \Ec_{n,\alpha}(f_0,\gamma)\asymp (n\alpha^2)^{-\frac{\beta}{2\beta+1}} 
$$
\end{example}

\begin{example}\label{Ex Pareto}
Assume that $f_0$ is the density of the Pareto distribution with parameters $a>0$ and $k>0$, that is 
\begin{equation*}\label{Eq. density Pareto}
f_0(x)=\frac{ka^k}{x^{k+1}}I(x\geq a).
\end{equation*}
It holds
$$
 \left[\log\left(C (n\alpha^2)^{\frac{4\beta+4}{4\beta+3}\cdot \frac{2\beta}{k(4\beta+3)+3\beta+3}+\frac{2}{4\beta+3}}  \right)\right]^{-1}(n\alpha^2)^{-\frac{2k\beta}{k(4\beta+3)+3\beta+3}} \lesssim \Ec^{\text{NI}}_{n,\alpha}(f_0,\gamma) \lesssim (n\alpha^2)^{-\frac{2k\beta}{k(4\beta+3)+3\beta+3}} ,
$$
and 
$$
 \Ec_{n,\alpha}(f_0,\gamma)\asymp (n\alpha^2)^{-\frac{k\beta}{k(2\beta+1)+\beta+1}}.
$$
\end{example}

\begin{example}\label{Ex Exponential}
Assume that $f_0$ is the density of the exponential distribution with parameter $\lambda>0$, that is 
\begin{equation*}\label{Eq. density exponential}
f_0(x)=\lambda \exp(-\lambda x)I(x\geq 0).
\end{equation*}
It holds 
$$
\left[\log\left(C \log(n\alpha^2)^{\frac{4\beta+4}{4\beta+3}}(n\alpha^2)^\frac{2}{4\beta+3}  \right)\right]^{-1}\log(n\alpha^2)^{\frac{3\beta+3}{4\beta+3}}(n\alpha^2)^{-\frac{2\beta}{4\beta+3}}  \lesssim \Ec^{\text{NI}}_{n,\alpha}(f_0,\gamma) \lesssim \log(n\alpha^2)^{\frac{3\beta+3}{4\beta+3}}(n\alpha^2)^{-\frac{2\beta}{4\beta+3}},
$$
and 
$$
 \Ec_{n,\alpha}(f_0,\gamma)\asymp \log(n\alpha^2)^{\frac{\beta+1}{2\beta+1}}(n\alpha^2)^{-\frac{\beta}{2\beta+1}}
$$
\end{example}

\begin{example}\label{Ex Normal}
Assume that $f_0$ is the density of the normal distribution with parameters $0$ and $1$, that is 
\begin{equation*}\label{Eq. density normal}
f_0(x)=\frac{1}{\sqrt{2\pi}}\exp\left( -\frac{x^2}{2} \right).
\end{equation*}
It holds 
$$
 \left[\log\left(C \log(n\alpha^2)^{\frac{4\beta+4}{2(4\beta+3)}}(n\alpha^2)^\frac{2}{4\beta+3}  \right)\right]^{-1}\log(n\alpha^2)^{\frac{3\beta+3}{2(4\beta+3)}}(n\alpha^2)^{-\frac{2\beta}{4\beta+3}} \lesssim \Ec^{\text{NI}}_{n,\alpha}(f_0,\gamma) \lesssim  \log(n\alpha^2)^{\frac{3\beta+3}{2(4\beta+3)}}(n\alpha^2)^{-\frac{2\beta}{4\beta+3}},
$$
and 
$$
 \Ec_{n,\alpha}(f_0,\gamma)\asymp  \log(n\alpha^2)^{\frac{\beta+1}{2(2\beta+1)}}(n\alpha^2)^{-\frac{\beta}{2\beta+1}}
$$
\end{example}

\begin{example}\label{Ex Cauchy}
Assume that $f_0$ is the density of the Cauchy distribution with parameters $0$ and $a>0$, that is 
\begin{equation*}\label{Eq. density Cauchy}
f_0(x)=\frac{1}{\pi a}\frac{a^2}{x^2+a^2}.
\end{equation*}
It holds 
$$
\left[\log\left(C (n\alpha^2)^{\frac{4\beta+4}{4\beta+3}\cdot\frac{2\beta}{7\beta+6} + \frac{2}{4\beta+3} } \right)\right]^{-1}(n\alpha^2)^{-\frac{2\beta}{7\beta+6}}  \lesssim \Ec^{\text{NI}}_{n,\alpha}(f_0,\gamma) \lesssim (n\alpha^2)^{-\frac{2\beta}{7\beta+6}},
$$
and 
$$
 \Ec_{n,\alpha}(f_0,\gamma)\asymp (n\alpha^2)^{-\frac{\beta}{3\beta+2}}
$$
\end{example}

\begin{example}\label{Ex Spiky null}
Assume that the density $f_0$ is given by
\begin{equation*}\label{Eq. density spiky null}
f_0(x)=
\begin{cases}
L_0x & \text{if } 0\leq x\leq \frac{1}{\sqrt{L_0}}\\
2\sqrt{L_0}-L_0x &  \text{if } \frac{1}{\sqrt{L_0}}\leq x\leq \frac{2}{\sqrt{L_0}} \\
0 & \text{otherwise.}
\end{cases}
\end{equation*}
It holds 
$$
\left[\log\left(C (n\alpha^2)^\frac{2}{4\beta+3}  \right)\right]^{-1}(n\alpha^2)^{-\frac{2\beta}{4\beta+3}} \lesssim \Ec^{\text{NI}}_{n,\alpha}(f_0,\gamma) \lesssim  (n\alpha^2)^{-\frac{2\beta}{4\beta+3}},
$$
and 
$$
 \Ec_{n,\alpha}(f_0,\gamma)\asymp (n\alpha^2)^{-\frac{\beta}{2\beta+1}}
$$
\end{example}

\begin{example}\label{Ex Beta}
Assume that $f_0$ is the density of the Beta distribution with parameters $a\geq 1$ and $b\geq 1$, that is 
\begin{equation}\label{Eq. density Beta}
f_0(x)=\frac{1}{B(a,b)}x^{a-1}(1-x)^{b-1}I(0< x<1),
\end{equation}
where $B(\cdot,\cdot)$ is the Beta function. It holds
$$
\left[\log\left(C(n\alpha^2)^\frac{2}{4\beta+3}  \right)\right]^{-1}(n\alpha^2)^{-\frac{2\beta}{4\beta+3}}  \lesssim \Ec^{\text{NI}}_{n,\alpha}(f_0,\gamma) \lesssim (n\alpha^2)^{-\frac{2\beta}{4\beta+3}},
$$
and 
$$
 \Ec_{n,\alpha}(f_0,\gamma)\asymp (n\alpha^2)^{-\frac{\beta}{2\beta+1}}.
$$
Note that the density $f_0$ given by \eqref{Eq. density Beta} can be defined for all $a>0$ and $b>0$. However, $f_0$ is Hölder continuous for no exponent $\beta \in (0,1]$ if $a<1$ or $b<1$. Note also that if $a=1$ and $b=1$ then $f_0$ is the density of the continuous uniform distribution on $[0,1]$, and this case has already been tackled in Example \ref{Ex uniform}.
Now, if $a=1$ and $b>1$ (respectively $a>1$ and $b=1$), one can check that $f_0$ is Hölder continuous with exponent $\beta$ for all $\beta\in (0,\min\{b-1,1 \}]$ (respectively $\beta\in (0,\min\{a-1,1 \}])$. Finally, if $a>1$ and $b>1$ then $f_0$ is is Hölder continuous with exponent $\beta$ for all $\beta\in (0,\min\{a-1,b-1,1 \}]$.
\end{example}

\begin{example}\label{Ex Slow varying function}
Assume that the density $f_0$ is given by $$f_0(x)= \frac{A \log(2)^A}{(x+2) \log^{A+1}(x+2) }I(x\geq 0),$$ for some $A>0$ which can be arbitrarily small but fixed. It holds
$$\left[\log\left(C a_*^{\frac{4\beta+4}{4\beta+3}}(n\alpha^2)^\frac{2}{4\beta+3}  \right)\right]^{-1}\left[\log(a_*) \right]^{-1} a_*^{\frac{3\beta+3}{4\beta+3}}(n\alpha^2)^{-\frac{2\beta}{4\beta+3}} \lesssim \Ec_{n,\alpha}^\text{NI}(f_0,\gamma) \lesssim a_*^{\frac{3\beta+3}{4\beta+3}}(n\alpha^2)^{-\frac{2\beta}{4\beta+3}},$$
where 
$$
a_*=\sup\left\{ a \geq 0 : \frac{(\log 2)^{A}}{ \log^{A}(2+a) } \geq a^{\frac{3\beta+3}{4\beta +3}}(n\alpha^2)^{-\frac{2\beta}{4\beta +3}}+\frac{1}{\sqrt{n\alpha^2}}\right\}.
$$ It is easy to see that $a_* >1 $ is up to some log factors a polynomial of $n \alpha^2$: $a_*^{\frac{3\beta+3}{4\beta +3}} \asymp (n\alpha^2)^{\frac{2\beta}{4\beta +3}}/ \log^A(2+a_*) $ and therefore 
$$
a_*^{\frac{3\beta+3}{4\beta +3}}(n\alpha^2)^{-\frac{2\beta}{4\beta +3}} \asymp \frac 1{\log^A(n \alpha^2)}. 
$$
In the interactive case
$$\left[\log(b_*) \right]^{-1} b_*^{\frac{\beta+1}{2\beta+1}}(n\alpha^2)^{-\frac{\beta}{2\beta+1}}\lesssim \Ec_{n,\alpha}(f_0,\gamma)\lesssim b_*^{\frac{\beta+1}{2\beta+1}}(n\alpha^2)^{-\frac{\beta}{2\beta+1}}, $$
where
$$b_*=\sup\left\{ b \geq 0 : \frac{(\log 2)^{A}}{ \log^{A}(2+b) } \geq b^{\frac{\beta+1}{2\beta +1}}(n\alpha^2)^{-\frac{\beta}{2\beta +1}}+\frac{1}{\sqrt{n\alpha^2}}\right\}.$$ Similarly to the non-interactive case, $b_*$ is up to log factors a polynomial of $n \alpha^2$ and therefore 
$$
b_*^{\frac{\beta+1}{2\beta+1}}(n\alpha^2)^{-\frac{\beta}{2\beta+1}} \asymp  \frac 1{\log^A(n \alpha^2)}.
$$

\end{example}

\appendix
\section{Proofs of Section \ref{Section Non-interactive scenario}}

\subsection{Proof of Proposition \ref{Prop Privacy of the NI Mechanism}}\label{Section Proof Privacy of the NI mechanism}

Let $i\in\llbr 1,n\rrbr$.
Set $\sigma:=2\Vert \psi \Vert_\infty/(\alpha h)$.
The conditional density of $Z_i$ given $X_i=y$ can be written as 
$$
q^{Z_i\mid X_i=y}(z)=\prod_{j=1}^N \frac{1}{2\sigma}\exp\left(-\frac{\vert z_j-\psi_h(x_j-y)\vert}{\sigma}  \right).
$$
Thus, by the reverse and the ordinary triangle inequality,
\begin{align*}
\frac{q^{Z_i\mid X_i=y}(z)}{q^{Z_i\mid X_i=y'}(z)}&=\prod_{j=1}^N \exp\left(\frac{\vert z_j-\psi_h(x_j-y')\vert-\vert z_j-\psi_h(x_j-y)\vert}{\sigma}  \right)\\
&\leq \prod_{j=1}^N \exp\left(\frac{\vert \psi_h(x_j-y')-\psi_h(x_j-y)\vert}{\sigma}  \right)\\
&\leq \exp\left(\frac{1}{\sigma h}\sum_{j=1}^{N}\left\vert \psi\left(\frac{x_j-y'}{h}\right)-\psi\left(\frac{x_j-y}{h}\right)\right\vert  \right)\\
&\leq \exp\left(\frac{1}{\sigma h}\sum_{j=1}^{N}\left[ \left\vert \psi\left(\frac{x_j-y'}{h}\right)\right\vert+\left\vert\psi\left(\frac{x_j-y}{h}\right)\right\vert \right] \right)\\
&\leq \exp\left(\frac{2\Vert \psi\Vert_\infty}{\sigma h}  \right)\\
&\leq \exp(\alpha),
\end{align*}
where the second to last inequality follows from the fact that for a fixed $y$ the quantity $\psi((x_j-y)/h)$ is non-zero for at most one coefficient $j\in\llbr 1,N\rrbr$. This is a consequence of Assumption \ref{Assumption on the kernel}.
This proves that $Z_i$ is an $\alpha$-locally differentially private view of $X_i$ for all $i\in\llbr 1,n\rrbr $.\\
Consider now $i\in\llbr n+1,2n \rrbr$.
For all $j\in\llbr 1,N\rrbr$ it holds
$$
\frac{\Pb\left(Z_i=c_\alpha \mid X_i\notin B \right)}{\Pb\left(Z_i=c_\alpha \mid X_i \in B_j \right)}=1+\frac{1}{c_\alpha}= \frac{2e^\alpha}{e^\alpha+1}.
$$
Since $2 \leq e^\alpha+1\leq 2e^\alpha$, we obtain
$$
e^{-\alpha}\leq 1 \leq \frac{\Pb\left(Z_i=c_\alpha \mid X_i\notin B \right)}{\Pb\left(Z_i=c_\alpha \mid X_i \in B_j \right)}\leq e^\alpha.
$$
It also holds
$$
\frac{\Pb\left(Z_i=-c_\alpha \mid X_i\notin B \right)}{\Pb\left(Z_i=-c_\alpha \mid X_i \in B_j \right)}=1-\frac{1}{c_\alpha}= \frac{2}{e^\alpha+1} \in [e^{-\alpha}, e^{\alpha}].
$$
Now, for all $(j,k)\in\llbr 1, N\rrbr ^2$ it holds 
$$
\frac{\Pb\left(Z_i=c_\alpha \mid X_i \in B_k \right)}{\Pb\left(Z_i=c_\alpha \mid X_i \in B_j \right)}=\frac{\Pb\left(Z_i=-c_\alpha \mid X_i\in B_k \right)}{\Pb\left(Z_i=-c_\alpha \mid X_i \in B_j \right)}=1\in[e^{-\alpha}, e^\alpha].
$$
This proves that $Z_i$ is an $\alpha$-locally differentially private view of $X_i$ for all $i\in\llbr n+1,2n\rrbr$.

\subsection{Proof of Theorem \ref{Thrm Upper Bound NI}}\label{App Proof Upper Bound NI}

\begin{proof}[Proof of Proposition~\ref{Prop Test Stat NI}]
1. Equality \eqref{Eq Mean of the chi2stat NI} follows from the independance of $Z_i$ and $Z_k$ for $i\neq k$ and from $\Eb[Z_{ij}]=\psi_h\ast f(x_j)$.
We now prove \eqref{Eq Variance of the chi2stat NI}. Set $a_{h,j}:=\psi_h\ast f(x_j)$ and let us define
$$
\widehat{U}_B=\frac{1}{n(n-1)}\sum_{i\neq k}\sum_{j=1}^N\left( Z_{ij}-a_{h,j} \right)\left( Z_{kj}-a_{h,j} \right), 
$$
$$
\widehat V_B=\frac{2}{n}\sum_{i=1}^n\sum_{j=1}^N\left( a_{h,j}-f_0(x_j) \right)\left( Z_{ij}-a_{h,j} \right), 
$$ 
and observe that we have 
$$
S_B=\widehat{U}_B+\widehat{V}_B+\sum_{j=1}^N(a_{h,j}-f_0(x_j))^2.
$$
Note that $\Cov(\widehat{U}_B,\widehat{V}_B)=0$.
We thus have
$$
\Var (S_B)=\Var (\widehat{U}_B)+\Var (\widehat{V}_B),
$$
and we will bound from above $\Var (\widehat{U}_B)$ and $\Var (\widehat{V}_B)$ separately.
We begin with  $\Var (\widehat{V}_B)$.
Since $\widehat{V}_B$ is centered, it holds
\begin{align*}
\Var (\widehat{V}_B)&=\Eb[\widehat{V}_B^2]\\
&=\frac{4}{n^2}\sum_{i=1}^n\sum_{j=1}^N\sum_{t=1}^n\sum_{k=1}^N \left( a_{h,j}-f_0(x_j) \right)\left( a_{h,k}-f_0(x_k) \right)\Eb\left[\left( Z_{ij}-a_{h,j} \right)\left( Z_{tk}-a_{h,k} \right)\right].
\end{align*}
Note that if $t\neq i$, the independance of $Z_i$ and $Z_t$ yields
$$
\Eb\left[\left( Z_{ij}-a_{h,j} \right)\left( Z_{tk}-a_{h,k} \right)\right]=0.
$$
Moreover, since the $W_{ij}$, $j=1,\ldots, N$ are independent of $X_i$ and $\Eb[W_{ij}]=0$ we have
\begin{align*}
\Eb\left[\left( Z_{ij}-a_{h,j} \right)\left( Z_{ik}-a_{h,k} \right)\right]&=\Eb\left[ \left( \psi_h\left(x_j-X_i\right)+\frac{2\Vert \psi \Vert_\infty}{\alpha h} W_{ij}-a_{h,j}\right)\left( \psi_h\left(x_k-X_i\right)+\frac{2\Vert \psi \Vert_\infty}{\alpha h} W_{ik}-a_{h,k}\right)\right]\\
&=\Eb\left[\psi_h\left(x_j-X_i\right)\psi_h\left(x_k-X_i\right) \right]-a_{h,k}\Eb\left[ \psi_h\left(x_j-X_i\right) \right]+\frac{4\Vert \psi \Vert_\infty^2}{\alpha^2 h^2}\Eb\left[ W_{ij}W_{ik} \right]\\
&\hspace{2cm} -a_{h,j}\Eb\left[ \psi_h\left(x_k-X_i\right) \right]+a_{h,j}a_{h,k}\\
&=\left[\int \left(\psi_h\left(x_j-y\right)\right)^2f(y)\dd y +\frac{8\Vert \psi \Vert_\infty^2}{\alpha^2 h^2}\right]I(j=k)-a_{h,j}a_{h,k},
\end{align*}
where the last equality is a consequence of Assumption \ref{Assumption on the kernel}.
We thus obtain
\begin{align*}
\Var (\widehat{V}_B)&=\frac{4}{n}\sum_{j=1}^N\left( a_{h,j}-f_0(x_j) \right)^2\left[\int( \psi_h\left(x_j-y\right))^2f(y)\dd y +\frac{8\Vert \psi \Vert_\infty^2}{\alpha^2 h^2}\right]\\
&\hspace{2cm}-\frac{4}{n}\sum_{j=1}^N\sum_{k=1}^N\left( a_{h,j}-f_0(x_j) \right)\left( a_{h,k}-f_0(x_k) \right)a_{h,j}a_{h,k}\\
&=\frac{4}{n}\sum_{j=1}^N\left( a_{h,j}-f_0(x_j) \right)^2\left[\int (\psi_h\left(x_j-y\right))^2f(y)\dd y +\frac{8\Vert \psi \Vert_\infty^2}{\alpha^2 h^2}\right]
-\frac{4}{n}\left(\sum_{j=1}^N\left( a_{h,j}-f_0(x_j) \right)a_{h,j}\right)^2\\
&\leq \frac{4}{n}\sum_{j=1}^N\left( a_{h,j}-f_0(x_j) \right)^2\left[\int (\psi_h\left(x_j-y\right))^2f(y)\dd y +\frac{8\Vert \psi \Vert_\infty^2}{\alpha^2 h^2}\right].
\end{align*}
Now,  $\int (\psi_h\left(x_j-y\right))^2f(y)\dd y\leq \Vert \psi_h\Vert_\infty^2\leq \Vert \psi\Vert_\infty^2/h^2\leq \Vert \psi\Vert_\infty^2/(\alpha^2h^2)$ if $\alpha\in(0,1]$.
We finally obtain 
$$
\Var (\widehat{V}_B)\leq \frac{36\Vert \psi \Vert_\infty^2}{n\alpha^2 h^2}\sum_{j=1}^N\left( a_{h,j}-f_0(x_j) \right)^2.
$$
We now bound from above $\Var (\widehat{U}_B)$.
One can rewrite $\widehat{U}_B$ as
$$
\widehat{U}_B=\frac{1}{n(n-1)}\sum_{i\neq k}h(Z_i,Z_k),
$$
where 
$$
h(Z_i,Z_k)=\sum_{j=1}^N\left( Z_{ij}-a_{h,j} \right)\left( Z_{kj}-a_{h,j} \right).
$$
Using a result for the variance of a $U$-statistic (see for instance Lemma A, p.183 in \cite{Serfling_1980}), we have
$$
\binom{n}{2}\Var(\widehat{U}_B)=2(n-2)\zeta_1+\zeta_2,
$$
where 
$$
\zeta_1=\Var\left(\Eb\left[ h(Z_1,Z_2)\mid Z_1 \right] \right) , \text{ and } \zeta_2=\Var\left( h(Z_1,Z_2) \right).
$$
We have $\zeta_1=0$ since $\Eb\left[ h(Z_1,Z_2)\mid Z_1 \right] =0$ and thus 
$$
\Var (\widehat{U}_B)=\frac{2}{n(n-1)}\Var\left( h(Z_1,Z_2) \right).
$$
Write
\begin{align*}
h(Z_1,Z_2)&= \sum_{j=1}^N \left( \psi_h\left(x_j-X_1\right)+\frac{2\Vert \psi \Vert_\infty}{\alpha h} W_{1j}-a_{h,j}\right)\left( \psi_h\left(x_j-X_2\right)+\frac{2\Vert \psi \Vert_\infty}{\alpha h} W_{2j}-a_{h,j}\right)\\
&=\sum_{j=1}^N\left(\psi_h\left(x_j-X_1\right)-a_{h,j}\right)\left(\psi_h\left(x_j-X_2\right)-a_{h,j}\right) +\frac{4\Vert \psi \Vert_\infty^2}{\alpha^2 h^2}\sum_{j=1}^N W_{1j}W_{2j} \\
&\hspace{1cm} +\frac{2\Vert \psi \Vert_\infty}{\alpha h}\sum_{j=1}^N W_{1j}(\psi_h(x_j-X_2)-a_{h,j})+\frac{2\Vert \psi \Vert_\infty}{\alpha h}\sum_{j=1}^N W_{2j}(\psi_h(x_j-X_1)-a_{h,j})\\
&=: \tilde{T}_1+\tilde{T}_2+\tilde{T}_3+\tilde{T}_4.
\end{align*}
We thus have $\Var(h(Z_1,Z_2))=\sum_{i=1}^4 \Var(\tilde{T}_i)+2\sum_{i<j}\Cov(\tilde{T}_i,\tilde{T}_j)$. 
Observe that $\Cov(\tilde{T}_i,\tilde{T}_j)=0$ for $i<j$ and $\Var(\tilde{T}_3)=\Var(\tilde{T}_4)$.
We thus have
$$
\Var(h(Z_1,Z_2))=\Var(\tilde{T}_1)+\Var(\tilde{T}_2)+2\Var(\tilde{T}_3).
$$
The independence of the random variables $(W_{ij})_{i,j}$ yields 
$$
\Var(\tilde{T}_2)=\frac{64\Vert \psi \Vert_\infty^4N}{\alpha^4 h^4}.
$$
The independence of the random variables $(W_{ij})_{i,j}$ and  their independence with $X_2$ yield 
\begin{align*}
\Var(\tilde{T}_3)&=\Eb\left[\tilde{T}_3^2 \right]\\
&=\frac{4\Vert \psi \Vert_\infty^2}{\alpha^2 h^2}\Eb\left[ \sum_{j=1}^N W_{1j}(\psi_h(x_j-X_2)-a_{h,j})\sum_{k=1}^N W_{1k}(\psi_h(x_k-X_2)-a_{h,k}) \right]\\
&=\frac{4\Vert \psi \Vert_\infty^2}{\alpha^2 h^2}\sum_{j=1}^N\sum_{k=1}^N\Eb\left[ W_{1j}W_{1k}\right]\Eb\left[ (\psi_h(x_j-X_2)-a_{h,j})(\psi_h(x_k-X_2)-a_{h,k}) \right]\\
&=\frac{8\Vert \psi \Vert_\infty^2}{\alpha^2 h^2}\sum_{j=1}^N\Eb\left[ (\psi_h(x_j-X_2)-a_{h,j})^2 \right]\\
&\leq \frac{8\Vert \psi \Vert_\infty^2}{\alpha^2 h^2}\sum_{j=1}^N\Eb\left[ (\psi_h(x_j-X_2))^2 \right].
\end{align*}
Now, since $y\mapsto \psi_h(x_j-y)$ is null outside $B_j$ (consequence of Assumption \ref{Assumption on the kernel}), it holds
$$
\sum_{j=1}^N\Eb\left[ (\psi_h(x_j-X_2))^2 \right]=\sum_{j=1}^N\int_{B_j}\left(\psi_h(x_j-y) \right)^2f(y) dy\leq \Vert \psi_h\Vert_\infty^2\sum_{j=1}^N\int_{B_j}f\leq \Vert \psi_h\Vert_\infty^2,
$$
and thus 
$$
\Var(\tilde{T}_3)\leq \frac{8\Vert \psi \Vert_\infty^4}{\alpha^2 h^4}.
$$
By independence of $X_1$ and $X_2$, it  holds $\Eb[\tilde{T}_1]=0$, and
\begin{align*}
\Var(\tilde{T}_1)&=\Eb\left[  \tilde{T}_1^2\right]\\
&=\sum_{j=1}^N\sum_{k=1}^N\Eb\left[\left(\psi_h\left(x_j-X_1\right)-a_{h,j}\right)\left(\psi_h\left(x_j-X_2\right)-a_{h,j}\right)\left(\psi_h\left(x_k-X_1\right)-a_{h,k}\right)\left(\psi_h\left(x_k-X_2\right)-a_{h,k}\right) \right]\\
&=\sum_{j=1}^N\sum_{k=1}^N\Eb\left[\left(\psi_h\left(x_j-X_1\right)-a_{h,j}\right)\left(\psi_h\left(x_k-X_1\right)-a_{h,k}\right)\right]\Eb\left[\left(\psi_h\left(x_j-X_2\right)-a_{h,j}\right)\left(\psi_h\left(x_k-X_2\right)-a_{h,k}\right) \right]\\
&= \sum_{j=1}^N\sum_{k=1}^N\left[\int \psi_h(x_j-y)\psi_h(x_k-y)f(y)\dd y- a_{h,j}a_{h,k} \right]^2\\
&=\sum_{j=1}^N\sum_{k=1}^N \left( \int \psi_h(x_j-y)\psi_h(x_k-y)f(y)\dd y \right)^2-2\sum_{j=1}^N\sum_{k=1}^Na_{h,j}a_{h,k}\int \psi_h(x_j-y)\psi_h(x_k-y)f(y)\dd y\\
&\hspace{2cm}+\sum_{j=1}^N\sum_{k=1}^Na_{h,j}^2a_{h,k}^2.
\end{align*}
Assumption \ref{Assumption on the kernel} yields $\int \psi_h(x_j-y)\psi_h(x_k-y)f(y)\dd y=0$ if $j\neq k$. 
We thus obtain
\begin{align*}
\Var(\tilde{T}_1)&=\sum_{j=1}^N\left( \int (\psi_h(x_j-y))^2f(y)\dd y \right)^2-2\sum_{j=1}^Na_{h,j}^2\int \left(\psi_h(x_j-y)\right)^2f(y)\dd y+\left(\sum_{j=1}^Na_{h,j}^2\right)^2.
\end{align*}
Now, since $y\mapsto \psi_h(x_j-y)$ is  null outside $B_j$ (consequence of Assumption \ref{Assumption on the kernel}),  observe that 
$$
\sum_{j=1}^N\left( \int (\psi_h(x_j-y))^2f(y)\dd y \right)^2\leq \frac{\Vert \psi\Vert_\infty^4}{h^4}\sum_{j=1}^N\left(\int_{B_j}f \right)^2\leq \frac{\Vert \psi\Vert_\infty^4}{h^4}\sum_{j=1}^N\int_{B_j}f \leq \frac{\Vert \psi\Vert_\infty^4}{h^4},
$$
and 
$$
\left(\sum_{j=1}^Na_{h,j}^2\right)^2=\left(\sum_{j=1}^N \left( \int \psi_h(x_j-y)f(y)\dd y \right)^2\right)^2\leq \frac{\Vert \psi\Vert_\infty^4}{h^4}\left[\sum_{j=1}^N\left(\int_{B_j}f \right)^2\right]^2\leq \frac{\Vert \psi\Vert_\infty^4}{h^4},
$$
yielding $\Var(\tilde{T}_1)\leq 2\frac{\Vert \psi\Vert_\infty^4}{h^4}$.
We thus have
$$
\Var(\widehat{U}_B)\leq \frac{2}{n(n-1)}\left[2\frac{\Vert \psi\Vert_\infty^4}{h^4}+ \frac{64\Vert \psi \Vert_\infty^4N}{\alpha^4 h^4}+ \frac{16\Vert \psi \Vert_\infty^4}{\alpha^2 h^4} \right]\leq \frac{164\Vert \psi \Vert_\infty^4N}{n(n-1)\alpha^4 h^4}.
$$
Finally,
$$
\Var (S_B)\leq \frac{36\Vert \psi \Vert_\infty^2}{n\alpha^2 h^2}\sum_{j=1}^N\left( a_{h,j}-f_0(x_j) \right)^2+\frac{164\Vert \psi \Vert_\infty^4N}{n(n-1)\alpha^4 h^4}.
$$

2.  For all $i\in\llbr n+1,2n\rrbr$ it holds
\begin{align*}
\Eb_{Q_{f}^n}[Z_i]&=\Eb\left[Z_i\mid X_i\notin B \right]\Pb\left(X_i \notin B \right)+\sum_{j=1}^N\Eb\left[Z_i\mid X_i\in B_j \right]\Pb\left(X_i\in B_j \right)\\
&=\left[ c_\alpha\cdot \frac{1}{2}\left(1+\frac{1}{c_\alpha} \right)-c_\alpha\cdot \frac{1}{2}\left(1-\frac{1}{c_\alpha} \right)\right]\Pb\left(X_i \notin B \right)+\sum_{j=1}^N\left[c_\alpha\cdot\frac{1}{2}-c_\alpha\cdot\frac{1}{2} \right]\Pb\left(X_i\in B_j \right)\\
&=\Pb\left(X_i \notin B \right).
\end{align*}
This yields $\Eb_{Q_{f}^n}[T_B] =\int_{\overline B} (f-f_0)$, and using the independence of the $Z_i$, $i=n+1,\ldots,2n$ we obtain
$$
\Var_{Q_{f}^n}[T_B]=\frac{1}{n^2}\sum_{i=n+1}^{2n}\Var(Z_i)=\frac{1}{n^2}\sum_{i=n+1}^{2n}\left[\Eb[Z_i^2]-\Eb[Z_i]^2 \right]=\frac{1}{n}\left(c_\alpha^2-\left(\int_{\overline B} f \right)^2  \right).
$$
\end{proof}

We can now proove Theorem \ref{Thrm Upper Bound NI}.
We first prove that the choice of $t_1$ and $t_2$ in \eqref{Eq Non Interactive t_1 and t_2} gives $\Pb_{Q_{f_0}^n}(\Phi=1)\leq \gamma/2$.
Since $\Eb_{Q_{f_0}^n}[T_B]=0$, Chebyshev's inequality and Proposition~\ref{Prop Test Stat NI} yield for $\alpha\in(0,1]$
$$
 \Pb_{Q_{f_0}^n}(T_B \geq t_2)\leq \Pb_{Q_{f_0}^n}(\vert T_B \vert \geq t_2)\leq \frac{\Var_{Q_{f_0}^n}(T_B)}{t_2^2}\leq \frac{c_\alpha^2 }{n t_2^2}\leq \frac{5 }{n \alpha^2 t_2^2}= \frac{\gamma}{4}.
 $$
  If $t_1>\Eb_{Q_{f_0}^n}[S_B]= \sum_{j=1}^N\left([\psi_h\ast f_0](x_j)-f_0(x_j) \right)^2$, then Chebychev's inequality and Proposition~\ref{Prop Test Stat NI} yield
 \begin{align*}
 \Pb_{Q_{f_0}^n}(S_B \geq t_1)&\leq \Pb_{Q_{f_0}^n}(\vert S_B -\Eb_{Q_{f_0}^n}[S_B]\vert \geq t_1-\Eb_{Q_{f_0}^n}[S_B])\\
 &\leq \frac{\Var_{Q_{f_0}^n}(S_B)}{(t_1-\Eb_{Q_{f_0}^n}[S_B])^2}\\
 &\leq \frac{\frac{36\Vert \psi\Vert_\infty^2}{n \alpha^2 h^2}\sum_{j=1}^N\left([\psi_h\ast f_0](x_j)-f_0(x_j) \right)^2}{\left(t_1-\sum_{j=1}^N\left([\psi_h\ast f_0](x_j)-f_0(x_j) \right)^2\right)^2}+\frac{\frac{164\Vert \psi\Vert_\infty^4N}{n(n-1)\alpha^4h^4}}{\left(t_1-\sum_{j=1}^N\left([\psi_h\ast f_0](x_j)-f_0(x_j) \right)^2\right)^2}.
 \end{align*}
 Observe that 
 $$
 t_1\geq \sum_{j=1}^N\left([\psi_h\ast f_0](x_j)-f_0(x_j) \right)^2+\max\left\{\sqrt{ \frac{288\Vert \psi\Vert_\infty^2}{\gamma n\alpha^2 h^2} \sum_{j=1}^N\left([\psi_h\ast f_0](x_j)-f_0(x_j) \right)^2}, \sqrt{\frac{1312\Vert \psi\Vert_\infty^4N}{\gamma n(n-1)\alpha^4 h^4} }\right\}.
 $$
 Indeed for $f\in H(\beta,L)$ with $\beta\leq 1$ it holds $\left\vert[\psi_h\ast f](x_j)-f(x_j) \right\vert \leq LC_\beta h^{\beta}$ for all $j\in\llbr 1,N\rrbr$ where $C_\beta=\int_{-1}^1\vert u\vert^\beta \vert \psi(u)\vert \dd u$, and thus using $ab\leq a^2/2+b^2 /2$ we obtain
 \begin{align*}
 &\sum_{j=1}^N\left([\psi_h\ast f_0](x_j)-f_0(x_j) \right)^2+\max\left\{\sqrt{ \frac{288\Vert \psi\Vert_\infty^2}{\gamma n\alpha^2 h^2} \sum_{j=1}^N\left([\psi_h\ast f_0](x_j)-f_0(x_j) \right)^2}, \sqrt{\frac{1312\Vert \psi\Vert_\infty^4N}{\gamma n(n-1)\alpha^4 h^4} }\right\}\\
 &\hspace{1cm}\leq L_0^2C_\beta^2Nh^{2\beta}+\max\left\{\frac{1}{2}L_0^2C_\beta^2Nh^{2\beta}+\frac{144\Vert \psi\Vert_\infty^2}{\gamma n\alpha^2 h^2}, \sqrt{\frac{1312\Vert \psi\Vert_\infty^4N}{\gamma n(n-1)\alpha^4 h^4} }\right\}\\
 &\hspace{1cm}\leq \frac{3}{2}L_0^2C_\beta^2Nh^{2\beta}+\frac{144\Vert \psi\Vert_\infty^2}{\gamma n\alpha^2 h^2}+ \frac{52\Vert\psi\Vert_\infty^2\sqrt{N}}{\sqrt{\gamma}n\alpha^2h^2}\\
&\hspace{1cm}\leq \frac{3}{2}L_0^2C_\beta^2Nh^{2\beta}+ \frac{196\Vert\psi\Vert_\infty^2\sqrt{N}}{\gamma n\alpha^2h^2}=t_1.
\end{align*}
Then it holds
\begin{align*}
  \Pb_{Q_{f_0}^n}(S_B \geq t_1)&\leq \frac{\frac{36\Vert \psi\Vert_\infty^2}{n \alpha^2 h^2}\sum_{j=1}^N\left([\psi_h\ast f_0](x_j)-f_0(x_j) \right)^2}{(t_1-\sum_{j=1}^N\left([\psi_h\ast f_0](x_j)-f_0(x_j) \right)^2)^2}+\frac{\frac{164\Vert \psi\Vert_\infty^4N}{n(n-1)\alpha^4h^4}}{(t_1-\sum_{j=1}^N\left([\psi_h\ast f_0](x_j)-f_0(x_j) \right)^2)^2}\\
  &\leq \frac{\gamma}{8}+\frac{\gamma}{8}\leq \frac{\gamma}{4},
\end{align*}
and thus
$$
\Pb_{Q_{f_0}^n}(\Phi=1)\leq  \Pb_{Q_{f_0}^n}(T_B \geq t_2) +\Pb_{Q_{f_0}^n}(S_B \geq t_1)\leq \frac{\gamma}{2}.
$$
We now exhibit $\rho_1, \rho_2>0$ such that
$$
\begin{cases}
\int_{B}\vert f-f_0\vert\geq \rho_1 \Rightarrow \Pb_{Q_f^n}(S_B< t_1)\leq \gamma/2\\
\int_{\bar{B}}\vert f-f_0\vert\geq \rho_2 \Rightarrow \Pb_{Q_f^n}(T_B< t_2)\leq \gamma/2.
\end{cases}
$$
In this case, for all $f\in H(\beta,L)$ satisfying $\Vert f-f_0\Vert_1\geq \rho_1+\rho_2$ it holds
$$
\Pb_{Q_{f_0}^n}(\Phi=1)+\Pb_{Q_f^n}(\Phi=0)\leq \frac{\gamma}{2}+ \min\left\{ \Pb_{Q_f^n}( S_B< t_1),  \Pb_{Q_f^n}(T_B< t_2)\right\}\leq \frac{\gamma}{2}+\frac{\gamma}{2}=\gamma,
$$
since $\int_B\vert f-f_0\vert+\int_{\bar{B}}\vert f-f_0\vert=\Vert f-f_0\Vert_1\geq \rho_1+\rho_2$ implies $\int_{B}\vert f-f_0\vert\geq \rho_1$ or $ \int_{\bar{B}}\vert f-f_0\vert\geq \rho_2$.
Consequently, $\rho_1+\rho_2$ will provide an upper bound on $\Ec_{n,\alpha}^{\text{NI}}(f_0,\gamma)$.

If $\int_{\overline B} (f-f_0)=\Eb_{Q_{f}^n}[T_B]>t_2$ then Chebychev's inequality yields
\begin{align*}
\Pb_{Q_{f}^n}(T_B < t_2)&=\Pb_{Q_{f}^n}\left(\Eb_{Q_{f}^n}[T_B]-T_B > \Eb_{Q_{f}^n}[T_B]-t_2\right)\\
&\leq \Pb_{Q_{f}^n}\left(\left\vert \Eb_{Q_{f}^n}[T_B]- T_B \right\vert > \Eb_{Q_{f}^n}[ T_B]-t_2\right)\\
&\leq \frac{\Var_{Q_{f}^n}(T_B)}{\left(\Eb_{Q_{f}^n}[ T_B]-t_2\right)^2}\\
&\leq \frac{c_\alpha^2 }{n \left(\int_{\overline B} (f-f_0)-t_2 \right)^2}.
\end{align*}
Now, observe that
$$
\int_{\bar{B}}(f-f_0)\geq \int_{\bar{B}}\vert f-f_0\vert-2\int_{\bar{B}}f_0.
$$
Thus, setting
$$
\rho_2=2\int_{\bar{B}}f_0+\left( 1+\frac{1}{\sqrt{2}}\right)t_2,
$$
we obtain that $\int_{\bar{B}}\vert f-f_0\vert\geq \rho_2$ implies 
$$
\Pb_{Q_{f}^n}(T_B < t_2)\leq \frac{2c_\alpha^2}{n t_2^2}\leq \frac{10}{n\alpha^2 t_2^2}=\frac{\gamma}{2}.
$$
We now exhibit $\rho_1$ such that $\int_B\vert f-f_0\vert\geq \rho_1$ implies $\Pb_{Q_{f}^n}(S_B < t_1)\leq \gamma/2.$
First note that if the following relation holds
\begin{equation}\label{Proof NI Upper bound Relation 1}
\Eb_{Q_{f}^n}[S_B]= \sum_{j=1}^N\left\vert[\psi_h\ast f](x_j)-f_0(x_j) \right\vert^2 \geq t_1 + \sqrt{\frac{2\Var_{Q_{f}^n}(S_B)}{\gamma}},   
\end{equation}
then Chebychev's inequality yields
$$
\Pb_{Q_{f}^n}(S_B < t_1)\leq \Pb_{Q_{f}^n}\left( S_B\leq \Eb_{Q_{f}^n}[S_B]-\sqrt{\frac{2\Var_{Q_{f}^n}(S_B)}{\gamma}} \right)\leq \frac{\gamma}{2}.
$$
Using $\sqrt{a+b}\leq \sqrt{a}+\sqrt{b}$ for all $a,b>0$ and $ab\leq a^2/2+b^2/2$  we have
\begin{align*}
\sqrt{\frac{2\Var_{Q_{f}^n}(S_B)}{\gamma}}&\leq \sqrt{\frac{72\Vert \psi\Vert_\infty^2}{\gamma n \alpha^2 h^2}\sum_{j=1}^N\left([\psi_h\ast f](x_j)-f_0(x_j) \right)^2+\frac{328\Vert \psi\Vert_\infty^4N}{\gamma n(n-1)\alpha^4h^4}}\\
&\leq \sqrt{\frac{72\Vert \psi\Vert_\infty^2}{\gamma n \alpha^2 h^2}\sum_{j=1}^N\left([\psi_h\ast f](x_j)-f_0(x_j) \right)^2}+\sqrt{\frac{656\Vert \psi\Vert_\infty^4N}{\gamma n^2\alpha^4h^4}}\\
&\leq \frac{1}{2}\sum_{j=1}^N\left([\psi_h\ast f](x_j)-f_0(x_j) \right)^2+\frac{36\Vert \psi\Vert_\infty^2}{\gamma n \alpha^2 h^2}+\frac{26\Vert \psi\Vert_\infty^2\sqrt{N}}{\sqrt{\gamma}n\alpha^2h^2}\\
&\leq \frac{1}{2}\sum_{j=1}^N\left([\psi_h\ast f](x_j)-f_0(x_j) \right)^2+\frac{62\Vert \psi\Vert_\infty^2\sqrt{N}}{\gamma n\alpha^2h^2}.
\end{align*}
Thus, if 
\begin{equation}\label{Proof NI Upper bound Relation 2}
\sum_{j=1}^N\left\vert[\psi_h\ast f](x_j)-f_0(x_j) \right\vert^2 \geq 2\left[t_1 + \frac{62\Vert \psi\Vert_\infty^2\sqrt{N}}{\gamma n\alpha^2h^2}\right]
\end{equation}
then \eqref{Proof NI Upper bound Relation 1} holds and we have $\Pb_{Q_{f}^n}(S_B < t_1)\leq \gamma/2$.
We now link $\sum_{j=1}^N\left\vert[\psi_h\ast f](x_j)-f_0(x_j) \right\vert^2$ to $\int_B\vert f-f_0\vert$.
According to Cauchy-Schwarz inequality we have
$$
\left( \sum_{j=1}^N\left\vert[\psi_h\ast f](x_j)-f_0(x_j) \right\vert \right)^2 \leq N \sum_{j=1}^N\left\vert[\psi_h\ast f](x_j)-f_0(x_j) \right\vert^2.
$$
We also have
\begin{align*}
\left\vert \int_B \vert f-f_0\vert-\sum_{j=1}^N2h\vert \psi_h\ast f(x_j)-f_0(x_j)\vert  \right\vert &=
\left\vert \sum_{j=1}^N \int_{B_j} \vert f-f_0\vert-\sum_{j=1}^N2h\vert \psi_h\ast f(x_j)-f_0(x_j)\vert  \right\vert \\
&= \left\vert \sum_{j=1}^N \int_{B_j}\left( \vert f(x)-f_0(x)\vert-\vert \psi_h\ast f(x_j)-f_0(x_j)\vert \right)  \dd x\right\vert \\
&\leq \sum_{j=1}^N \int_{B_j}\left\vert  f(x)-f_0(x)- \psi_h\ast f(x_j)+f_0(x_j)  \right\vert \dd x\\
&\leq \sum_{j=1}^N \int_{B_j}\left(\vert f(x)-f(x_j)\vert +\vert f(x_j) - \psi_h\ast f(x_j)\vert+ \vert f_0(x_j)-f_0(x) \vert \right) \dd x\\
&\leq \left[1 + C_\beta  +\frac{L_0}{L} \right]L h^\beta \vert B\vert.
\end{align*}
We thus have 
$$
\sum_{j=1}^N\left\vert[\psi_h\ast f](x_j)-f_0(x_j) \right\vert^2\geq \frac{1}{4Nh^2}\left( \int_B \vert f-f_0\vert- \left[1 + C_\beta  +\frac{L_0}{L} \right]L h^\beta \vert B\vert\right)^2.
$$
Thus, if 
$$
\int_B \vert f-f_0\vert\geq \left[1 + C_\beta  +\frac{L_0}{L} \right]L h^\beta \vert B\vert + 2h\sqrt{N}\sqrt{ 2t_1 + \frac{124\Vert \psi\Vert_\infty^2\sqrt{N}}{\gamma n\alpha^2h^2} } =:\rho_1
$$
then \eqref{Proof NI Upper bound Relation 2} holds and we have $\Pb_{Q_{f}^n}(S_B < t_1)\leq \gamma/2$.
Consequently
\begin{align*}
    \Ec_{n,\alpha}^{\text{NI}}(f_0,\gamma) & \leq \rho_1+\rho_2 \\
    &\leq \left[1 + C_\beta  +\frac{L_0}{L} \right]L h^\beta \vert B\vert + 2h\sqrt{N}\sqrt{ 2t_1 + \frac{124\Vert \psi\Vert_\infty^2\sqrt{N}}{\gamma n\alpha^2h^2} }+ 2\int_{\bar{B}}f_0+\left( 1+\frac{1}{\sqrt{2}}\right)t_2\\
    & \leq C(L,L_0,\beta,\gamma,\psi)\left[ h^\beta \vert B\vert + Nh^{\beta+1}+ \frac{N^{3/4}}{\sqrt{n\alpha^2}} + \int_{\bar{B}}f_0+ \frac{1}{\sqrt{n\alpha^2}}\right]\\
    &\leq C(L,L_0,\beta,\gamma,\psi)\left[ h^\beta \vert B\vert + \frac{\vert B\vert^{3/4}}{h^{3/4} \sqrt{n\alpha^2}} + \int_{\bar{B}}f_0 + \frac{1}{\sqrt{n\alpha^2}}\right]
\end{align*}
where we have used $\sqrt{a+b}\leq \sqrt{a}+\sqrt{b}$ for $a,b>0$ to obtain the second to last inequality.
Taking $h\asymp \vert B \vert^{-1/(4\beta+3)}(n\alpha^2)^{-2/(4\beta+3)}$ yields
$$
\Ec_{n,\alpha}^{\text{NI}}(f_0,\gamma)\leq C(L,L_0,\beta,\gamma,\psi)\left[\vert B \vert^{\frac{3\beta+3}{4\beta+3}}(n\alpha^2)^{-\frac{2\beta}{4\beta+3}} + \int_{\overline B} f_0+ \frac{1}{\sqrt{n\alpha^2}}\right].
$$

\subsection{Proof of Lemma~\ref{LemmaBSNONI}}
For $j=1,\ldots,N$, write 
$$
v_j=\sum_{k=1}^N a_{kj}\psi_k.
$$
Note that since $(\psi_1,\dots,\psi_N)$ and $(v_1,\ldots,v_N)$ are two orthonormal bases of $W_N$, the  matrix $(a_{kj})_{kj}$ is orthogonal.
We can write
$$
f_\nu(x)= f_0(x)+\delta\sum_{j=1}^N\sum_{k=1}^N\frac{\nu_ja_{kj}}{\tilde{\lambda}_j} \psi_k(x), \quad x\in\Rb.
$$
Define
$$
A_b=\left\{\nu\in \{-1,1\}^N : \left\vert \sum_{j=1}^N\frac{\nu_ja_{kj}}{\tilde{\lambda}_j} \right\vert\leq \frac{1}{\sqrt{h}}\sqrt{\log\left( \frac{2N}{b} \right)}\text{ for all } 1\leq k\leq N\right\}.
$$
The union bound and Hoeffding inequality yield
\begin{align*}
\Pb_\nu(A_b^c)&\leq \sum_{k=1}^N\Pb\left( \left\vert \sum_{j=1}^N\frac{\nu_ja_{kj}}{\tilde{\lambda}_j} \right\vert > \frac{1}{\sqrt{h}}\sqrt{\log\left( \frac{2N}{b} \right)} \right)\\
& \leq \sum_{k=1}^N 2\exp\left(-\frac{2\log\left( \frac{2N}{b} \right)}{h\sum_{j=1}^N4\frac{a_{kj}^2}{\tilde{\lambda}_j^2}}  \right)\\
&\leq b,
\end{align*}
 where the last inequality follows from $\tilde{\lambda}_j^2\geq 2h$ for all $j$ and $\sum_{j=1}^Na_{kj}^2=1$. 
We thus have $\Pb_\nu(A_b)\geq 1-b$.

We now prove $i)$. 
Since $\int \psi_k= 0$ for all $k=1,\ldots,n$, it holds $\int f_\nu=\int f_0=1$ for all $\nu$.
Since $\text{Supp}(\psi_k)=B_k$ for all $k=1,\ldots,N$, it holds $f_\nu \equiv f_0$ on $B^c$ and thus $f_\nu$ is non-negative on $B^c$.
Now, for $x\in B_k$ it holds
$$
f_\nu(x)= f_0(x)+\delta\sum_{j=1}^N\frac{\nu_ja_{kj}}{\tilde{\lambda}_j} \psi_k(x)\geq C_0(B)-\frac{\delta\Vert \psi\Vert_\infty}{\sqrt{h}}\left\vert \sum_{j=1}^N\frac{\nu_ja_{kj}}{\tilde{\lambda}_j} \right\vert.
$$ 
Moreover, for any $\nu\in A_b$, we have 
$$
\frac{\delta\Vert \psi\Vert_\infty}{\sqrt{h}}\left\vert \sum_{j=1}^N\frac{\nu_ja_{kj}}{\tilde{\lambda}_j} \right\vert\leq  \frac{\delta\Vert \psi\Vert_\infty}{h} \sqrt{\log\left( \frac{2N}{b} \right)}\leq C_0(B) 
$$
since $\delta$ is assumed to satisfy $\delta\leq \frac{h}{\sqrt{\log(2N/b)}}\min\left\{ \frac{C_0(B)}{\Vert\psi\Vert_\infty} , \frac{1}{2}\left(1-\frac{L_0}{L} \right)h^\beta\right\}$. 
Thus, $f_\nu$ is non-negative on $\Rb$ for all $\nu\in A_b$.

To prove $ii)$, we have to show that $\vert f_\nu(x)-f_\nu(y)\vert \leq L \vert x-y\vert^\beta$, for all $\nu\in A_b$, for all $x,y\in\Rb$.
Since $f_\nu\equiv f_0$ on $B^c$ and $f_0\in H(\beta,L_0)$, this result is trivial for $x,y\in B^c$.
If $x\in B_l$ and $y\in B_k$ it holds 
\begin{align*}
\vert f_\nu(x)-f_\nu(y)\vert &\leq \vert f_0(x)-f_0(y)\vert + \left\vert \delta\sum_{j=1}^N\frac{\nu_ja_{lj}}{\tilde{\lambda}_j} \psi_l(x) - \delta\sum_{j=1}^N\frac{\nu_ja_{kj}}{\tilde{\lambda}_j} \psi_k(y)  \right\vert\\
&\leq L_0\vert x-y\vert^\beta+ \left\vert \delta\sum_{j=1}^N\frac{\nu_ja_{lj}}{\tilde{\lambda}_j} \psi_l(x) - \delta\sum_{j=1}^N\frac{\nu_ja_{lj}}{\tilde{\lambda}_j} \psi_l(y)  \right\vert\\
&\hspace{4cm}+\left\vert \delta\sum_{j=1}^N\frac{\nu_ja_{kj}}{\tilde{\lambda}_j} \psi_k(x) - \delta\sum_{j=1}^N\frac{\nu_ja_{kj}}{\tilde{\lambda}_j} \psi_k(y)  \right\vert\\
&\leq L_0\vert x-y\vert^\beta+\frac{\delta}{\sqrt{h}} \left\vert \sum_{j=1}^N\frac{\nu_ja_{lj}}{\tilde{\lambda}_j}  \right\vert\left\vert \psi\left(\frac{x-x_l}{h}\right)-\psi\left(\frac{y-x_l}{h}\right)\right\vert\\
&\hspace{4cm}+\frac{\delta}{\sqrt{h}} \left\vert \sum_{j=1}^N\frac{\nu_ja_{kj}}{\tilde{\lambda}_j}  \right\vert\left\vert \psi\left(\frac{x-x_k}{h}\right)-\psi\left(\frac{y-x_k}{h}\right)\right\vert\\
&\leq L_0\vert x-y\vert^\beta + \frac{\delta}{h^{\beta+1/2}} \left\vert \sum_{j=1}^N\frac{\nu_ja_{lj}}{\tilde{\lambda}_j}\right\vert\cdot  L\vert x-y\vert^\beta+ \frac{\delta}{h^{\beta+1/2}} \left\vert \sum_{j=1}^N\frac{\nu_ja_{kj}}{\tilde{\lambda}_j}\right\vert\cdot  L\vert x-y\vert^\beta\\
&=\left(\frac{L_0}{L}+ \frac{\delta}{h^{\beta+1/2}} \left\vert \sum_{j=1}^N\frac{\nu_ja_{lj}}{\tilde{\lambda}_j}\right\vert+\frac{\delta}{h^{\beta+1/2}} \left\vert \sum_{j=1}^N\frac{\nu_ja_{kj}}{\tilde{\lambda}_j}\right\vert\right)L\vert x-y\vert^\beta,
\end{align*}
where we have used $\psi\in H(\beta,L)$.
Observe that for all $k=1,\ldots,n$ and for all $\nu\in A_b$ it holds
$$
\frac{\delta}{h^{\beta+1/2}} \left\vert \sum_{j=1}^N\frac{\nu_ja_{kj}}{\tilde{\lambda}_j}\right\vert\leq  \frac{\delta}{h^{\beta+1}}\cdot \sqrt{\log\left( \frac{2N}{b} \right)}\leq \frac{1}{2}\left(1-\frac{L_0}{L} \right),
$$
since $\delta$ is assumed to satisfy $\delta\leq \frac{h}{\sqrt{\log(2N/b)}}\min\left\{ \frac{C_0(B)}{\Vert\psi\Vert_\infty} , \frac{1}{2}\left(1-\frac{L_0}{L} \right)h^\beta\right\}$.
Thus,  it holds $\vert f_\nu (x)-f_\nu(y)\vert\leq L\vert x-y\vert^\beta$ for all $\nu\in A_b$, $x\in B_l$ and $y\in B_k$.
The case $x\in B^c$ and $y\in B_k$ can be handled in a similar way, which ends the proof of $ii)$.

We now prove $iii)$.
It holds
\begin{align*}
\int_\Rb\vert f_\nu-f_0\vert &=\int_\Rb\left\vert \delta\sum_{j=1}^N\frac{\nu_j}{\tilde{\lambda_j}}v_j(x) \right\vert\dd x = \delta\sum_{k=1}^N\int_{B_k}\left\vert \sum_{j=1}^N\frac{\nu_j}{\tilde{\lambda_j}}v_j(x) \right\vert\dd x\\
&= \delta\sum_{k=1}^N\int_{B_k}\left\vert \sum_{j=1}^N\frac{\nu_ja_{kj}}{\tilde{\lambda_j}}\psi_k(x) \right\vert\dd x\\
&= \delta\sum_{k=1}^N\left\vert \sum_{j=1}^N\frac{\nu_ja_{kj}}{\tilde{\lambda_j}} \right\vert\int_{B_k}\left\vert \psi_k(x) \right\vert\dd x\\
&= C_1\delta \sqrt{h}\sum_{k=1}^N\left\vert \sum_{j=1}^N\frac{\nu_ja_{kj}}{\tilde{\lambda_j}} \right\vert,
\end{align*}
where $C_1=\int_{-1}^1\vert \psi\vert$.
For all $\nu\in A_b$ it thus holds
$$
\int_\Rb\vert f_\nu-f_0\vert\geq  C_1\frac{\delta h}{\sqrt{\log\left( \frac{2N}{b} \right)}}\sum_{k=1}^N\left\vert \sum_{j=1}^N\frac{\nu_ja_{kj}}{\tilde{\lambda_j}} \right\vert^2.
$$
Moreover,
\begin{align*}
\sum_{k=1}^N\left\vert \sum_{j=1}^N\frac{\nu_ja_{kj}}{\tilde{\lambda_j}} \right\vert^2&= \sum_{k=1}^N\left( \sum_{j=1}^N\left(\frac{\nu_ja_{kj}}{\tilde{\lambda_j}}\right)^2+\sum_{j\neq l}\frac{\nu_ja_{kj}}{\tilde{\lambda_j}}\frac{\nu_la_{kl}}{\tilde{\lambda_l}}  \right)\\
&=\sum_{j=1}^N\frac{1}{\tilde{\lambda}_j^2}\sum_{k=1}^Na_{kj}^2+\sum_{j\neq l}\frac{\nu_j\nu_l}{\tilde{\lambda}_j\tilde{\lambda}_l}\sum_{k=1}^Na_{kj}a_{kl}\\
&=\sum_{j=1}^N\frac{1}{\tilde{\lambda}_j^2},
\end{align*}
since the matrix $(a_{kj})_{k,j}$ is orthogonal.
Thus, for all $\nu\in A_b$ it holds 
$$
\Vert f_\nu-f_0\Vert_1\geq  C_1\frac{\delta h}{\sqrt{\log\left( \frac{2N}{b} \right)}}\sum_{j=1}^N\frac{1}{\tilde{\lambda}_j^2}.
$$
Set $\Jc=\{j\in\llbr 1,N\rrbr : z_\alpha^{-1}\lambda_j\geq \sqrt{2h} \}$, we have for all $\nu\in A_b$
\begin{align*}
\Vert f_\nu-f_0\Vert_1&\geq C_1\frac{\delta h}{\sqrt{\log\left( \frac{2N}{b} \right)}} \sum_{j=1}^N\left(\frac{1}{2h} I(z_\alpha^{-1}\lambda_j<\sqrt{2h})+\frac{z_\alpha^2}{\lambda_j^2}I(z_\alpha^{-1}\lambda_j\geq \sqrt{2h}) \right)\\
&=C_1\frac{\delta h}{\sqrt{\log\left( \frac{2N}{b} \right)}}\left(\frac{1}{2h} (N-\vert \Jc\vert) +\sum_{j\in\Jc}\frac{z_\alpha^2}{\lambda_j^2}\right) \\
&\geq C_1\frac{\delta h}{\sqrt{\log\left( \frac{2N}{b} \right)}} \left(\frac{N}{2h}-\frac{\vert \Jc\vert}{2h}+z_\alpha^2\vert \Jc\vert^2\left(\sum_{j\in\Jc}\lambda_j^2\right)^{-1}\right)  \\
&= C_1\frac{\delta N}{2\sqrt{\log\left( \frac{2N}{b} \right)}}\left( 1-\frac{\vert \Jc\vert}{N}+\left(\frac{\vert \Jc\vert}{N}\right)^2 \vert B\vert z_\alpha^2\left(\sum_{j\in\Jc}\lambda_j^2\right)^{-1} \right),
\end{align*}
where the second to last inequality follows from the inequality between harmonic and artithmetic means.
Now,
\begin{align*}
\sum_{j\in\Jc}\lambda_j^2\leq \sum_{j=1}^N\lambda_j^2&=\sum_{j=1}^N \langle Kv_j,v_j\rangle\\
&=\sum_{j=1}^N\left\langle \frac{1}{n}\sum_{i=1}^n\int_\Rb\left(\int_{\Zc_i} \frac{q_i(z_i\mid y)q_i(z_i\mid \cdot)\1_B(y)\1_B(\cdot)}{g_{0,i}(z_i)}\dd \mu_i(z_i) \right)v_j(y)\dd y , v_j \right\rangle \\
&=\frac{1}{n}\sum_{i=1}^n\int_{\Zc_i}\sum_{j=1}^N\left(\int_\Rb \int_\Rb\frac{q_i(z_i\mid y)q_i(z_i\mid x)\1_B(y)\1_B(x)}{g_{0,i}(z_i)}v_j(x)v_j(y)\dd x\dd y\right)\dd \mu_i(z_i) \\
&=\frac{1}{n}\sum_{i=1}^n\int_{\Zc_i}\sum_{j=1}^N\left(\int_\Rb \frac{q_i(z_i\mid x)\1_B(x)}{g_{0,i}(z_i)}v_j(x)\dd x\right)^2g_{0,i}(z_i)\dd \mu_i(z_i) \\
&=\frac{1}{n}\sum_{i=1}^n\int_{\Zc_i}\sum_{j=1}^N\left(\int_\Rb \left(\frac{q_i(z_i\mid x)}{g_{0,i}(z_i)}-e^{-2\alpha}\right)\1_B(x)v_j(x)\dd x\right)^2g_{0,i}(z_i)\dd \mu_i(z_i),
\end{align*}
since $\int \1_B(x) v_j(x)dx=0$.
Recall that  $q_i$ satisfies $e^{-\alpha }\leq q_i(z_i\mid x)\leq e^\alpha$ for all $z_i\in\Zc_i$ and all $x\in \Rb$.
This implies $e^{-\alpha}\leq g_{0,i}(z_i)\leq e^\alpha$, and therefore $0\leq f_{i,z_i}(x):=\frac{q_i(z_i\mid x)}{g_{0,i}(z_i)}-e^{-2\alpha}\leq z_\alpha$.
Writing $f_{i,z_i,B}=\1_B\cdot f_{i,z_i}$, we have
\begin{align*}
\sum_{j=1}^N\left(\int_\Rb \left(\frac{q_i(z_i\mid x)}{g_{0,i}(z_i)}-e^{-2\alpha}\right)\1_B(x) v_j(x)\dd x\right)^2&=\sum_{j=1}^N \langle f_{i,z_i,B}, v_j\rangle^2=\left\Vert \sum_{j=1}^N \langle f_{i,z_i,B}, v_j\rangle v_j \right\Vert_2^2  \\
&=\left\Vert \text{Proj}_{\text{Vect}(v_1,\ldots,v_N)}( f_{i,z_i,B}) \right\Vert_2^2 \\
&\leq \left\Vert  f_{i,z_i,B}\right\Vert_2^2\leq z_\alpha^2\vert B\vert.
\end{align*}
Moreover, $\int_{\Zc_i} g_{0,i}(z_i)\dd \mu_i(z_i)=\int_\Rb(\int_{\Zc_i}q_i(z_i\mid x)\dd \mu_i(z_i))f_0(x)\dd x=\int_\Rb f_0=1$.
This gives $\sum_{j\in\Jc}\lambda_j^2\leq z_\alpha^2 \vert B\vert$ and for all $\nu\in A_b$
$$
\Vert f_\nu-f_0\Vert_1\geq  C_1\frac{\delta N}{2\sqrt{\log\left( \frac{2N}{b} \right)}}\left( 1-\frac{\vert \Jc\vert}{N}+\left(\frac{\vert \Jc\vert}{N}\right)^2 \right)\geq \frac{3C_1}{8}\frac{\delta  N }{\sqrt{\log\left( \frac{2N}{b} \right)}}.
$$


\section{Proofs of Section \ref{Section Intercative scenario}}

\subsection{Proof of Proposition \ref{Prop Privacy Interactive mechanism}}\label{App Proof privacy interactive mechanism}

Let $i\in\llbr 1,n\rrbr $.
Since $Z_i$ depends only on $X_i$, condition \eqref{eq alphaLDPconstraint} reduces to 
\begin{equation}\label{Eq First part proof privacy Intercative mechanism}
 \frac{q^{Z_i\mid X_i=y}(z)}{q^{Z_i\mid X_i=y'}(z)}\leq e^\alpha, \quad \forall y,y'\in\Rb, \, \forall z\in\Rb^N,   
\end{equation}
where $q^{Z_i\mid X_i=y}$ denotes the conditional density of $Z_i$ given $X_i=y$.
It holds
$$
q^{Z_i\mid X_i=y}(z)=\prod_{j=1}^N \frac{\alpha}{4}\exp\left(-\frac{\alpha \vert z_j-I(y\in B_j)\vert}{2}  \right).
$$
Thus, by the reverse and the ordinary triangle inequality,
\begin{align*}
\frac{q^{Z_i\mid X_i=y}(z)}{q^{Z_i\mid X_i=y'}(z)}&=\prod_{j=1}^N \exp\left(\frac{\alpha\left[\vert z_j-I(y'\in B_j)\vert-\vert z_j-I(y\in B_j)\vert\right]}{2}  \right)\\
&\leq \prod_{j=1}^N \exp\left( \frac{\alpha\vert I(y\in B_j)-I(y'\in B_j)\vert}{2}  \right)\\
&= \exp\left(\frac{\alpha}{2}\sum_{j=1}^{N}\vert I(y\in B_j)-I(y'\in B_j)\vert \right)\\
&\leq \exp(\alpha),
\end{align*}
which proves \eqref{Eq First part proof privacy Intercative mechanism}.\\
Consider now $i\in \llbr n+1,2n\rrbr $. Since $Z_i$ depends only on $X_i$ and on $Z_1,\ldots, Z_n$, condition \eqref{eq alphaLDPconstraint} reduces for $i\in\llbr n+1, 2n\rrbr $ to 
\begin{equation}\label{Eq Second part proof privacy Intercative mechanism}
   \frac{\Pb\left(Z_i=z  \mid X_i\in A , Z_1=z_1,\ldots, Z_n=z_n \right)}{\Pb\left(Z_i=z \mid X_i\in F , Z_1=z_1,\ldots, Z_n=z_n \right)}\in [e^{-\alpha},e^\alpha] 
\end{equation}
for all $z\in\{-c_\alpha\tau,c_\alpha \tau\}$, $A,F\in \{\overline{B},B_1,\ldots,B_N \}$ and $z_1,\ldots,z_n\in \Rb^N$.
For all $j,k\in\llbr 1,N\rrbr$, for all $z_1,\ldots,z_n$ it holds
$$
\frac{\Pb\left(Z_i=c_\alpha\tau  \mid X_i\in B_j , Z_1=z_1,\ldots, Z_n=z_n \right)}{\Pb\left(Z_i=c_\alpha \tau \mid X_i\in B_k , Z_1=z_1,\ldots, Z_n=z_n \right)}=\frac{1+\frac{[\hat{p}_j-p_0(j)]_{-\tau}^\tau}{c_\alpha \tau}}{1+\frac{[\hat{p}_k-p_0(k)]_{-\tau}^\tau}{c_\alpha \tau}}\in \left[\frac{c_\alpha-1}{c_\alpha +1}, \frac{c_\alpha+1}{c_\alpha -1} \right]=[e^{-\alpha},e^{\alpha}],
$$
and a similar result holds for $z=-c_\alpha \tau$.
For all $j\in\llbr 1,N\rrbr$, for all $z_1,\ldots,z_n$ it holds
$$
\frac{\Pb\left(Z_i=c_\alpha\tau  \mid X_i\in B_j , Z_1=z_1,\ldots, Z_n=z_n \right)}{\Pb\left(Z_i=c_\alpha \tau \mid X_i\in \overline{B} , Z_1=z_1,\ldots, Z_n=z_n \right)}=1+\frac{[\hat{p}_j-p_0(j)]_{-\tau}^\tau}{c_\alpha \tau}\in \left[1-\frac{1}{c_\alpha},1+\frac{1}{c_\alpha} \right]\subset[e^{-\alpha},e^{\alpha}],
$$
and a similar result holds for $z=-c_\alpha \tau$.
This ends the proof of \eqref{Eq Second part proof privacy Intercative mechanism}.\\
Consider now $i\in\llbr 2n+1, 3n\rrbr$. Since $Z_i$ depends only on $X_i$, condition \eqref{eq alphaLDPconstraint} reduces for $i\in\llbr 2n+1, 3n\rrbr $ to 
\begin{equation*}
   \frac{\Pb\left(Z_i=z  \mid X_i\in A \right)}{\Pb\left(Z_i=z \mid X_i\in F \right)}\in [e^{-\alpha},e^\alpha] , \quad \forall A,F\in \{\overline{B},B_1,\ldots,B_N\}, \, \forall z\in\{-c_\alpha,c_\alpha\}.
\end{equation*}
We have already proved this in the proof of Proposition \ref{Prop Privacy of the NI Mechanism}.

\subsection{Analysis of the mean and variance of the statistic $D_B$} \label{App Analysis D_B}

\begin{proof}[Proof of Proposition~\ref{Prop Test Stat Int}]
1. For all $i\in\llbr n+1,2n\rrbr$ it holds
\begin{align*}
&\Pb\left(Z_i=\pm c_\alpha\tau \mid Z_1,\ldots, Z_n \right)\\
&\hspace{1cm}=\sum_{j=1}^N\Pb\left(Z_i= \pm c_\alpha\tau \mid X_i\in B_j \right)\Pb(X_i\in B_j)+\Pb\left(Z_i= \pm c_\alpha\tau \mid X_i\in \bar{B} \right)\Pb(X_i\in \bar{B})\\
&\hspace{1cm}=\sum_{j=1}^N\frac{1}{2}\left( 1 \pm \frac{[\widehat{p}_j-p_0(j)]_{-\tau}^\tau}{c_\alpha\tau} \right)p(j)+\frac{1}{2}\int_{\bar{B}}f.
\end{align*}
For $i\in\llbr n+1,2n\rrbr$ we thus have 
\begin{align*}
  \Eb[Z_i\mid Z_1,\ldots,Z_n]&=c_\alpha\tau \Pb(Z_i=c_\alpha\tau\mid Z_1,\ldots,Z_n)-c_\alpha\tau \Pb(Z_i=-c_\alpha\tau\mid Z_1,\ldots,Z_n) \\
  & =\sum_{j=1}^Np(j)[\widehat{p}_j-p_0(j)]_{-\tau}^\tau.
\end{align*}
Thus,
$$
\Eb[D_B]=\Eb\left[ \Eb[D_B\mid Z_1,\ldots, Z_n] \right]=\sum_{j=1}^N\{p(j)-p_0(j)\}\Eb\left[[\widehat{p}_j-p_0(j)]_{-\tau}^\tau\right].
$$
The proof of \eqref{Eq Lower bound on the Mean of the Interactive test statistic} is similar to the proof of Theorem 3 in \cite{Berrett_Butucea_2020_DiscreteDistribTesting}.

2. Write
$$
\Var(D_B)=\Eb\left[\Var\left(D_B \mid Z_1,\ldots,Z_n \right) \right] + \Var\left(\Eb\left[D_B \mid Z_1,\ldots, Z_n \right] \right).
$$
It holds
$$
\Eb\left[D_B \mid Z_1,\ldots, Z_n \right]=\sum_{j=1}^N\{p(j)-p_0(j)\}[\widehat{p}_j-p_0(j)]_{-\tau}^\tau,
$$
and
\begin{align*}
   \Var\left(D_B \mid Z_1,\ldots,Z_n \right)&=\Var\left(\frac{1}{n}\sum_{i=n+1}^{2n}Z_i-\sum_{j=1}^Np_0(j)[\widehat{p}_j-p_0(j)]_{-\tau}^\tau \mid Z_1,\ldots,Z_n \right) \\
   &=\Var\left(\frac{1}{n}\sum_{i=n+1}^{2n}Z_i \mid Z_1,\ldots,Z_n \right)\\
   &=\frac{1}{n^2}\sum_{i=n+1}^{2n}\Var\left(Z_i \mid Z_1,\ldots,Z_n \right)\\
   &\leq \frac{1}{n^2}\sum_{i=n+1}^{2n}\Eb\left[Z_i^2 \mid Z_1,\ldots,Z_n \right]\\
   &\leq \frac{c_\alpha^2\tau^2}{n},
\end{align*}
where we have used the independence of the random variables $(Z_i)_{i=n+1,\ldots,2n}$ conditionnally on $Z_1,\ldots,Z_n$.
This gives
\begin{align*}
    \Var(D_B)&\leq \frac{c_\alpha^2\tau^2}{n}+ \sum_{j=1}^N\{p(j)-p_0(j)\}^2\Var\left([\widehat{p}_j-p_0(j)]_{-\tau}^\tau\right)\\
    &+\sum_{j_1\neq j_2}\{p(j_1)-p_0(j_1)\}\{p(j_2)-p_0(j_2)\}\Cov([\widehat{p}_{j_1}-p_0(j_1)]_{-\tau}^\tau,[\widehat{p}_{j_2}-p_0(j_2)]_{-\tau}^\tau).
\end{align*}
Set $P_j=[\widehat{p}_j-p_0(j)]_{-\tau}^\tau$.
We will prove that 
\begin{equation}\label{Proof Var of D_B preliminary result 1}
    \Var(P_j)\leq \frac{10}{n\alpha^2}\exp\left(-\frac{n\alpha^2(p(j)-p_0(j))^2}{168} \right), \quad \forall j\in\llbr 1,N\rrbr, 
\end{equation}
and
\begin{equation}\label{Proof Var of D_B preliminary result 2}
    \left\vert \Cov(P_{j_1},P_{j_2})\right\vert \leq \frac{2p(j_1)p(j_2)}{n}\exp\left(-\frac{n\alpha^2\left[(p(j_1)-p_0(j_1))^2+(p(j_2)-p_0(j_2))^2\right]}{336} \right)
\end{equation}
for all $j_1,j_2\in \llbr 1, N\rrbr$, $j_1\neq j_2$.
We admit these results for the moment and finish the proof of Proposition~\ref{Prop Test Stat Int}.
Using \eqref{Proof Var of D_B preliminary result 1} and \eqref{Proof Var of D_B preliminary result 2} we obtain
\begin{align*}
    \Var(D_B)&\leq \frac{c_\alpha^2\tau^2}{n}+\frac{10}{n\alpha^2}\sum_{j=1}^N\{p(j)-p_0(j)\}^2\exp\left(-\frac{n\alpha^2(p(j)-p_0(j))^2}{168} \right)\\
    &\hspace{1.5cm}+\frac{2}{n}\left[\sum_{j=1}^N\vert p(j)-p_0(j)\vert p(j)\exp\left(-\frac{n\alpha^2(p(j)-p_0(j))^2}{336} \right)\right]^2\\
    &\leq \frac{c_\alpha^2\tau^2}{n}+\frac{10}{n\alpha^2}\sum_{j=1}^N\{p(j)-p_0(j)\}^2\exp\left(-\frac{n\alpha^2(p(j)-p_0(j))^2}{168} \right)\\
    &\hspace{1.5cm}+\frac{2}{n}\left[\sum_{j=1}^Np(j)^2 \right]\left[\sum_{j=1}^N\vert p(j)-p_0(j)\vert^2 \exp\left(-\frac{n\alpha^2(p(j)-p_0(j))^2}{168} \right)\right]\\
    &\leq \frac{c_\alpha^2\tau^2}{n}+\frac{12}{n\alpha^2}\sum_{j=1}^N\vert p(j)-p_0(j)\vert^2\exp\left(-\frac{n\alpha^2(p(j)-p_0(j))^2}{168} \right),
\end{align*}
where the second to last inequality follows from Cauchy Schwarz inequality.
Now, observe that if $a_j:=\vert p(j)-p_0(j) \vert \neq 0$, then we can write 
$$
\vert p(j)-p_0(j)\vert\exp\left(-\frac{n\alpha^2(p(j)-p_0(j))^2}{168} \right)=\min\{\tau, a_j\}\cdot \frac{a_j/\tau}{\min\{1, a_j/\tau\}}\exp\left( -\frac{1}{168}\left(\frac{a_j}{\tau} \right)^2\right),
$$
where we recall that $\tau=1/\sqrt{n\alpha^2}$.
The study of the function $g:x\mapsto [x/\min\{ 1,x\}]\exp(-x^2/168)$ gives $g(x)\leq \sqrt{84}e^{-1/2}$ for all $x\geq 0$.
We thus have 
$$
\Var(D_B)\leq \frac{c_\alpha^2\tau^2}{n}+\frac{12e^{-1/2}\sqrt{84}}{n\alpha^2}\sum_{j=1}^N\vert p(j)-p_0(j)\vert\min\left\{\tau, \vert p(j)-p_0(j)\vert \right\}.
$$
Using that $\alpha^2c_\alpha^2\leq 5$ for all $\alpha\in(0,1)$, we finally obtain the claim of Proposition~\ref{Prop Test Stat Int},
$$
\Var(D_B)\leq \frac{5}{(n\alpha^2)^2}+\frac{67}{n\alpha^2}D_\tau(f).
$$
It remains now to prove \eqref{Proof Var of D_B preliminary result 1} and \eqref{Proof Var of D_B preliminary result 2}.
We will use the following concentration inequality which is an application of Bernstein's inequality (see for instance Corollary 2.11 in \cite{boucheron2013concentration})
\begin{equation}\label{Eq Concentration inequality proof variance of the stat in the interactive scenario}
\Pb\left(\vert \widehat{p}_j-p(j)\vert \geq x\right)\leq 2\exp\left( -\frac{n\alpha^2x^2}{42}\right), \quad \text{for all } 0<x\leq \frac{1}{\alpha}.
\end{equation}
Let us prove \eqref{Proof Var of D_B preliminary result 1}.
Let $j\in\llbr 1,N\rrbr$.
We first deal with the case where $p(j)-p_0(j)\geq 2\tau $. 
We have
\begin{align*}
    \Var\left([\widehat{p}_j-p_0(j)]_{-\tau}^{\tau} \right)&= \Var\left([\widehat{p}_j-p_0(j)]_{-\tau}^{\tau}-\tau \right)\\
    &\leq \Eb\left[\left( [\widehat{p}_j-p_0(j)]_{-\tau}^{\tau}-\tau\right)^2 \right]\\
    &=\Eb\left[(-2\tau)^2\1\left(\widehat{p}_j-p_0(j)\leq -\tau \right) + (\widehat{p}_j-p_0(j)-\tau)^2\1\left(\widehat{p}_j-p_0(j)\in[-\tau,\tau] \right)\right]\\
    &\leq 4\tau^2 \Pb\left( \widehat{p}_j-p_0(j)\leq \tau \right)\\
    &=4\tau^2 \Pb\left(p(j)-\widehat{p}_j\geq p(j)-p_0(j)-\tau \right) \\
    &\leq 4\tau^2 \Pb\left(\vert p(j)-\widehat{p}_j\vert \geq p(j)-p_0(j)-\tau \right)
\end{align*}
Now, if $p(j)-p_0(j)\geq 2\tau $ then we have $0<p(j)-p_0(j)-\tau\leq p(j)\leq 1\leq 1/\alpha$ and \eqref{Eq Concentration inequality proof variance of the stat in the interactive scenario} gives
\begin{align*}
    \Var\left([\widehat{p}_j-p_0(j)]_{-\tau}^{\tau} \right)& \leq  8\tau^2\exp\left(-\frac{n\alpha^2\left\{p(j)-p_0(j)-\tau \right\}^2}{42} \right)\\
    &\leq \frac{8}{n\alpha^2}\exp\left(-\frac{n\alpha^2\left\{p(j)-p_0(j)\right\}^2}{168} \right),
\end{align*}
which ends the proof of \eqref{Proof Var of D_B preliminary result 1} for the elements $j\in\llbr 1,N\rrbr$ such that $p(j)-p_0(j)\geq 2\tau$.
Starting from $\Var\left([\widehat{p}_j-p_0(j)]_{-\tau}^{\tau} \right)= \Var\left([\widehat{p}_j-p_0(j)]_{-\tau}^{\tau}+\tau \right)$, a similar proof gives \eqref{Proof Var of D_B preliminary result 1} for the elements $j\in\llbr 1,N\rrbr$ such that $p(j)-p_0(j)\leq -2\tau$.
It remains to deal with the case $\vert p(j)-p_0(j) \vert <2\tau$.
In this case, using that $[\cdot]_{-\tau}^{\tau}$ is Lipschitz continuous with Lipschitz constant $1$ we have
\begin{align*}
  \Var\left([\widehat{p}_j-p_0(j)]_{-\tau}^{\tau} \right)&=  \Var\left([\widehat{p}_j-p_0(j)]_{-\tau}^{\tau}- [p_j-p_0(j)]_{-\tau}^{\tau}\right)\\
  &\leq \Eb\left[ \left( [\widehat{p}_j-p_0(j)]_{-\tau}^{\tau}- [p_j-p_0(j)]_{-\tau}^{\tau} \right)^2 \right]\\
  &\leq \Eb\left[ \vert \widehat{p}_j-p(j)\vert^2 \right]\\
  &=\Var(\widehat{p}_j)\\
  &=\frac{1}{n^2}\sum_{i=1}^n\Var\left(I(X_i\in B_j)\right)+\frac{4}{n^2\alpha^2}\sum_{i=1}^n\Var(W_{ij})\\
  &\leq \frac{9}{n\alpha^2}\\
  &=\frac{9}{n\alpha^2}\exp\left(\frac{n\alpha^2\left\{p(j)-p_0(j)\right\}^2}{168}\right)\exp\left(-\frac{n\alpha^2\left\{p(j)-p_0(j)\right\}^2}{168}\right)\\
  &\leq \frac{9\exp(1/42)}{n\alpha^2}\exp\left(-\frac{n\alpha^2\left\{p(j)-p_0(j)\right\}^2}{168}\right),
\end{align*}
where the last inequality follows from the assumption $\vert p(j)-p_0(j)\vert \leq 2\tau=2/\sqrt{n\alpha^2}$.
This ends the proof of \eqref{Proof Var of D_B preliminary result 1}.
We now prove \eqref{Proof Var of D_B preliminary result 2}.
For all $i\in \llbr 1,n+1\rrbr$, we will write
\begin{align*}
    &\Eb_i\left[\cdot \right]=\Eb\left[\cdot \mid X_1,\ldots,X_{i-1} \right],\\
    &\Eb_i^j\left[\cdot \right]=\frac{1}{p(j)}\Eb\left[\cdot \1(X_i\in B_j) \mid X_1,\ldots,X_{i-1} \right],\\
    &\Eb_i^{comp}\left[\cdot \right]=\frac{1}{p\left(\overline{B}\right)}\Eb\left[\cdot \1(X_i\in\overline{B}) \mid X_1,\ldots,X_{i-1} \right].
\end{align*}
Observe that 
\begin{equation}\label{Proof Var Inter eq a.s.1}
    \Eb_i^j\left[P_{j_1} \right] \overset{a.s.}{=} \Eb_i^{j_2}\left[P_{j_1} \right], \quad \forall j, j_2\neq j_1,
\end{equation}
and
\begin{equation}\label{Proof Var Inter eq a.s.2}
    \Eb_i^{comp}\left[P_{j_1} \right] \overset{a.s.}{=} \Eb_i^{j_2}\left[P_{j_1} \right], \quad \forall  j_2\neq j_1,
\end{equation}
where we recall that $P_j=[\widehat{p}_j-p_0(j)]_{-\tau}^\tau$.
Let $j_1,j_2\in\llbr 1,N\rrbr$, $j_1\neq j_2$.
We have
\begin{align*}
    \Cov\left(P_{j_1},P_{j_2} \right)&=\Cov\left(\Eb_{n+1}\left[P_{j_1}\right],\Eb_{n+1}\left[P_{j_2}\right] \right)\\
    &=\Eb\left[\Eb_{n+1}\left[P_{j_1}\right] \Eb_{n+1}\left[P_{j_2}\right]\right]-\Eb\left[P_{j_1} \right]\Eb\left[P_{j_2} \right]\\
    &=\Eb\left[\sum_{i=1}^n\left(\Eb_{i+1}\left[P_{j_1}\right] \Eb_{i+1}\left[P_{j_2}\right]-\Eb_{i}\left[P_{j_1}\right] \Eb_{i}\left[P_{j_2}\right] \right) \right],
\end{align*}
where the sum in the last line is a telescoping sum.
We thus have 
\begin{equation}\label{Proof var Inter eq Cov 1}
  \Cov\left(P_{j_1},P_{j_2} \right)=\sum_{i=1}^n \Eb\left[ \Eb_{i+1}\left[P_{j_1}\right] \Eb_{i+1}\left[P_{j_2}\right]-\Eb_{i}\left[P_{j_1}\right] \Eb_{i}\left[P_{j_2}\right] \right] . 
\end{equation}
Now, it holds
\begin{align*}
   \Eb_{i}\left[P_{j_1}\right]&=\Eb\left[P_{j_1}\mid X_1,\ldots,X_{i-1} \right]\\ 
   &=\Eb\left[P_{j_1}\cdot\left(\sum_{j=1}^N \1(X_i\in B_j)+\1(X_i\in\overline{B}) \right)\mid X_1,\ldots,X_{i-1} \right]\\
   &=\sum_{j=1}^N p(j)\Eb_i^j\left[ P_{j_1}\right]+p\left(\overline{B} \right)\Eb_i^{comp}\left[ P_{j_1}\right]\\
   &=p(j_1)\Eb_i^{j_1}\left[ P_{j_1}\right]+\displaystyle\sum_{\substack{j=1\ \\j\neq j_1}}^N p(j)\Eb_i^{j_2}\left[ P_{j_1}\right]+p\left(\overline{B} \right)\Eb_i^{j_2}\left[ P_{j_1}\right],
\end{align*}
where the last equality follows from \eqref{Proof Var Inter eq a.s.1} and \eqref{Proof Var Inter eq a.s.2}.
We thus obtain
\begin{equation}\label{Proof var Inter eq Cov 2}
    \Eb_{i}\left[P_{j_1}\right]=p(j_1)\Eb_i^{j_1}\left[ P_{j_1}\right]+(1-p(j_1))\Eb_i^{j_2}\left[ P_{j_1}\right].
\end{equation}
Similarly, it holds
\begin{equation}\label{Proof var Inter eq Cov 3}
    \Eb_{i}\left[P_{j_2}\right]=p(j_2)\Eb_i^{j_2}\left[ P_{j_2}\right]+(1-p(j_2))\Eb_i^{j_1}\left[ P_{j_2}\right].
\end{equation}
We now compute $\Eb_{X_i}\left[\Eb_{i+1}\left[P_{j_1}\right] \Eb_{i+1}\left[P_{j_2}\right] \right]$.
We have
\begin{align*}
    &\Eb_{X_i}\left[\Eb_{i+1}\left[P_{j_1}\right] \Eb_{i+1}\left[P_{j_2}\right] \right]\\
    &=\int_{\Rb}f(y_i)\left[\int_{\Rb^{n-i}}P_{j_1}(X_1,\ldots,X_{i-1},y_i,y_{i+1},\ldots y_n)f(y_{i+1})\cdots f(y_n) \dd y_{i+1}\cdots \dd y_{n} \right.\\
    &\hspace{2cm}\left.\cdot\int_{\Rb^{n-i}}P_{j_2}(X_1,\ldots,X_{i-1},y_i,y'_{i+1},\ldots y'_n)f(y'_{i+1})\cdots f(y'_n) \dd y'_{i+1}\cdots \dd y'_{n} \right]\dd y_i\\
    &=\sum_{j=1}^N\int_{\Rb}f(y_i)\1(y_i\in B_j)\left[\int_{\Rb^{n-i}}P_{j_1}(X_1,\ldots,X_{i-1},y_i,y_{i+1},\ldots y_n)f(y_{i+1})\cdots f(y_n) \dd y_{i+1}\cdots \dd y_{n} \right.\\
    &\hspace{3.5cm}\left.\cdot\int_{\Rb^{n-i}}P_{j_2}(X_1,\ldots,X_{i-1},y_i,y'_{i+1},\ldots y'_n)f(y'_{i+1})\cdots f(y'_n) \dd y'_{i+1}\cdots \dd y'_{n} \right]\dd y_i\\
    &\hspace{1cm}+\int_{\Rb}f(y_i)\1(y_i\in \overline{B})\left[\int_{\Rb^{n-i}}P_{j_1}(X_1,\ldots,X_{i-1},y_i,y_{i+1},\ldots y_n)f(y_{i+1})\cdots f(y_n) \dd y_{i+1}\cdots \dd y_{n} \right.\\
    &\hspace{3.5cm}\left.\cdot\int_{\Rb^{n-i}}P_{j_2}(X_1,\ldots,X_{i-1},y_i,y'_{i+1},\ldots y'_n)f(y'_{i+1})\cdots f(y'_n) \dd y'_{i+1}\cdots \dd y'_{n} \right]\dd y_i
\end{align*}
For $j=1,\ldots,N$, let $x_j$ be such that $B_j=[x_j-h, x_j+h]$. 
Observe that if $y_i\in \mathring{B}_j $ then it holds $\1(y_i\in B_k)=\delta_{j,k}=\1(x_j\in B_k)$ where $\delta$ is the Kronecker delta.
Observe also that if $y_i\in \overline{B}$ then it holds $\1(y_i\in B_k)=0=\1(z\in B_k)$ for some $z\in\overline{B}$.
This gives
\begin{equation}\label{Proof var Inter eq Cov 4}
    P_k\left(X_1,\ldots,X_{i-1},y_i,y_{i+1},\ldots,y_n \right)\1(y_i\in\mathring{B}_j)=P_k\left(X_1,\ldots,X_{i-1},x_j,y_{i+1},\ldots,y_n \right)\1(y_i\in\mathring{B}_j),
\end{equation}
and
\begin{equation}\label{Proof var Inter eq Cov 5}
    P_k\left(X_1,\ldots,X_{i-1},y_i,y_{i+1},\ldots,y_n \right)\1(y_i\in\overline{B})=P_k\left(X_1,\ldots,X_{i-1},z,y_{i+1},\ldots,y_n \right)\1(y_i\in\overline{B}).
\end{equation}
We thus have
\begin{align*}
    &\Eb_{X_i}\left[\Eb_{i+1}\left[P_{j_1}\right] \Eb_{i+1}\left[P_{j_2}\right] \right]\\
    &=\sum_{j=1}^Np(j)\left[\int_{\Rb^{n-i}}P_{j_1}(X_1,\ldots,X_{i-1},x_j,y_{i+1},\ldots y_n)f(y_{i+1})\cdots f(y_n) \dd y_{i+1}\cdots \dd y_{n} \right.\\
    &\hspace{3.5cm}\left.\cdot\int_{\Rb^{n-i}}P_{j_2}(X_1,\ldots,X_{i-1},x_j,y'_{i+1},\ldots y'_n)f(y'_{i+1})\cdots f(y'_n) \dd y'_{i+1}\cdots \dd y'_{n} \right]\\
    &\hspace{1cm}+p(\overline{B})\left[\int_{\Rb^{n-i}}P_{j_1}(X_1,\ldots,X_{i-1},z,y_{i+1},\ldots y_n)f(y_{i+1})\cdots f(y_n) \dd y_{i+1}\cdots \dd y_{n} \right.\\
    &\hspace{3.5cm}\left.\cdot\int_{\Rb^{n-i}}P_{j_2}(X_1,\ldots,X_{i-1},z,y'_{i+1},\ldots y'_n)f(y'_{i+1})\cdots f(y'_n) \dd y'_{i+1}\cdots \dd y'_{n} \right].
\end{align*}
Now, observe that 
\begin{equation}\label{Proof var Inter eq Cov 6}
   \int_{\Rb^{n-i}}P_{k}(X_1,\ldots,X_{i-1},x_j,y_{i+1},\ldots y_n)f(y_{i+1})\cdots f(y_n) \dd y_{i+1}\cdots \dd y_{n}=\Eb_i^j[P_k]. 
\end{equation}
Indeed, it holds
\begin{align*}
    \Eb_i^j[P_k]&= \frac{1}{p(j)}\Eb\left[P_k\1(X_i\in B_j)\mid X_1,\ldots,X_{i-1} \right]  \\
    &=\frac{1}{p(j)}\int_{\Rb^{n-i+1}}P_k(X_1,\ldots,X_{i-1},y_i,y_{i+1},\ldots,y_n)\1(y_i\in B_j)f(y_i)f(y_{i+1})\cdots f(y_{n})\dd y_i\dd y_{i+1}\dd y_n \\
    &=\int_{\Rb^{n-i}}P_k(X_1,\ldots,X_{i-1},x_j,y_{i+1},\ldots,y_n)f(y_{i+1})\cdots f(y_{n})\dd y_{i+1}\dd y_n,
\end{align*}
where the last equality follows from \eqref{Proof var Inter eq Cov 4}.
Similarly, using \eqref{Proof var Inter eq Cov 5} one can prove that for $z\in\overline{B}$ it holds
$$
\int_{\Rb^{n-i}}P_{k}(X_1,\ldots,X_{i-1},z,y_{i+1},\ldots y_n)f(y_{i+1})\cdots f(y_n) \dd y_{i+1}\cdots \dd y_{n} =\Eb_i^{comp}[P_k].
$$
We thus have
\begin{align*}
    &\Eb_{X_i}\left[\Eb_{i+1}\left[P_{j_1}\right] \Eb_{i+1}\left[P_{j_2}\right] \right]=\sum_{j=1}^Np(j)\Eb_i^j[P_{j_1}]\Eb_i^j[P_{j_2}] +p(\overline{B})\Eb_i^{comp}[P_{j_1}]\Eb_i^{comp}[P_{j_2}],
\end{align*}
and, using \eqref{Proof Var Inter eq a.s.1} and \eqref{Proof Var Inter eq a.s.2} we finally obtain
\begin{multline}\label{Proof var Inter eq Cov 7}
    \Eb_{X_i}\left[\Eb_{i+1}\left[P_{j_1}\right] \Eb_{i+1}\left[P_{j_2}\right] \right]=p(j_1)\Eb_i^{j_1}\left[ P_{j_1}\right]\Eb_i^{j_1}\left[ P_{j_2}\right]+p(j_2)\Eb_i^{j_2}\left[ P_{j_1}\right]\Eb_i^{j_2}\left[ P_{j_2}\right]\\
    +\left(1-p(j_1)-p(j_2) \right)\Eb_i^{j_2}\left[ P_{j_1}\right]\Eb_i^{j_1}\left[ P_{j_2}\right].
\end{multline}
Putting \eqref{Proof var Inter eq Cov 2}, \eqref{Proof var Inter eq Cov 3} and \eqref{Proof var Inter eq Cov 7} in \eqref{Proof var Inter eq Cov 1}, we obtain
\begin{align*}
    &\Cov\left(P_{j_1},P_{j_2} \right)\\
    &=\sum_{i=1}^n\Eb\left[p(j_1)\Eb_i^{j_1}\left[ P_{j_1}\right]\Eb_i^{j_1}\left[ P_{j_2}\right]+p(j_2)\Eb_i^{j_2}\left[ P_{j_1}\right]\Eb_i^{j_2}\left[ P_{j_2}\right]  +\left(1-p(j_1)-p(j_2) \right)\Eb_i^{j_2}\left[ P_{j_1}\right]\Eb_i^{j_1}\left[ P_{j_2}\right]\right.\\
    &\hspace{2cm}+\left. \left\{p(j_1)\Eb_i^{j_1}\left[ P_{j_1}\right]+(1-p(j_1))\Eb_i^{j_2}\left[ P_{j_1}\right] \right\}\left\{p(j_2)\Eb_i^{j_2}\left[ P_{j_2}\right]+(1-p(j_2))\Eb_i^{j_1}\left[ P_{j_2}\right] \right\} \right]\\
    &=\sum_{i=1}^np(j_1)p(j_2) \Eb\left[\left(\Eb_i^{j_1}\left[ P_{j_1}\right]-\Eb_i^{j_2}\left[ P_{j_1}\right] \right)\left(\Eb_i^{j_1}\left[ P_{j_2}\right]-\Eb_i^{j_2}\left[ P_{j_2}\right] \right) \right],
\end{align*}
and Cauchy-Schwarz inequality gives 
\begin{equation}\label{Proof var inter eq Cauchy Schwarz}
    \left\vert \Cov(P_{j_1},P_{j_2}) \right\vert \leq \sum_{i=1}^np(j_1)p(j_2)\sqrt{\Eb\left[\left(\Eb_i^{j_1}\left[ P_{j_1}\right]-\Eb_i^{j_2}\left[ P_{j_1}\right] \right)^2 \right]}\sqrt{\Eb\left[\left(\Eb_i^{j_1}\left[ P_{j_2}\right]-\Eb_i^{j_2}\left[ P_{j_2}\right] \right)^2 \right]}.
\end{equation}
Now, using \eqref{Proof var Inter eq Cov 6} and Jensen's inequality we have
\begin{align*}
   &\Eb\left[\left(\Eb_i^{j_1}\left[ P_{j_1}\right]-\Eb_i^{j_2}\left[ P_{j_1}\right] \right)^2 \right] \\
   & =\Eb\left[\left\{\int_{\Rb^{n-i}}\left( P_{j_1}(X_1,\ldots,X_{i-1},x_{j_1},y_{i+1},\ldots,y_n)-P_{j_1}(X_1,\ldots,X_{i-1},x_{j_2},y_{i+1},\ldots,y_n) \right)\right.\right.\\
   &\hspace{9cm}\left.\left. f(y_{i+1})\cdots f(y_n) \dd y_{i+1}\cdots \dd y_n \right\}^2\right]\\
   &\leq \Eb\left[\int_{\Rb^{n-i}}\left\{ P_{j_1}(X_1,\ldots,X_{i-1},x_{j_1},y_{i+1},\ldots,y_n)-P_{j_1}(X_1,\ldots,X_{i-1},x_{j_2},y_{i+1},\ldots,y_n)\right\}^2\right.\\
   &\hspace{9.4cm}\left. f(y_{i+1})\cdots f(y_n) \dd y_{i+1}\cdots \dd y_n \right]\\
   &=\Eb\left[\left\{P_{j_1}(X_1,\ldots,X_{i-1},x_{j_1},X_{i+1},\ldots,X_n)-P_{j_1}(X_1,\ldots,X_{i-1},x_{j_2},X_{i+1},\ldots,X_n) \right\}^2\right]\\
   &=\Eb\left[\left(\left[\frac{1}{n}+Y\right]_{-\tau}^\tau -\left[Y\right]_{-\tau}^\tau \right)^2 \right],
\end{align*}
where 
$$
Y=\frac{1}{n}\displaystyle\sum_{\substack{k=1 \\ k\neq i}}^n \1(X_k\in B_{j_1})+\frac{2}{n\alpha}\sum_{k=1}^nW_{k j_1}-p_0(j_1).
$$
Note that since $[\cdot]_{-\tau}^\tau$ is continuous Lipschitz with Lipschitz constant $1$, it holds
$$ 
\Eb\left[\left(\Eb_i^{j_1}\left[ P_{j_1}\right]-\Eb_i^{j_2}\left[ P_{j_1}\right] \right)^2 \right]\leq \frac{1}{n^2}.
$$
However, we can provide another bound when $\vert p(j_1)-p_0(j_1)\vert \geq 2(\tau+1/n)$.
Assume that $ p(j_1)-p_0(j_1) \geq 2(\tau+1/n)$.
We have
\begin{align*}
  &\Eb\left[\left(\Eb_i^{j_1}\left[ P_{j_1}\right]-\Eb_i^{j_2}\left[ P_{j_1}\right] \right)^2 \right]\\
  &\leq \Eb\left[\left(\left[\frac{1}{n}+Y\right]_{-\tau}^\tau -\left[Y\right]_{-\tau}^\tau \right)^2\1(Y\leq \tau) \right]+\Eb\left[\left(\left[\frac{1}{n}+Y\right]_{-\tau}^\tau -\left[Y\right]_{-\tau}^\tau \right)^2\1(Y> \tau) \right] \\
  &\leq \frac{1}{n^2}\Pb(Y\leq \tau)\\
  &=\frac{1}{n^2}\Pb\left(\frac{1}{n}\displaystyle\sum_{\substack{k=1 \\ k\neq i}}^n \1(X_k\in B_{j_1})+\frac{2}{n\alpha}\sum_{k=1}^nW_{k j_1}-p_0(j_1) \leq \tau \right)\\
  &\leq \frac{1}{n^2}\Pb\left(\frac{1}{n}\sum_{k=1}^n \1(X_k\in B_{j_1})-\frac{1}{n}+\frac{2}{n\alpha}\sum_{k=1}^nW_{k j_1}-p_0(j_1) \leq \tau \right)\\
  &=\frac{1}{n^2}\Pb\left(\widehat{p}_{j_1}\leq \tau+\frac{1}{n}+p_0(j_1) \right)\\
  &\leq \frac{1}{n^2}\Pb\left(\vert \widehat{p}_{j_1}-p(j_1)\vert \geq p(j_1)-p_0(j_1)-\tau-\frac{1}{n} \right)
\end{align*}
Now, if $ p(j_1)-p_0(j_1) \geq 2(\tau+1/n)$ then we have $0<p(j_1)-p_0(j_1)-\tau-\frac{1}{n} \leq p(j_1)\leq 1\leq \frac{1}{\alpha}$ and \eqref{Eq Concentration inequality proof variance of the stat in the interactive scenario} gives
\begin{align*}
    \Eb\left[\left(\Eb_i^{j_1}\left[ P_{j_1}\right]-\Eb_i^{j_2}\left[ P_{j_1}\right] \right)^2 \right]&\leq \frac{2}{n^2}\exp\left(-\frac{n\alpha^2\left(p(j_1)-p_0(j_1)-\tau-1/n \right)^2}{42} \right)\\
    &\leq \frac{2}{n^2}\exp\left(-\frac{n\alpha^2\left(p(j_1)-p_0(j_1)\right)^2}{168} \right).
\end{align*}
One can prove the same result if $ p(j_1)-p_0(j_1) \leq -2(\tau+1/n)$, and similar bounds with $j_1$ replaced by $j_2$ hold for $\Eb\left[\left(\Eb_i^{j_1}\left[ P_{j_2}\right]-\Eb_i^{j_2}\left[ P_{j_2}\right] \right)^2\right]$.
We can now conclude.\\
If $j_1\neq j_2$ are such that $\vert p(j_1)-p_0(j_1)\vert \geq 2(\tau+1/n)$ and $\vert p(j_2)-p_0(j_2)\vert \geq 2(\tau+1/n)$ then \eqref{Proof var inter eq Cauchy Schwarz} gives
$$
\left\vert \Cov\left(P_{j_1},P_{j_2} \right)\right\vert \leq \frac{2p(j_1)p(j_2)}{n}\exp\left(-\frac{n\alpha^2\left[(p(j_1)-p_0(j_1))^2+(p(j_2)-p_0(j_2))^2\right]}{336} \right).
$$
If $j_1\neq j_2$ are such that $\vert p(j_1)-p_0(j_1)\vert < 2(\tau+1/n)$ and $\vert p(j_2)-p_0(j_2)\vert \geq 2(\tau+1/n)$ then \eqref{Proof var inter eq Cauchy Schwarz} gives
\begin{align*}
    &\left\vert \Cov\left(P_{j_1},P_{j_2} \right)\right\vert\\
    &\leq \frac{\sqrt{2}p(j_1)p(j_2)}{n}\exp\left(-\frac{n\alpha^2(p(j_2)-p_0(j_2))^2}{336} \right)\\
    &=\frac{\sqrt{2}p(j_1)p(j_2)}{n}\exp\left(-\frac{n\alpha^2\left[(p(j_1)-p_0(j_1))^2+(p(j_2)-p_0(j_2))^2\right]}{336} \right)\exp\left(\frac{n\alpha^2(p(j_1)-p_0(j_1))^2}{336} \right)\\
    &\leq \frac{\sqrt{2}\exp(1/21)p(j_1)p(j_2)}{n}\exp\left(-\frac{n\alpha^2\left[(p(j_1)-p_0(j_1))^2+(p(j_2)-p_0(j_2))^2\right]}{336} \right),
\end{align*}
since $\vert p(j_1)-p_0(j_1)\vert < 2(\tau+1/n)\leq 4/\sqrt{n\alpha^2}$.
The same result holds if $j_1\neq j_2$ are such that $\vert p(j_1)-p_0(j_1)\vert \geq  2(\tau+1/n)$ and $\vert p(j_2)-p_0(j_2)\vert < 2(\tau+1/n)$.
Finally, if $j_1\neq j_2$ are such that $\vert p(j_1)-p_0(j_1)\vert <  2(\tau+1/n)$ and $\vert p(j_2)-p_0(j_2)\vert < 2(\tau+1/n)$, then \eqref{Proof var inter eq Cauchy Schwarz} gives
\begin{align*}
    \left\vert \Cov\left(P_{j_1},P_{j_2} \right)\right\vert &\leq \frac{p(j_1)p(j_2)}{n}\\
    &\leq \frac{p(j_1)p(j_2)}{n}\exp\left(\frac{2}{21}\right)\exp\left(-\frac{n\alpha^2\left[(p(j_1)-p_0(j_1))^2+(p(j_2)-p_0(j_2))^2\right]}{336} \right),
\end{align*}
which ends the proof of \eqref{Proof Var of D_B preliminary result 2}.
\end{proof}

\subsection{Proof of Theorem \ref{Thrm Interactive Upper bound}}\label{App Proof upper bound interactive scenario}

The outline of the proof is similar to that of Theorem \ref{Thrm Upper Bound NI} : we first prove that the choice of $t_1$ and $t_2$ in \eqref{Eq interactive t1 and t2} yields $\Pb_{Q_{f_0}^n}(\Phi=1)\leq \gamma/2$ and we then exhibit $\rho_1,\rho_2>0$ such that
$$
\begin{cases}
\int_{B}\vert f-f_0\vert\geq \rho_1 \Rightarrow \Pb_{Q_f^n}(D_B< t_1)\leq \gamma/2\\
\int_{\bar{B}}\vert f-f_0\vert\geq \rho_2 \Rightarrow \Pb_{Q_f^n}(T_B< t_2)\leq \gamma/2.
\end{cases}
$$
The quantity $\rho_1+\rho_2$ will then provide an upper bound on  $\Ec_{n,\alpha}(f_0,\gamma)$.

We have already seen in the proof of the upper bound in the non-interactive scenario that the choice $t_2=\sqrt{20/(n\alpha^2\gamma)}$ gives $\Pb_{Q_{f_0}^n}(T_B\geq t_2)\leq \gamma/4$.
Moreover, Chebychev's inequality and Proposition~\ref{Prop Test Stat Int} yield
\begin{align*}
    \Pb_{Q_{f_0}^n}( D_B\geq t_1)=\Pb_{Q_{f_0}^n}( D_B-\Eb_{Q_{f_0}^n}[D_B]\geq t_1)&\leq \Pb_{Q_{f_0}^n}\left(\vert D_B-\Eb_{Q_{f_0}^n}[D_B]\vert \geq t_1\right)\\
    &\leq \frac{\Var_{Qf_0^n}(D_B)}{t_1^2}\\
    &\leq \frac{5}{(n\alpha^2)^2t_1^2}
    \leq \frac{\gamma}{4}
\end{align*}
for $t_1= 2\sqrt{5}/(n\alpha^2\sqrt{\gamma})$.
We thus have 
$$
\Pb_{Q_{f_0}^n}(\Phi=1)\leq \Pb_{Q_{f_0}^n}( D_B\geq t_1)+  \Pb_{Q_{f_0}^n}(T_B\geq t_2)\leq \frac{\gamma}{2}.
$$
We have seen in the proof of Theorem \ref{Thrm Upper Bound NI} (upper bound in the non-interactive scenario) that if we set 
$$
\rho_2=2\int_{\bar{B}}f_0+\left( 1+\frac{1}{\sqrt{2}}\right)t_2,
$$
then we have 
$$
\int_{\bar{B}}\vert f-f_0\vert\geq \rho_2\implies\Pb_{Q_{f}^n}(T_B < t_2)\leq\frac{\gamma}{2}.
$$
It remains now to exhibit $\rho_1$ such that $\int_B\vert f-f_0\vert\geq \rho_1$ implies $\Pb_{Q_{f}^n}(D_B < t_1)\leq \gamma/2.$
Chebychev's inequality gives
\begin{align*}
\Pb_{Q_f^n}(D_B< t_1)
&=\Pb_{Q_f^n}\left(\Eb_{Q_f^n}[D_B]-D_B>\Eb_{Q_f^n}[D_B]- t_1\right)\\
&\leq \frac{\Var_{Qf^n}(D_B)}{\left( \Eb_{Q_f^n}[D_B]- t_1\right)^2}\\
&\leq \frac{\frac{5}{(n\alpha^2)^2}}{\left( \Eb_{Q_f^n}[D_B]- t_1\right)^2}+\frac{\frac{67D_\tau(f)}{n\alpha^2}}{\left( \Eb_{Q_f^n}[D_B]- t_1\right)^2},
\end{align*}
if $\Eb_{Q_f^n}[D_B]- t_1>0$.
Now, observe that if $D_\tau(f)\geq 12(t_1+6\tau/\sqrt{n})$, Proposition~\ref{Prop Test Stat Int} implies 
$$
\Eb_{Q_f^n}[D_B]-t_1\geq \frac{1}{6}D_\tau(f)-\frac{6\tau}{\sqrt{n}}-t_1\geq t_1+\frac{6\tau}{\sqrt{n}}\geq t_1,
$$
and 
$$
\Eb_{Q_f^n}[D_B]-t_1\geq \frac{1}{6}D_\tau(f)-\left(\frac{6\tau}{\sqrt{n}}+t_1\right)\geq \frac{1}{6}D_\tau(f)-\frac{1}{12}D_\tau(f)=\frac{1}{12}D_\tau(f).
$$
Thus, if $D_\tau(f)\geq 12(t_1+6\tau/\sqrt{n})$ we obtain 
$$
\Pb_{Q_f^n}(D_B< t_1)\leq \frac{5}{(n\alpha^2)^2t_1^2}+\frac{144\times 67}{n\alpha^2D_\tau(f)}
=\frac{\gamma}{4}+\frac{9648}{n\alpha^2D_\tau(f)}.
$$
Thus, if $D_\tau(f)$ satisfies
\begin{equation*}
D_\tau(f)\geq\frac{C_\gamma}{n\alpha^2}, \quad \text{with } C_\gamma=  \max\left\{ \frac{24\sqrt{5}+72}{\sqrt{\gamma}},\frac{9648\times 4}{\gamma} \right\}
\end{equation*}
then we have $\Pb_{Q_f^n}(D_B< t_1)\leq \gamma/2$.
We now exhibit $\rho_1$ such that $\int_B \vert f-f_0\vert \geq \rho_1$ implies $D_\tau(f)\geq C_\gamma/(n\alpha^2)$.
To this aim, we will use the following facts
\begin{enumerate}[label=\roman*)]
\item $D_\tau(f)\geq \min\left\{ \sum_{j=1}^N\vert p(j)-p_0(j)\vert^2, \tau\sqrt{\sum_{j=1}^N\vert p(j)-p_0(j)\vert^2} \right\}$,
\item  $\sum_{j=1}^N\vert p(j)-p_0(j)\vert^2\geq C_\gamma^2/(n\alpha^2)\Rightarrow \min\left\{ \sum_{j=1}^N\vert p(j)-p_0(j)\vert^2, \tau\sqrt{\sum_{j=1}^N\vert p(j)-p_0(j)\vert^2} \right\}\geq C_\gamma/(n\alpha^2)$,
\item $\left(\int_B\vert f-f_0\vert \right)^2\leq 4(L+L_0)^2\vert B\vert^2h^{2\beta}+\vert B\vert/(2h)\sum_{j=1}^N\vert p(j)-p_0(j)\vert^2$.
\end{enumerate}
We admit for now these three facts  and conclude the proof of our upper bound.
If we have
$$
\left(\int_B\vert f-f_0\vert \right)^2\geq 4(L+L_0)^2\vert B\vert^2h^{2\beta}+\frac{\vert B \vert}{2h}\frac{C_\gamma^2}{n\alpha^2}
$$
then iii) implies
$$
\sum_{j=1}^N\vert p(j)-p_0(j)\vert^2\geq \frac{C_\gamma^2}{n\alpha^2},
$$
and ii) combined with i) yield $D_\tau(f)\geq C_\gamma/(n\alpha^2)$ and thus $\Pb_{Q_f^n}(D_B< t_1)\leq \gamma/2$.
We can then take 
$$
\rho_1= \sqrt{4(L+L_0)^2\vert B\vert^2h^{2\beta}+\frac{\vert B \vert}{2h}\frac{C_\gamma^2}{n\alpha^2}}.
$$
For all $f\in H(\beta,L)$ satisfying $\Vert f-f_0\Vert_1\geq \rho_1+\rho_2$ it holds
$$
\Pb_{Q_{f_0}^n}(\Phi=1)+\Pb_{Q_f^n}(\Phi=0)\leq \frac{\gamma}{2}+ \min\left\{ \Pb_{Q_f^n}( D_B< t_1),  \Pb_{Q_f^n}(T_B< t_2)\right\}\leq \frac{\gamma}{2}+\frac{\gamma}{2}=\gamma,
$$
since $\Vert f-f_0\Vert_1\geq \rho_1+\rho_2$ implies $\int_{B}\vert f-f_0\vert\geq \rho_1$ or $ \int_{\bar{B}}\vert f-f_0\vert\geq \rho_2$.
Consequently, we have
\begin{align*}
   \Ec_{n,\alpha}(f_0,\gamma)&\leq \rho_1+\rho_2= \sqrt{4(L+L_0)^2\vert B\vert^2h^{2\beta}+\frac{\vert B \vert}{2h}\frac{C_\gamma^2}{n\alpha^2}}+2\int_{\bar{B}}f_0+\left( 1+\frac{1}{\sqrt{2}}\right)t_2\\
   &\leq C(L,L_0,\gamma) \left[\vert B\vert h^\beta + \sqrt{\frac{\vert B\vert}{hn\alpha^2}} +\int_{\bar{B}}f_0 + \frac{1}{\sqrt{n\alpha^2}} \right].
\end{align*}
The choice 
$
h\asymp \vert B\vert ^{-\frac{1}{2\beta+1}}(n\alpha^2)^{-\frac{1}{2\beta+1}}
$
 yields 
$$
 \Ec_{n,\alpha}(f_0,\gamma) \leq C\left[ \vert B\vert ^{\frac{\beta+1}{2\beta+1}}(n\alpha^2)^{-\frac{\beta}{2\beta+1}} +\int_{\overline B} f_0+ \frac{1}{\sqrt{n\alpha^2}} \right],
$$
which ends the proof of Theorem \ref{Thrm Interactive Upper bound}.
It remains to prove i), ii) and iii).
Let's start with the proof of i).
If $\tau\geq \sqrt{\sum_{j=1}^N\vert p(j)-p_0(j)\vert^2}$, then $\tau\geq \vert p(j)-p_0(j)\vert$ for all $j$, and we thus have 
$$
D_\tau(f)= \sum_{j=1}^N\vert p(j)-p_0(j)\vert^2 = \min\left\{ \sum_{j=1}^N\vert p(j)-p_0(j)\vert^2, \tau\sqrt{\sum_{j=1}^N\vert p(j)-p_0(j)\vert^2} \right\}.
$$
We now deal with the case $\tau< \sqrt{\sum_{j=1}^N\vert p(j)-p_0(j)\vert^2}$.
In this case, we can write
\begin{align*}
D_\tau(f)-\tau\sqrt{\sum_{j=1}^N\vert p(j)-p_0(j)\vert^2}&=\sum_{j=1}^N\vert p(j)-p_0(j)\vert \min\left\{ \vert p(j)-p_0(j)\vert, \tau  \right\}-\tau\frac{\sum_{j=1}^N\vert p(j)-p_0(j)\vert^2}{\sqrt{\sum_{k=1}^N\vert p(k)-p_0(k)\vert^2}}\\
&=\sum_{j=1}^N\vert p(j)-p_0(j)\vert\underbrace{\left[\min\left\{ \vert p(j)-p_0(j)\vert, \tau  \right\}- \frac{\tau\vert p(j)-p_0(j)\vert}{\sqrt{\sum_{k=1}^N\vert p(k)-p_0(k)\vert^2}}  \right]}_{=:A_j},
\end{align*}
and $A_j\geq 0$ for all $j$.
Indeed, if $j$ is such that $\vert p(j)-p_0(j)\vert<\tau$ it holds
$$
A_j=\vert p(j)-p_0(j)\vert\left[1-\frac{\tau}{\sqrt{\sum_{k=1}^N\vert p(k)-p_0(k)\vert^2}}  \right]\geq 0,
$$
and if $j$ is such that $\vert p(j)-p_0(j)\vert\geq \tau$ it holds
$$
A_j=\tau\left[1-\frac{\vert p(j)-p_0(j)\vert}{\sqrt{\sum_{k=1}^N\vert p(k)-p_0(k)\vert^2}}  \right]\geq 0.
$$
Thus, if $\tau< \sqrt{\sum_{j=1}^N\vert p(j)-p_0(j)\vert^2}$ we have
$$
D_\tau(f)\geq \tau\sqrt{\sum_{j=1}^N\vert p(j)-p_0(j)\vert^2}= \min\left\{ \sum_{j=1}^N\vert p(j)-p_0(j)\vert^2, \tau\sqrt{\sum_{j=1}^N\vert p(j)-p_0(j)\vert^2} \right\},
$$
which end the proof of i).
We now prove ii). 
Assume that $\sum_{j=1}^N\vert p(j)-p_0(j)\vert^2\geq C_\gamma^2/(n\alpha^2)$.
It holds $C_\gamma^2\geq C_\gamma$ since $C_\gamma\geq 1$ and we thus have $\sum_{j=1}^N\vert p(j)-p_0(j)\vert^2\geq C_\gamma/(n\alpha^2)$. 
It also holds 
$$
\tau\sqrt{\sum_{j=1}^N\vert p(j)-p_0(j)\vert^2} \geq \tau\cdot\frac{C_\gamma}{\sqrt{n\alpha^2}}=\frac{C_\gamma}{n\alpha^2},
$$
yielding ii).
Finally, Cauchy-Schwarz inequality yields 
\begin{align*}
\left(\int_B\vert f-f_0\vert\right)^2&\leq \vert B\vert \int_B\vert f-f_0\vert^2\\
&\leq \vert B\vert\cdot \left\vert \int_B\vert f-f_0\vert^2-\frac{1}{2h}\sum_{j=1}^N\left( p(j)-p_0(j)\right)^2  \right\vert+\frac{\vert B \vert}{2h}\sum_{j=1}^N\left( p(j)-p_0(j)\right)^2.
\end{align*}
Now, observe that 
$$
 \left\vert \int_B\vert f-f_0\vert^2-\frac{1}{2h}\sum_{j=1}^N\left( p(j)-p_0(j)\right)^2  \right\vert= \left\vert \sum_{j=1}^N\int_{B_j}\left[(f-f_0)(x)- \frac{p(j)-p_0(j)}{2h}\right]^2\dd x \right\vert, 
$$
and observe also that for $x\in B_j$ it holds
\begin{align*}
\left\vert(f-f_0)(x)- \frac{p(j)-p_0(j)}{2h}\right\vert &= \left\vert \frac{1}{2h}\int_{B_j}[(f-f_0)(x)-(f-f_0)(u)] \dd u \right\vert\\
&\leq \frac{1}{2h}\int_{B_j}\left[\vert f(x)-f(u) \vert +\vert f_0(x)-f_0(u) \vert \right]\dd u\\
&\leq \frac{L+L_0}{2h}\int_{B_j}\vert x-u\vert^\beta\dd u\\
&\leq  \frac{L+L_0}{2h}\int_{B_j}(2h)^\beta\dd u\\
&\leq 2(L+L_0)h^\beta.
\end{align*}
This gives 
$$
\left\vert \int_B\vert f-f_0\vert^2-\frac{1}{2h}\sum_{j=1}^N\left( p(j)-p_0(j)\right)^2  \right\vert \leq \sum_{j=1}^N\int_{B_j} 4(L+L_0)^2h^{2\beta}=4(L+L_0)^2\vert B \vert h^{2\beta},
$$
which yields iii).

\subsection{Proof of Theorem \ref{Thrm Interactive Lower bound}}\label{App Proof Lower bound interactive scenario}

Let  $B\subset \Rb$ be a non-empty  compact set, and let $(B_j)_{j=1,\ldots,N}$ be a partition of $B$, $h>0$ be the bandwidth and $(x_1,\ldots,x_N)$ be the centering points, that is $B_j=[x_j-h,x_j+h]$ for all $j\in \llbr 1,N\rrbr$.
Let $\psi : [-1,1]\rightarrow \Rb$ be such that $\psi\in H(\beta,L)$, $\int \psi=0$ and $\int \psi^2=1$.
For $j\in\llbr 1,N\rrbr$, define
$$
\psi_j:t\in\Rb \mapsto \frac{1}{\sqrt{h}}\psi\left(\frac{t-x_j}{h} \right) .
$$
Note that the support of $\psi_j$ is $B_j$, $\int \psi_j=0$ and $(\psi_j)_{j=1,\ldots,N}$ is an orthonormal family.

For $\delta>0$ and $\nu\in\Vc_N=\{-1,1\}^N$, define the functions 
$$f_\nu:x\in\Rb\mapsto f_0(x)+\delta\sum_{j=1}^N\nu_j\psi_j(x),$$
The following lemma shows that for $\delta$ properly chosen, for all $\nu\in\Vc_N$, $f_\nu$ is a density belonging to $H(\beta,L)$ and $f_\nu$ is sufficiently far away from $f_0$ in a $L_1$ sense.
\begin{lemma}
If the parameter $\delta$  appearing in the definition of $f_\nu$ satisfies
$$
\delta\leq \sqrt{h}\cdot\min\left\{ \frac{C_0(B)}{\Vert\psi\Vert_\infty} , \frac{1}{2}\left(1-\frac{L_0}{L} \right)h^\beta\right\},
$$
where $C_0(B):=\min\{f_0(x) : x\in B \}$,  then we have
\begin{enumerate}[label=\roman*)]
\item $f_\nu\geq 0$ and $\int f_\nu=1$, for all $\nu\in \Vc_N$,
\item $f_\nu\in H(\beta,L)$, for all $\nu\in \Vc_N$,
\item $\Vert f_\nu-f_0\Vert_1=   C_1\delta N\sqrt{h}$, for all $\nu\in \Vc_N$, with $C_1=\int_{-1}^1\vert \psi\vert$.
\end{enumerate}
\end{lemma}

\begin{proof}
We first prove $i)$. 
Since $\int \psi_j= 0$ for all $j=1,\ldots,n$, it holds $\int f_\nu=\int f_0=1$ for all $\nu$.
Since $\text{Supp}(\psi_k)=B_k$ for all $k=1,\ldots,N$, it holds $f_\nu \equiv f_0$ on $B^c$ and thus $f_\nu$ is non-negative on $B^c$.
Now, for $x\in B_j$ it holds for all $\nu\in\Vc_N$
$$
f_\nu(x)= f_0(x)+\delta\nu_j \psi_j(x)\geq C_0(B)-\delta\Vert \psi_j \Vert_\infty \geq C_0(B)-\frac{\delta\Vert \psi\Vert_\infty}{\sqrt{h}}\geq 0,
$$ 
since $\delta\leq C_0(B)\sqrt{h}/\Vert \psi \Vert_\infty$
Thus, $f_\nu$ is non-negative on $\Rb$ for all $\nu\in \Vc_N$.

To prove $ii)$, we have to show that $\vert f_\nu(x)-f_\nu(y)\vert \leq L \vert x-y\vert^\beta$, for all $\nu\in \Vc_N$, for all $x,y\in\Rb$.
Since $f_\nu\equiv f_0$ on $B^c$ and $f_0\in H(\beta,L_0)$, this result is trivial for $x,y\in B^c$.
If $x\in B_l$ and $y\in B_k$ it holds 
\begin{align*}
\vert f_\nu(x)-f_\nu(y)\vert &\leq \vert f_0(x)-f_0(y)\vert + \left\vert \delta\nu_l \psi_l(x) - \delta\nu_k \psi_k(y)  \right\vert\\
&\leq L_0\vert x-y\vert^\beta+ \left\vert \delta\nu_l \psi_l(x) - \delta\nu_l \psi_l(y)  \right\vert+\left\vert \delta\nu_k \psi_k(x) - \delta\nu_k \psi_k(y)  \right\vert\\
&\leq L_0\vert x-y\vert^\beta+\frac{\delta}{\sqrt{h}} \left\vert \psi\left(\frac{x-x_l}{h}\right)-\psi\left(\frac{y-x_l}{h}\right)\right\vert+\frac{\delta}{\sqrt{h}} \left\vert \psi\left(\frac{x-x_k}{h}\right)-\psi\left(\frac{y-x_k}{h}\right)\right\vert\\
&\leq L_0\vert x-y\vert^\beta + \frac{\delta}{h^{\beta+1/2}} \cdot  L\vert x-y\vert^\beta+ \frac{\delta}{h^{\beta+1/2}}\cdot  L\vert x-y\vert^\beta\\
&=\left(\frac{L_0}{L}+ \frac{2\delta}{h^{\beta+1/2}} \right)L\vert x-y\vert^\beta\\
&\leq L\vert x-y\vert^\beta
\end{align*}
where we have used $\psi\in H(\beta,L)$ and $\delta\leq \frac{h^{\beta+1/2}}{2}\left(1-\frac{L_0}{L}\right)$.
Thus,  it holds $\vert f_\nu (x)-f_\nu(y)\vert\leq L\vert x-y\vert^\beta$ for all $\nu\in \Vc_N$, $x\in B_l$ and $y\in B_k$.
The case $x\in B^c$ and $y\in B_k$ can be handled in a similar way, which ends the proof of $ii)$.

We now prove $iii)$.
It holds
$$
\int_\Rb\vert f_\nu-f_0\vert =\int_\Rb\left\vert \delta\sum_{j=1}^N\nu_j\psi_j(x) \right\vert\dd x = \sum_{k=1}^N\int_{B_k}\left\vert \delta\nu_k \psi_k(x) \right\vert\dd x = \delta N\sqrt{h} \int_{-1}^1\vert \psi\vert.
$$
\end{proof}

For a privacy mechanism $Q\in \Qc_\alpha$, we denote by $Q_{f_0}^n$ (respectively $Q_{f_\nu}^n$)  the distribution of $(Z_1,\ldots,Z_n)$ when the $X_i$'s are distributed according to $f_0$ (respectively to $f_\nu$). We set $\bar{Q}^n=1/2^N\sum_{\nu\in\Vc_N} Q_{f_\nu}^n$.
If $\delta$ is chosen such that $\delta\leq \sqrt{h}\cdot\min\left\{ \frac{C_0(B)}{\Vert\psi\Vert_\infty} , \frac{1}{2}\left(1-\frac{L_0}{L} \right)h^\beta\right\}$, setting  $\rho^\star =C_1\delta N\sqrt{h}$, we deduce from the above lemma that if
\begin{equation}\label{Proof interactive LB Cond1}
\text{KL}(Q_{f_0}^n,\bar{Q}^n) \leq 2(1-\gamma)^2 \text{ for all } Q\in \Qc_\alpha,
\end{equation}
 then it holds
$$
 \inf_{Q\in \Qc_\alpha}\inf_{\phi\in \Phi_Q} \sup_{f\in H_1(\rho^\star)}\left\{\Pb_{Q_{f_0}^n}(\phi=1)+ \Pb_{Q_{f}^n}(\phi=0) \right\}\geq \gamma,
$$
where $H_1(\rho^\star):= \{ f \in H(\beta, L) : f\geq 0, \int f=1, \Vert f-f_0\Vert_1 \geq \rho^\star\}$, and consequently $\Ec_{n,\alpha}(f_0,\gamma)\geq \rho^\star$.
Indeed, if \eqref{Proof interactive LB Cond1} holds, then we have
\begin{align*}
\inf_{Q\in \Qc_\alpha}\inf_{\phi\in \Phi_Q} \sup_{f\in H_1(\rho^\star)}\left\{\Pb_{Q_{f_0}^n}(\phi=1)+ \Pb_{Q_{f}^n}(\phi=0) \right\}&\geq \inf_{Q\in \Qc_\alpha}\inf_{\phi\in \Phi_Q} \left( \Pb_{Q_{f_0}^n}(\phi=1)+ \frac{1}{2^N}\sum_{\nu\in\Vc_N}\Pb_{Q_{f_\nu}^n}(\phi=0) \right)\\
&= \inf_{Q\in \Qc_\alpha}\inf_{\phi\in \Phi_Q} \left( 1-\left[\Pb_{Q_{f_0}^n}(\phi=0)-\Pb_{\bar{Q}^n}(\phi=0)  \right] \right)\\
&\geq \inf_{Q\in \Qc_\alpha}\left[ 1-\text{TV}(Q_{f_0}^n, \bar{Q}^n)\right]\\
&\geq \inf_{Q\in \Qc_\alpha}\left[ 1-\sqrt{\frac{\text{KL}(Q_{f_0}^n, \bar{Q}^n)}{2}}\right]\\
&\geq \gamma,
\end{align*}
where the second to last inequality follows from Pinsker's inequality.
We now prove that \eqref{Proof interactive LB Cond1} holds under an extra assumption on $\delta$.
Fix a privacy mechanism $Q\in \Qc_\alpha$.
The conditionnal distribution of $Z_i$ given $Z_1,\ldots,Z_{i-1}$ when $X_i$ is distributed according to $f_0$ or $f_\nu$ will be denoted by $\Lc^{(0)}_{Z_i\mid z_{1:(i-1)}}(dz_i)=\int_\Rb Q_i(dz_i\mid x_i,z_{1:(i-1)})f_0(x_i)dx_i$ and $\Lc^{(\nu)}_{Z_i\mid z_{1:(i-1)}}(dz_i)=\int_\Rb Q_i(dz_i\mid x_i,z_{1:(i-1)})f_\nu(x_i)dx_i$ respectively.
The joint distribution of $Z_1,\ldots,Z_i$ when $X_1,\ldots,X_i$ are i.i.d. from $f_0$  will be denoted by
$$
\Lc^{(0)}_{Z_1,\ldots,Z_i}(dz_{1:i})=\Lc^{(0)}_{Z_i\mid z_{1:(i-1)}}(dz_i)\cdots\Lc^{(0)}_{Z_2\mid z_1}(dz_2)\Lc^{(0)}_{Z_1}(dz_1).
$$
The convexity and tensorization of the Kullback-Leibler divergence give
\begin{align*}
\text{KL}(Q_{f_0}^n,\bar{Q}^n)&\leq \frac{1}{2^N}\sum_{\nu\in\Vc}\text{KL}(Q_{f_0}^n,Q_{f_\nu}^n) \\
&= \frac{1}{2^N}\sum_{\nu\in\Vc}\sum_{i=1}^n \int_{\Zc^{i-1}}\text{KL}\left(\Lc^{(0)}_{Z_i\mid z_{1:(i-1)}} ,\Lc^{(\nu)}_{Z_i\mid z_{1:(i-1)}} \right)\Lc^{(0)}_{Z_1,\ldots,Z_{i-1}}(d z_{1:(i-1)})  .
\end{align*}
According to lemma B.3 in \cite{Butucea_Rohde_Steinberger_2020},there exists  a probability measure $\mu_{z_{1:(i-1)}}$ on $\Zc$ and a family of $\mu_{z_{1:(i-1)}}$-densities $z_i\mapsto q_i(\cdot \mid x_i,z_{1:(i-1)})$ of $Q_i(\cdot \mid x_i,z_{1:(i-1)})$, $x_i\in\Rb$  such that 
$$
e^{-\alpha}\leq q_i(z_i \mid x_i,z_{1:(i-1)})\leq e^{\alpha}, \quad \forall z_i\in\Zc, \forall x_i\in\Rb.
$$
We can thus write $\Lc^{(0)}_{Z_i\mid z_{1:(i-1)}}(dz_i)=m^{(0)}_i(z_i\mid z_{1:(i-1)}) d\mu_{z_{1:(i-1)}}(z_i)$, and $\Lc^{(\nu)}_{Z_i\mid z_{1:(i-1)}}(dz_i)=m^{(\nu)}_i(z_i\mid z_{1:(i-1)}) d\mu_{z_{1:(i-1)}}(z_i)$ with $m^{(0)}_i(z_i\mid z_{1:(i-1)})=\int_\Rb q_i(z_i\mid x_i,z_{1:(i-1)})f_0(x_i)\dd x_i$ and $m^{(\nu)}_i(z_i\mid z_{1:(i-1)})=\int_\Rb q_i(z_i\mid x_i,z_{1:(i-1)})f_\nu(x_i)\dd x_i$. 
Bounding the Kullback-Leibler divergence by the $\chi^2$-divergence, we have
\begin{align*}
&\text{KL}\left(\Lc^{(0)}_{Z_i\mid z_{1:(i-1)}} ,\Lc^{(\nu)}_{Z_i\mid z_{1:(i-1)}} \right)\\
&\leq  \int_{\Zc} \left( \frac{d\Lc^{(0)}_{Z_i\mid z_{1:(i-1)}}}{d\Lc^{(\nu)}_{Z_i\mid z_{1:(i-1)}}}-1  \right)^2\Lc^{(\nu)}_{Z_i\mid z_{1:(i-1)}}(dz_i) \\
&=\int_{\Zc}\left( \frac{m_{i}^{(0)}(z_i\mid z_{1:i-1})-m_{i}^{(\nu)}(z_i\mid z_{1:i-1})  }{m_{i}^{(\nu)}(z_i\mid z_{1:i-1})} \right)^2 m_{i}^{(\nu)}(z_i\mid z_{1:i-1})d\mu_{z_{1:(i-1)}}(z_i)   \\
&= \int_{\Zc} \left(\frac{\int_\Rb q_{i}(z_i\mid x, z_{1:i-1})\left(f_0(x)-f_\nu(x)  \right)  dx}{m_{i}^{(\nu)}(z_i\mid z_{1:i-1})}\right)^2  m_{i}^{(\nu)}(z_i\mid z_{1:i-1})d\mu_{z_{1:(i-1)}}(z_i)   \\
&=\int_{\Zc}\left[\int_{\Rb}\left(\frac{q_{i}(z_i\mid x, z_{1:i-1}) }{m_{i}^{(\nu)}(z_i\mid z_{1:i-1})}-e^{-2\alpha}\right) \left(f_0(x)-f_\nu(x)  \right)dx \right]^2 m_{i}^{(\nu)}(z_i\mid z_{1:i-1})d\mu_{z_{1:(i-1)}}(z_i), \\
\end{align*}
since $\int_\Rb(f_0-f_\nu)=0$. 
Recall that  $q_i$ satisfies $e^{-\alpha}\leq q_i(z_i \mid x,z_{1:(i-1)})\leq e^{\alpha}$.
Thus, we have $e^\alpha=\int e^{\alpha} f_\nu\geq m_i^{(\nu)}(z_i\mid z_{1:(i-1)})\geq e^{-\alpha}\int f_\nu=e^{-\alpha}$, and therefore 
$$
0\leq g_{i,z_{1:i}}(x):=\frac{q_i(z_i\mid x,z_{1:(i-1)} )}{m^{(\nu)}_i(z_i\mid z_{1:(i-1)})}-e^{-2\alpha}\leq z_\alpha=e^{2\alpha}-e^{-2\alpha}.
$$
Thus,
\begin{align*}
&\frac{1}{2^N}\sum_{\nu\in\Vc_N}\left[\int_{\Rb}\left(\frac{q_{i}(z_i\mid x, z_{1:i-1}) }{m_{i}^{(\nu)}(z_i\mid z_{1:i-1})}-e^{-2\alpha}\right) \left(f_0(x)-f_\nu(x)  \right)dx \right]^2 m_{i}^{(\nu)}(z_i\mid z_{1:i-1})\\
&\leq e^\alpha\delta^2\frac{1}{2^N}\sum_{\nu\in\Vc_N}\left[\sum_{k=1}^N\nu_k\int_{\Rb}g_{i,z_{1:i}}(x) \psi_k(x)dx \right]^2\\
&= e\delta^2\sum_{k=1}^N\left[\int_{\Rb}g_{i,z_{1:i}}(x) \psi_k(x)dx \right]^2\\
&\leq e\delta^2z_\alpha^2\sum_{k=1}^N \Vert \psi_k\Vert_1^2
\leq e\delta^2z_\alpha^2NhC_1^2 = \frac{e}{2}C_1^2\delta^2z_\alpha^2\vert B\vert,
\end{align*}
where we recall that $C_1=\int \vert \psi\vert$.
We thus obtain
$$
\text{KL}(Q_{f_0}^n,\bar{Q}^n)\leq \frac{e}{2}C_1^2\delta^2nz_\alpha^2\vert B\vert,
$$
and \eqref{Proof interactive LB Cond1} holds as soon as 
$$
\delta\leq\sqrt{\frac{4(1-\gamma)^2}{eC_1^2nz_\alpha^2\vert B\vert}}.
$$
Finally, taking $\delta=\min\left\{\sqrt{h}\cdot\min\left\{ \frac{C_0(B)}{\Vert\psi\Vert_\infty} , \frac{1}{2}\left(1-\frac{L_0}{L} \right)h^\beta\right\}, \sqrt{\frac{4(1-\gamma)^2}{eC_1^2nz_\alpha^2\vert B\vert}} \right\}$, we obtain
$$
\Ec_{n,\alpha}(f_0,\gamma)\geq C(\psi,\gamma)\min\left\{\vert B\vert\min\left\{ \frac{C_0(B)}{\Vert\psi\Vert_\infty} , \frac{1}{2}\left(1-\frac{L_0}{L} \right)h^\beta\right\}, \frac{\sqrt{\vert B\vert}}{\sqrt{h}\sqrt{n z_\alpha^2}}\right\}.
$$
If $B$ is chosen such that $C_0(B)=\min\{ f_0(x), x\in B\}\geq Ch^\beta$, then the bound becomes
$$
\Ec_{n,\alpha}(f_0,\gamma)\geq C(\psi,\gamma,L,L_0)\min\left\{\vert B\vert h^\beta, \frac{\sqrt{\vert B\vert}}{\sqrt{h}\sqrt{n z_\alpha^2}}\right\},
$$
and the choice $h\asymp \vert B \vert^{-1/(2\beta+1)}(nz_\alpha^2)^{-1/(2\beta+1)}$ yields
$$
\Ec_{n,\alpha}(f_0,\gamma)\geq C(\psi,\gamma,L,L_0)\vert B \vert^{\frac{\beta+1}{2\beta+1}}(nz_\alpha^2)^{-\frac{\beta}{2\beta+1}}.
$$
Note that with this choice of $h$, the condition $C_0(B)\geq Ch^\beta$ becomes $\vert B \vert^{\beta/(2\beta+1)}C_0(B) \geq C(nz_\alpha^2)^{-\beta/(2\beta+1)}$.

\section{Proofs of Section \ref{Section Examples}}\label{App Proof examples}

\subsection{Example \ref{Ex Pareto}}

We first prove the result for the non-interactive case.
Take 
$$
B=[a,T], \quad \text{with} \quad T=(n\alpha^2)^{\frac{2\beta}{k(4\beta+3)+3\beta+3}}.
$$
Note that $T>a$ for $n$ large enough.
Theorem \ref{Thrm Upper Bound NI} gives
\begin{align*}
\Ec_{n,\alpha}^\text{NI}(f_0,\gamma)&\lesssim(T-a)^{\frac{3\beta+3}{4\beta+3}}(n\alpha^2)^{-\frac{2\beta}{4\beta+3}} + \left(\frac{a}{T}\right)^k\\
&\lesssim  T^{\frac{3\beta+3}{4\beta+3}}(n\alpha^2)^{-\frac{2\beta}{4\beta+3}} + T^{-k} \\
&= (n\alpha^2)^{-\frac{2k\beta}{k(4\beta+3)+3\beta+3}}.
\end{align*}
To obtain the lower bound, we first check that condition \eqref{Eq. Condition on B for NI Lower bound} in Theorem \ref{Thrm Lower bound NI} is satisfied.
Since $T\rightarrow +\infty$ as $n\rightarrow \infty$, it holds for $n$ large enough
\begin{align*}
\vert B \vert^\frac{\beta}{4\beta+3}C_0(B) &=(T-a)^\frac{\beta}{4\beta+3}\frac{ka^k}{T^{k+1}} \\
&=ka^kT^{\frac{\beta}{4\beta+3}-(k+1)}\left(1-\frac{a}{T}\right)^\frac{\beta}{4\beta+3} \\
&\gtrsim T^\frac{\beta-(k+1)(4\beta+3)}{4\beta+3}\\
&\gtrsim C(n\alpha^2)^{-\frac{2\beta}{4\beta+3}}.
\end{align*}
Condition \eqref{Eq. Condition on B for NI Lower bound} is thus satisfied and Theorem \ref{Thrm Lower bound NI} thus yields for $n$ large enough
\begin{align*}
\Ec^{\text{NI}}_{n,\alpha}(f_0,\gamma)&\gtrsim \left[\log\left(C (T-a)^{\frac{4\beta+4}{4\beta+3}}(n\alpha^2)^\frac{2}{4\beta+3}  \right)\right]^{-1}(T-a)^{\frac{3\beta+3}{4\beta+3}}(n\alpha^2)^{-\frac{2\beta}{4\beta+3}}\\
&\gtrsim \left[\log\left(C T^{\frac{4\beta+4}{4\beta+3}}(n\alpha^2)^\frac{2}{4\beta+3}  \right)\right]^{-1}T^{\frac{3\beta+3}{4\beta+3}}(n\alpha^2)^{-\frac{2\beta}{4\beta+3}}\\
&\gtrsim  \left[\log\left(C (n\alpha^2)^{\frac{4\beta+4}{4\beta+3}\cdot \frac{2\beta}{k(4\beta+3)+3\beta+3}+\frac{2}{4\beta+3}}  \right)\right]^{-1}(n\alpha^2)^{-\frac{2k\beta}{k(4\beta+3)+3\beta+3}}.
\end{align*}
The proof in the interactive scenario follows the same lines at the exception of the choice of $T$ which should be taken as 
$$
T= (n\alpha^2)^{\frac{\beta}{k(2\beta+1)+\beta+1}}.
$$

\subsection{Example \ref{Ex Exponential}}

We first prove the result for the non-interactive case.
Take 
$$
B=[0,T], \quad \text{with} \quad T=\frac{1}{\lambda}\cdot\frac{2\beta}{4\beta+3}\log(n\alpha^2).
$$
Theorem \ref{Thrm Upper Bound NI} gives
\begin{align*}
\Ec_{n,\alpha}^\text{NI}(f_0,\gamma)&\lesssim T^{\frac{3\beta+3}{4\beta+3}}(n\alpha^2)^{-\frac{2\beta}{4\beta+3}} + \exp(-\lambda T)\\
&\lesssim  \log(n\alpha^2)^{\frac{3\beta+3}{4\beta+3}}(n\alpha^2)^{-\frac{2\beta}{4\beta+3}} + (n\alpha^2)^{-\frac{2\beta}{4\beta+3}} \\
&\lesssim \log(n\alpha^2)^{\frac{3\beta+3}{4\beta+3}}(n\alpha^2)^{-\frac{2\beta}{4\beta+3}}.
\end{align*}
Now, observe that 
$$ 
\vert B \vert^\frac{\beta}{4\beta+3}C_0(B) =T^\frac{\beta}{4\beta+3}\cdot \lambda\exp(-\lambda T) =\lambda T^\frac{\beta}{4\beta+3}(n\alpha^2)^{-\frac{2\beta}{4\beta+3}}\gtrsim (n\alpha^2)^{-\frac{2\beta}{4\beta+3}}.
$$
Thus, condition \eqref{Eq. Condition on B for NI Lower bound} is satisfied and Theorem \ref{Thrm Lower bound NI} yields 
\begin{align*}
\Ec^{\text{NI}}_{n,\alpha}(f_0,\gamma)&\gtrsim \left[\log\left(C T^{\frac{4\beta+4}{4\beta+3}}(n\alpha^2)^\frac{2}{4\beta+3}  \right)\right]^{-1}T^{\frac{3\beta+3}{4\beta+3}}(n\alpha^2)^{-\frac{2\beta}{4\beta+3}}\\
&\gtrsim \left[\log\left(C \log(n\alpha^2)^{\frac{4\beta+4}{4\beta+3}}(n\alpha^2)^\frac{2}{4\beta+3}  \right)\right]^{-1}\log(n\alpha^2)^{\frac{3\beta+3}{4\beta+3}}(n\alpha^2)^{-\frac{2\beta}{4\beta+3}}.
\end{align*}
The proof in the interactive scenario follows the same lines at the exception of the choice of $T$ which should be taken as 
$$
T=\frac{1}{\lambda}\cdot\frac{\beta}{2\beta+1}\log(n\alpha^2).
$$

\subsection{Example \ref{Ex Normal}}

We first prove the result for the non-interactive case.
Take 
$$
B=[-T,T], \quad \text{with} \quad T= \sqrt{\frac{4\beta}{4\beta+3}\log(n\alpha^2)}.
$$
Theorem \ref{Thrm Upper Bound NI} gives
\begin{align*}
\Ec_{n,\alpha}^\text{NI}(f_0,\gamma)&\lesssim (2T)^{\frac{3\beta+3}{4\beta+3}}(n\alpha^2)^{-\frac{2\beta}{4\beta+3}} + \frac{2}{\sqrt{2\pi}}\int_T^{+\infty}e^{-x^2/2}dx\\
&\lesssim  T^{\frac{3\beta+3}{4\beta+3}}(n\alpha^2)^{-\frac{2\beta}{4\beta+3}} + \frac{1}{T}\exp\left( -\frac{T^2}{2} \right) \\
&\lesssim  \log(n\alpha^2)^{\frac{3\beta+3}{2(4\beta+3)}}(n\alpha^2)^{-\frac{2\beta}{4\beta+3}} + (n\alpha^2)^{-\frac{2\beta}{4\beta+3}} \\
&\lesssim \log(n\alpha^2)^{\frac{3\beta+3}{2(4\beta+3)}}(n\alpha^2)^{-\frac{2\beta}{4\beta+3}}.
\end{align*}
Now, observe that 
$$ 
\vert B \vert^\frac{\beta}{4\beta+3}C_0(B) =(2T)^\frac{\beta}{4\beta+3}\cdot \frac{1}{\sqrt{2\pi}}\exp\left(-\frac{T^2}{2}  \right)  \gtrsim (n\alpha^2)^{-\frac{2\beta}{4\beta+3}}.
$$
Thus, condition \eqref{Eq. Condition on B for NI Lower bound} is satisfied and Theorem \ref{Thrm Lower bound NI} yields 
\begin{align*}
\Ec^{\text{NI}}_{n,\alpha}(f_0,\gamma)&\gtrsim \left[\log\left(C (2T)^{\frac{4\beta+4}{4\beta+3}}(n\alpha^2)^\frac{2}{4\beta+3}  \right)\right]^{-1}(2T)^{\frac{3\beta+3}{4\beta+3}}(n\alpha^2)^{-\frac{2\beta}{4\beta+3}}\\
&\gtrsim \left[\log\left(C \log(n\alpha^2)^{\frac{4\beta+4}{2(4\beta+3)}}(n\alpha^2)^\frac{2}{4\beta+3}  \right)\right]^{-1}\log(n\alpha^2)^{\frac{3\beta+3}{2(4\beta+3)}}(n\alpha^2)^{-\frac{2\beta}{4\beta+3}}
\end{align*}
The proof in the interactive scenario follows the same lines at the exception of the choice of $T$ which should be taken as 
$$
T= \sqrt{\frac{2\beta}{2\beta+1}\log(n\alpha^2)}.
$$

\subsection{Example \ref{Ex Cauchy}}

We first prove the result for the non-interactive case.
Take 
$$
B=[-T,T], \quad \text{with} \quad T= (n\alpha^2)^\frac{2\beta}{7\beta+6}.
$$
Theorem \ref{Thrm Upper Bound NI} gives
\begin{align*}
\Ec_{n,\alpha}^\text{NI}(f_0,\gamma)&\lesssim (2T)^{\frac{3\beta+3}{4\beta+3}}(n\alpha^2)^{-\frac{2\beta}{4\beta+3}} + \frac{2}{\pi a}\int_{T}^{+\infty}\frac{a^2}{a^2+x^2}dx\\
&\lesssim  T^{\frac{3\beta+3}{4\beta+3}}(n\alpha^2)^{-\frac{2\beta}{4\beta+3}} + \arctan\left( \frac{a}{T}\right) .
\end{align*}
Since $T\rightarrow \infty$ as $n\rightarrow \infty$, we have $\arctan(a/T)\sim_{n\rightarrow \infty}  a/T$ and thus $\arctan(a/T)\leq 2(a/T)$ for $n$ large enough.
This gives for $n$ large enough
$$
\Ec_{n,\alpha}^\text{NI}(f_0,\gamma)\lesssim   T^{\frac{3\beta+3}{4\beta+3}}(n\alpha^2)^{-\frac{2\beta}{4\beta+3}} + \frac{1}{T}= (n\alpha^2)^{-\frac{2\beta}{7\beta+6}}
$$
Now, observe that for $n$ large enough it holds
$$ 
\vert B \vert^\frac{\beta}{4\beta+3}C_0(B) =(2T)^\frac{\beta}{4\beta+3}\cdot \frac{1}{\pi a}\frac{a^2}{T^2+a^2} \gtrsim T^\frac{\beta}{4\beta+3}\cdot \frac{1}{T^2} = (n\alpha^2)^{-\frac{2\beta}{4\beta+3}}.
$$
Thus, condition \eqref{Eq. Condition on B for NI Lower bound} is satisfied and Theorem \ref{Thrm Lower bound NI} yields 
\begin{align*}
\Ec^{\text{NI}}_{n,\alpha}(f_0,\gamma)&\gtrsim \left[\log\left(C (2T)^{\frac{4\beta+4}{4\beta+3}}(n\alpha^2)^\frac{2}{4\beta+3}  \right)\right]^{-1}(2T)^{\frac{3\beta+3}{4\beta+3}}(n\alpha^2)^{-\frac{2\beta}{4\beta+3}}\\
&\gtrsim \left[\log\left(C (n\alpha^2)^{\frac{4\beta+4}{4\beta+3}\cdot\frac{2\beta}{7\beta+6} + \frac{2}{4\beta+3} } \right)\right]^{-1}(n\alpha^2)^{-\frac{2\beta}{7\beta+6}}.
\end{align*}
The proof in the interactive scenario follows the same lines at the exception of the choice of $T$ which should be taken as 
$$
T=(n\alpha^2)^\frac{\beta}{3\beta+2}.
$$

\subsection{Example \ref{Ex Spiky null}}

We first prove the result for the non-interactive case.
The upper bound is straightforward taking $B=[0, 2/\sqrt{L_0}]$.
For the lower bound, take 
$$
B=\left[ T , \frac{2}{\sqrt{L_0}}-T\right ], \quad \text{with} \quad T= (n\alpha^2)^{-\frac{2\beta}{4\beta+3}}.
$$
Note that for $n$ large enough it holds $T<1/(2\sqrt{L_0})$ and we thus have
$$ 
\vert B \vert^\frac{\beta}{4\beta+3}C_0(B) =\left( \frac{2}{\sqrt{L_0}}-2T \right)^\frac{\beta}{4\beta+3}\cdot L_0T \gtrsim T= (n\alpha^2)^{-\frac{2\beta}{4\beta+3}}.
$$
Thus, condition \eqref{Eq. Condition on B for NI Lower bound} is satisfied and Theorem \ref{Thrm Lower bound NI} yields 
\begin{align*}
\Ec^{\text{NI}}_{n,\alpha}(f_0,\gamma)&\gtrsim \left[\log\left(C \left( \frac{2}{\sqrt{L_0}}-2T \right)^{\frac{4\beta+4}{4\beta+3}}(n\alpha^2)^\frac{2}{4\beta+3}  \right)\right]^{-1}\left( \frac{2}{\sqrt{L_0}}-2T \right)^{\frac{3\beta+3}{4\beta+3}}(n\alpha^2)^{-\frac{2\beta}{4\beta+3}}\\
&\gtrsim \left[\log\left(C (n\alpha^2)^\frac{2}{4\beta+3}  \right)\right]^{-1}(n\alpha^2)^{-\frac{2\beta}{4\beta+3}}
\end{align*}
The proof in the interactive scenario follows the same lines at the exception of the choice of $T$ for the lower bound which should be taken as 
$$
T=(n\alpha^2)^{-\frac{\beta}{2\beta+1}}.
$$

\subsection{Example \ref{Ex Beta}}

Let $a\geq 1$, $b\geq 1$ with $a>1$ or $b>1$.
We first prove the result for the non-interactive case.
The upper bound is straightforward taking $B=[0, 1]$.
For the lower bound, we need to distinguish different cases.

\underline{\textit{Case 1 : $a > 1, b=1$.}} In this case  $f_0$ is strictly non-decreasing on $[0,1]$ and $f_0(0)=0$. In order that $f_0$ is bounded from below by a strictly positive quantity, we thus take $B$ of the form $B=[T_1,1]$ with $0<T_1<1$. 
We choose 
$$
T_1= (n\alpha^2)^{-\frac{2\beta}{(a-1)(4\beta+3)}}.
$$
Observe that that for $n$ large enough we have
$$
\vert B \vert^\frac{\beta}{4\beta+3}C_0(B) =\left[1-T_1  \right]^\frac{\beta}{4\beta+3}\cdot \frac{1}{B(a,1)} T_1^{a-1} \gtrsim T_1^{a-1}=(n\alpha^2)^{-\frac{2\beta}{4\beta+3}}
$$
Thus, condition \eqref{Eq. Condition on B for NI Lower bound} is satisfied and Theorem \ref{Thrm Lower bound NI} yields for $n$ large enough
\begin{align*}
\Ec^{\text{NI}}_{n,\alpha}(f_0,\gamma)&\gtrsim \left[\log\left(C \left[1-T_1  \right]^{\frac{4\beta+4}{4\beta+3}}(n\alpha^2)^\frac{2}{4\beta+3}  \right)\right]^{-1}\left[1-T_1  \right]^{\frac{3\beta+3}{4\beta+3}}(n\alpha^2)^{-\frac{2\beta}{4\beta+3}}\\
&\gtrsim \left[\log\left(C(n\alpha^2)^\frac{2}{4\beta+3}  \right)\right]^{-1}(n\alpha^2)^{-\frac{2\beta}{4\beta+3}}.
\end{align*}

\underline{\textit{Case 2 : $a= 1, b> 1$.}} In this case $f_0$ is strictly non-increasing on $[0,1]$ and $f_0(1)=0$. In order that $f_0$ is bounded from below by a strictly positive quantity, we thus take $B$ of the form $B=[0,1-T_2]$ with $0<T_2<1$. 
We choose 
$$
T_2= (n\alpha^2)^{-\frac{2\beta}{(b-1)(4\beta+3)}}.
$$
Observe that that for $n$ large enough we have
$$
\vert B \vert^\frac{\beta}{4\beta+3}C_0(B) =\left[1-T_2  \right]^\frac{\beta}{4\beta+3}\cdot \frac{1}{B(1,b)} T_2^{b-1} \gtrsim T_2^{b-1}=(n\alpha^2)^{-\frac{2\beta}{4\beta+3}}
$$
Thus, condition \eqref{Eq. Condition on B for NI Lower bound} is satisfied and Theorem \ref{Thrm Lower bound NI} yields for $n$ large enough
\begin{align*}
\Ec^{\text{NI}}_{n,\alpha}(f_0,\gamma)&\gtrsim \left[\log\left(C \left[1-T_2  \right]^{\frac{4\beta+4}{4\beta+3}}(n\alpha^2)^\frac{2}{4\beta+3}  \right)\right]^{-1}\left[1-T_2  \right]^{\frac{3\beta+3}{4\beta+3}}(n\alpha^2)^{-\frac{2\beta}{4\beta+3}}\\
&\gtrsim \left[\log\left(C(n\alpha^2)^\frac{2}{4\beta+3}  \right)\right]^{-1}(n\alpha^2)^{-\frac{2\beta}{4\beta+3}}.
\end{align*}

\underline{\textit{Case 3 : $a> 1, b> 1$.}}
In this case,  $f_0$ is non-decreasing on $[0,(a-1)/(a+b-2)]$, non-increasing on $[(a-1)/(a+b-2),1]$ and $f_0(0)=f_0(1)=0$. In order that $f_0$ is bounded from below by a strictly positive quantity, we thus take $B$ of the form $B=[T_3,1-T_4]$ and we choose 
$$
T_3= (n\alpha^2)^{-\frac{2\beta}{(a-1)(4\beta+3)}}, \quad T_4= (n\alpha^2)^{-\frac{2\beta}{(b-1)(4\beta+3)}}.
$$
Observe that for $n$ large enough it holds
$$
0<T_3<\frac{a-1}{a+b-2}<1-T_4<1.
$$
Observe that for $n$ large enough we have
\begin{align*}
\vert B \vert^\frac{\beta}{4\beta+3}C_0(B) &=\left[1-(T_3+T_4)  \right]^\frac{\beta}{4\beta+3}\cdot \frac{1}{B(a,b)}\min\left\{ T_3^{a-1}(1-T_3)^{b-1}, (1-T_4)^{a-1}T_4^{b-1} \right\}\\
&\gtrsim \min\left\{T_3^{a-1},T_4^{b-1}  \right\}\\
&\gtrsim (n\alpha^2)^{-\frac{2\beta}{4\beta+3}}.
\end{align*}
Thus, condition \eqref{Eq. Condition on B for NI Lower bound} is satisfied and Theorem \ref{Thrm Lower bound NI} yields for $n$ large enough
\begin{align*}
\Ec^{\text{NI}}_{n,\alpha}(f_0,\gamma)&\gtrsim \left[\log\left(C \left[1-(T_3+T_4)  \right]^{\frac{4\beta+4}{4\beta+3}}(n\alpha^2)^\frac{2}{4\beta+3}  \right)\right]^{-1}\left[1-(T_3+T_4)  \right]^{\frac{3\beta+3}{4\beta+3}}(n\alpha^2)^{-\frac{2\beta}{4\beta+3}}\\
&\gtrsim \left[\log\left(C(n\alpha^2)^\frac{2}{4\beta+3}  \right)\right]^{-1}(n\alpha^2)^{-\frac{2\beta}{4\beta+3}}
\end{align*}
The proof in the interactive scenario follows the same lines at the exception of the choice of $T_1$ and $T_2$ which should be taken as 
$$
T_1= T_3=(n\alpha^2)^{-\frac{\beta}{(a-1)(2\beta+1)}}, \quad T_2=T_4= (n\alpha^2)^{-\frac{\beta}{(b-1)(2\beta+1)}}.
$$

\subsection{Example \ref{Ex Slow varying function}}

We prove the result for the non interactive case. Take
$$B=B_{n,\alpha}\in \arg \inf_{B \text{ compact set}}
\left\{ \int_{\overline{B}}{ f_0} \geq \vert B \vert^{\frac{3\beta+3}{4\beta+3}}(n\alpha^2)^{-\frac{2\beta}{4\beta+3}} + \frac 1{\sqrt{n\alpha^2}}  \text{ and } \inf_B f_0 \geq \sup_{\overline{B}} f_0  \right\},
$$
It holds $B=B_{n,\alpha}=[0,a_*]$ with 
$$a_*=\sup\left\{ a : \frac{(\log 2)^{A}}{(\log(2+a))^{A}} \geq a^{\frac{3\beta+3}{4\beta +3}}(n\alpha^2)^{-\frac{2\beta}{4\beta +3}}+\frac{1}{\sqrt{n\alpha^2}}\right\},$$
and Theorem \ref{Thrm Upper Bound NI} thus gives
$$\Ec^{\text{NI}}_{n,\alpha}(f_0,\gamma)\lesssim a_*^{\frac{3\beta+3}{4\beta +3}}(n\alpha^2)^{-\frac{2\beta}{4\beta +3}}+\frac{1}{\sqrt{n\alpha^2}}\lesssim a_*^{\frac{3\beta+3}{4\beta +3}}(n\alpha^2)^{-\frac{2\beta}{4\beta +3}},$$
where the last inequality follows from $a_*\geq 1\geq (n\alpha^2)^{-{\frac{1}{2\beta+2}}}.$\\
Inspecting the proof of Theorem \ref{Thrm Lower bound NI}, we see that the lower bound can be rewritten 
$$ \Ec^{\text{NI}}_{n,\alpha}(f_0,\gamma)\gtrsim \left[\log\left(C \vert B\vert^{\frac{4\beta+4}{4\beta+3}}(n \alpha^2)^\frac{2}{4\beta+3}  \right)\right]^{-1} \min\left\{\vert B \vert C_0(B); \vert B \vert^{\frac{3\beta+3}{4\beta+3}}(n\alpha^2)^{-\frac{2\beta}{4\beta +3}} \right\}. $$
Yet, for $B=B_{n,\alpha}=[0,a_*]$ it holds
$$\vert B_{n,\alpha} \vert C_0(B_{n,\alpha})=\frac{A(\log 2)^Aa_*}{(a_*+2)(\log(2+a_*))^{A+1}}\gtrsim \frac{1}{\log(a_*)}\times a_*^{\frac{3\beta+3}{4\beta +3}}(n\alpha^2)^{-\frac{2\beta}{4\beta +3}},  $$
yielding 
$$\Ec_{n,\alpha}^\text{NI}(f_0,\gamma)\gtrsim \left[\log\left(C a_*^{\frac{4\beta+4}{4\beta+3}}(n\alpha^2)^\frac{2}{4\beta+3}  \right)\right]^{-1}\left[\log(a_*) \right]^{-1}a_*^{\frac{3\beta+3}{4\beta+3}}(n\alpha^2)^{-\frac{2\beta}{4\beta+3}}.$$

\printbibliography

\end{document}